\documentclass[a4paper,10pt]{article}

\usepackage[utf8]{inputenc} 
\usepackage{color}

\makeatletter
\newcommand{\LoadPackagesNow}{}
\newcommand{\LoadPackageLater}[1]{%
   \g@addto@macro{\LoadPackagesNow}{%
      \usepackage{#1}%
   }%
}
\makeatother

\usepackage{xspace}
\usepackage{xargs}
\usepackage{ifthen}
\usepackage{etoolbox}

\usepackage[
	table 
]{xcolor}
\usepackage[%
]{graphicx}

\usepackage{wrapfig}

\usepackage{subfigure}
\usepackage{caption}
\usepackage[
   tbtags,    
   sumlimits,  
   nointlimits, 
   namelimits, 
   reqno,     
]{amsmath} %

\usepackage{amsfonts}
\usepackage{mathrsfs} 
\usepackage{dsfont}
\usepackage{amssymb}
\LoadPackageLater{amsthm}
\LoadPackageLater{thmtools}
\usepackage[fixamsmath,disallowspaces]{mathtools}
\mathtoolsset{showonlyrefs}
\mathtoolsset{centercolon=true}

\usepackage{geometry}
\usepackage[%
	square,	
	comma,	
	numbers,	
	sort,		
	sort&compress,    
]{natbib}
\usepackage{framed}
\usepackage[normalem]{ulem}      
\usepackage{soul}		            
\usepackage{url} 
\usepackage{lipsum}
\usepackage[textsize=tiny,english,colorinlistoftodos,disable]{todonotes}

\usepackage{enumitem}
\definecolor{pdfurlcolor}{rgb}{0,0,0.6}
\definecolor{pdffilecolor}{rgb}{0.7,0,0}
\definecolor{pdflinkcolor}{rgb}{0,0,0.6}
\definecolor{pdfcitecolor}{rgb}{0,0,0.6}
\usepackage[
  colorlinks=true,         
  urlcolor=pdfurlcolor,    
  filecolor=pdffilecolor,  
  linkcolor=pdflinkcolor,  
  citecolor=pdfcitecolor,  %
  raiselinks=true,			 
  breaklinks,              
  verbose,
  hyperindex=true,         
  linktocpage=true,        
  hyperfootnotes=false,     
  bookmarks=true,          
  bookmarksopenlevel=1,    
  bookmarksopen=true,      
  bookmarksnumbered=true,  
  bookmarkstype=toc,       
  plainpages=false,        
  pageanchor=true,         
  pdftitle={Asymptotic Analysis of Inpainting via Universal Shearlet Systems},             
  pdfauthor={Martin Genzel, Gitta Kutyniok},            
  pdfcreator={LaTeX, hyperref, KOMA-Script}, 
  pdfdisplaydoctitle=true, 
  pdfstartview=FitH,       
  pdfpagemode=UseOutlines, 
  pdfpagelabels=true,           
  pdfpagelayout=SinglePage, 
]{hyperref}

\LoadPackagesNow

\usepackage{bm}

\makeatletter
\g@addto@macro\bfseries{\boldmath}
\makeatother

\newcommand{\ifargdef}[3][{}]{\ifthenelse{\equal{#2}{}}{#1}{#3}}

\newenvironment{properties}[1][{}]
{\begin{enumerate}[label={\textsc{(#1\arabic*)}},itemindent=2em]}
{\end{enumerate}} 
\newenvironment{proofsteps}[1]
{\begin{enumerate}[label={\textsc{(#1\arabic*)}},itemindent=4em,leftmargin=0em]}
{\end{enumerate}} 
\newenvironment{romanlist}
{\begin{enumerate}[label=\upshape(\roman*),itemindent=2em,leftmargin=0pt]}
{\end{enumerate}}

\newtheoremstyle{claim}
	{\topsep}{\topsep}%
	{\itshape}
	{}
	{\bfseries\boldmath}
	{}
	{.5em}
	{\thmname{#1} \thmnumber{#2} \thmnote{(#3)}}
\newtheoremstyle{definition}
	{\topsep}{\topsep}%
	{}
	{}
	{}
	{}
	{.5em}
	{{\bfseries\thmname{#1} \thmnumber{#2}} \thmnote{(#3)}}
\newtheoremstyle{remark}
	{\topsep}{\topsep}%
	{}
	{}
	{}
	{}
	{.5em}
	{{\itshape\thmname{#1}:}}

\declaretheorem[style=claim,numberwithin=section]{theorem}
\declaretheorem[style=claim,sibling=theorem]{proposition}
\declaretheorem[style=claim,sibling=theorem]{lemma}
\declaretheorem[style=claim,sibling=theorem]{corollary}

\declaretheorem[style=definition,sibling=theorem]{definition}
\declaretheorem[style=definition,sibling=theorem]{algorithm}

\newcommand{\opleft}[1]{\mathopen{}\left#1}
\newcommand{\opright}[1]{\right#1\mathclose{}}
\newcommandx{\braces}[4]{%
\ifstrequal{#3}{normal}{#1#4#2}{%
\ifstrequal{#3}{auto}{\left#1#4\right#2}{%
\ifstrequal{#3}{opauto}{\opleft#1#4\opright#2}{%
#3#1#4#3#2}}}%
}
\newcommandx{\opannot}[3][3=\downarrow]{\stackrel{\mathclap{\substack{#1 \\ #3 \vspace{2pt}}}}{#2}}
\newcommandx{\lineannot}[3][3=\rightarrow]{\mathllap{\boxed{\text{\textsmaller{#1}}} #3} #2}
\newcommandx{\multilineannot}[4][4=\rightarrow]{\mathllap{\boxed{\parbox{#1}{\RaggedRight\textsmaller{#2}}} #4} #3}
 
\newcommand{\N}{\mathbb{N}} 
\newcommand{\Nzero}{\mathbb{N}_0} 
\newcommand{\Z}{\mathbb{Z}} 
\newcommand{\R}{\mathbb{R}} 
\newcommand{\eps}{\varepsilon} 

\renewcommand{\iff}{\Leftrightarrow} 
\renewcommand{\implies}{\Rightarrow} 
\newcommand{\suchthat}[1][normal]{\ifstrequal{#1}{normal}{\mid}{#1|}} 

\newcommand{\setcompl}[1]{#1^c} 
\newcommand{\cardinality}{\#} 
\newcommand{\union}{\cup} 
\newcommand{\bigunion}{\bigcup} 
\newcommand{\boundary}[1]{\partial#1} 
\newcommandx{\intvcl}[3][1=normal]{\braces{[}{]}{#1}{#2, #3}} 
\newcommandx{\intvop}[3][1=normal]{\braces{(}{)}{#1}{#2, #3}} 
\newcommandx{\intvclop}[3][1=normal]{\braces{[}{)}{#1}{#2, #3}} 
\newcommandx{\intvopcl}[3][1=normal]{\braces{(}{]}{#1}{#2, #3}} 

\newcommand{\dotarg}{\ensuremath{\raisebox{.15ex}{$\scriptstyle [\cdot]$}}} 
\DeclareMathOperator*{\argmin}{argmin} 
\newcommandx{\abs}[2][1=normal]{\braces{\lvert}{\rvert}{#1}{#2}} 
\newcommand{\conj}[1]{\overline{#1}} 
\newcommandx{\ceil}[2][1=normal]{\braces{\lceil}{\rceil}{#1}{#2}} 
\newcommandx{\floor}[2][1=normal]{\braces{\lfloor}{\rfloor}{#1}{#2}} 
\newcommandx{\round}[2][1=normal]{\braces{[}{]}{#1}{#2}} 
\newcommandx{\der}[1]{D^{#1}} 
\newcommandx{\partder}[4][1={},4={}]{\frac{\partial^{#4} #2}{\partial #3^{#4}}\ifargdef{#1}{\Big|_{#1}}} 
\newcommandx{\integ}[4][1={},2={}]{\int_{#1}^{#2} #3 \, #4} 
\newcommandx{\asympffaster}[2][1=normal]{o\braces{(}{)}{#1}{#2}} 
\newcommandx{\asympfaster}[2][1=normal]{O\braces{(}{)}{#1}{#2}} 
\newcommandx{\asympeq}[2][1=normal]{\Theta\braces{(}{)}{#1}{#2}} 
\newcommandx{\asympsslower}[2][1=normal]{\omega\braces{(}{)}{#1}{#2}} 
\newcommandx{\asympslower}[2][1=normal]{\Omega\braces{(}{)}{#1}{#2}} 

\newcommand{\matr}[1]{\begin{bmatrix} #1 \end{bmatrix}} 
\newcommandx{\norm}[2][1=normal]{\braces{\|}{\|}{#1}{#2}} 
\renewcommandx{\sp}[3][1=normal]{\braces{\langle}{\rangle}{#1}{#2, #3}} 
\newcommand{\adj}[1]{{#1}^\ast} 
\newcommandx{\End}[2][2={}]{\mathcal{L}\opleft( #1 \ifargdef{#2}{, #2} \opright)} 
\newcommand{\orthsum}{\oplus} 
\newcommand{\T}{\top} 
\newcommand{\embeds}{\hookrightarrow} 

\newcommandx{\measure}[2][1=normal]{\operatorname{vol}\braces{(}{)}{#1}{#2}} 
\newcommand{\indset}[1]{\chi_{#1}} 
\newcommand{\indcoeff}[1]{\mathds{1}_{#1}} 
\DeclareMathOperator{\supp}{supp} 
\newcommandx{\Leb}[3][1={},3=normal]{L^{#2}\ifargdef{#1}{\braces{(}{)}{#3}{#1}}{}} 
\newcommandx{\Lebnorm}[4][1=normal,3={2},4={}]{\norm[#1]{#2}_{\Leb[#4]{#3}}} 
\renewcommandx{\l}[3][1={},3=normal]{\ell^{#2}\ifargdef{#1}{\braces{(}{)}{#3}{#1}}} 
\newcommandx{\lnorm}[4][1=normal,3={2},4={}]{\norm[#1]{#2}_{\l[#4]{#3}}} 
\newcommandx{\Smooth}[4][1={},3={},4=normal]{C_{#3}^{#2}\ifargdef{#1}{\braces{(}{)}{#4}{#1}}} 
\newcommandx{\Schwartz}[2][1={},2=normal]{\mathscr{S}\ifargdef{#1}{\braces{(}{)}{#2}{#1}}} 
\newcommandx{\Schwartzpoly}[2][1=normal]{\braces{\langle}{\rangle}{#1}{\abs[#1]{#2}} } 
\newcommandx{\Tempdistr}[2][1={},2=normal]{\mathscr{S}'\ifargdef{#1}{\braces{(}{)}{#2}{#1}}} 
\newcommandx{\distrinp}[3][1=normal]{\braces{\langle}{\rangle}{#1}{#2, #3}} 
\newcommand{\Linedistr}[1][]{\mathfrak{L}\ifargdef{#1}{_{#1}}{}} 
\newcommand{\conv}{\ast} 
\newcommandx{\ft}[3][1=default,2=auto]{
\ifstrequal{#1}{default}{\widehat{#3}}{
\ifstrequal{#1}{long}{{\braces{(}{)}{#2}{#3}}^{\wedge}}{}}} 
\newcommand{\ftop}{\mathcal{F}} 
\newcommandx{\ift}[3][1=default,2=auto]{
\ifstrequal{#1}{default}{\check{#3}}{
\ifstrequal{#1}{long}{{\braces{(}{)}{#2}{#3}}^{\vee}}{}}} 

\graphicspath{{images/}} 

\title{\bfseries Asymptotic Analysis of Inpainting via Universal Shearlet Systems}
\author{\hspace*{-1cm}
  Martin Genzel and Gitta Kutyniok\\[.5em]
  \textsc{\hspace*{-1cm}Department of Mathematics, Technische Universit\"at Berlin}\\
  10623 Berlin, Germany \\[.5em]
  \hspace*{-1cm}E-mails: \href{mailto:genzel@math.tu-berlin.de}{genzel@math.tu-berlin.de}, \href{mailto:kutyniok@math.tu-berlin.de}{kutyniok@math.tu-berlin.de}
}

\begin{document}
\listoftodos
\newcommand{\OpAnalysis}[1]{T_{#1}} 
\newcommand{\OpSynthesis}[1]{\adj T_{#1}} 
\newcommand{\OpFrame}[1]{S_{#1}} 
\newcommand{\defsf}{\varphi} 
\newcommand{\InpSp}{\mathcal{H}} 
\newcommand{\InpSpK}{\InpSp_K} 
\newcommand{\ProK}{P_K} 
\newcommand{\InpSpM}{\InpSp_M} 
\newcommand{\ProM}{P_M} 
\newcommand{\sig}{x^0} 
\newcommand{\sigrec}{x^\star} 
\newcommand{\PF}{\Phi} 
\newcommand{\pf}{\phi} 
\newcommand{\cluster}{\Lambda} 
\newcommand{\concentr}[2]{\kappa\ifargdef{#1}{\opleft( #1, #2 \opright)}} 
\newcommand{\clustercoh}[2]{\mu_c \ifargdef{#1}{( #1 , #2)}} 
\newcommand{\anorm}[2]{\norm{#1}_{1,#2}} 
\newcommand{\ver}{\mathrm{v}} 
\newcommand{\hor}{\mathrm{h}} 
\newcommand{\dir}{\imath} 
\newcommand{\meyerscal}{\phi} 
\newcommand{\Scalfunc}{\Phi} 
\newcommand{\Corofunc}{W} 
\newcommand{\Coro}{\mathscr{K}} 
\newcommand{\conefunc}{v} 
\newcommand{\Conefunc}[1]{V_{(#1)}} 
\newcommand{\Cone}[1]{\mathscr{C}_{(#1)}} 
\newcommand{\pscal}{A} 
\newcommand{\pshear}{S} 
\newcommand{\pscalcone}[2]{A_{#1,(#2)}} 
\newcommand{\shearcone}[1]{S_{(#1)}} 
\newcommand{\unishplain}{\psi} 
\newcommand{\unishplainft}{\ft{\unishplain}} 
\newcommand{\unish}[5][{}]{\unishplain_{#3,#4,#5}^{#2\ifargdef{#1}{,(#1)}}} 
\newcommand{\unishft}[5][{}]{\unishplainft_{#3,#4,#5}^{#2\ifargdef{#1}{,(#1)}}} 
\newcommand{\Scalparamdomain}{\mathsf{A}} 
\newcommand{\aj}{\alpha_j} 
\newcommandx{\Unish}[3][1=\meyerscal,2=\conefunc,3=(\aj)_j]{\operatorname{SH}(#1, #2, #3)} 
\newcommandx{\UnishLow}[1][1=\meyerscal]{\operatorname{SH}_{\mathrm{Low}}(#1)} 
\newcommandx{\UnishInt}[3][1=\meyerscal,2=\conefunc,3=(\aj)_j]{\operatorname{SH}_{\mathrm{Int}}(#1, #2, #3)} 
\newcommandx{\UnishBound}[3][1=\meyerscal,2=\conefunc,3=(\aj)_j]{\operatorname{SH}_{\mathrm{Bound}}(#1, #2, #3)} 
\newcommand{\Unishgroup}{\Gamma} 
\newcommand{\Unishind}{\gamma} 
\newcommand{\lmax}{l_j} 
\newcommand{\weight}{w} 
\newcommand{\wlen}{\rho} 
\newcommand{\model}{{\weight\Linedistr}} 
\newcommand{\modelrec}{\model^\star} 
\newcommand{\Corofilter}{F} 
\newcommand{\mdiam}{h} 
\newcommand{\mask}[1]{\mathscr{M}_{#1}} 
\newcommand{\Unishshort}{\Psi} 
\newcommand{\scalpm}[1]{#1^{\pm1}} 
\newcommand{\translind}[2]{#1^{(#2)}} 
\maketitle
\begin{abstract}
Recently introduced inpainting algorithms using a combination of applied harmonic analysis and compressed sensing have turned out to be very successful.
One key ingredient is a carefully chosen representation system which provides (optimally) sparse approximations of the original image. 
Due to the common assumption that images are typically governed by anisotropic features, directional representation systems have often been utilized. 
One prominent example of this class are \emph{shearlets}, which have the additional benefit to allowing faithful implementations. 
Numerical results show that shearlets significantly outperform wavelets in inpainting tasks. One of those software packages, \url{www.shearlab.org}, even offers the flexibility of using a different parameter for each scale, which is not yet covered by shearlet theory.

In this paper, we first introduce \emph{universal shearlet systems} which are associated with an arbitrary scaling sequence, thereby modeling the previously mentioned flexibility.
In addition, this novel construction allows for a smooth transition between wavelets and shearlets and therefore enables us to analyze them in a uniform fashion.
For a large class of such scaling sequences, we first prove that the associated universal shearlet systems form band-limited Parseval frames for $\Leb[\R^2]{2}$ consisting of Schwartz functions.
Secondly, we analyze the performance for inpainting of this class of universal shearlet systems within a distributional model situation using an $\l{1}$-analysis minimization algorithm for reconstruction.
Our main result in this part states that, provided the scaling sequence is comparable to the size of the (scale-dependent) gap, nearly-perfect inpainting is achieved at sufficiently fine scales.
\end{abstract}

\vspace{.1in}


\vspace{.1in}


\vspace{.1in}
\textbf{Acknowledgements.}
The second author acknowledges support by the Einstein Foundation Berlin, by the Einstein Center for Mathematics
Berlin (ECMath), by Deutsche Forschungsgemeinschaft (DFG) Grant KU 1446/14, by the DFG Collaborative
Research Center TRR 109 ``Discretization in Geometry and Dynamics'', and by the DFG Research Center
\textsc{Matheon} ``Mathematics for key technologies'' in Berlin.

\vspace{.1in}

\section{Introduction}

In data processing, due to frequent failures of technologies, one of the most natural questions to ask is whether and how lost data can be recovered.
In the context of imaging sciences, this is termed \emph{inpainting} due to the according physical task when restoring a painting.
Given a digital image which is corrupted, it is certainly impossible to reconstruct the missing content without any prior knowledge.
For instance, one may think of an image of the sky where several birds can be seen, but with one bird now being deleted.
Hence, each inpainting method has to make certain further assumptions on the given image signal.

Inpainting algorithms based on variational approaches propagate information from the boundaries and intend to guarantee smoothness; see, e.g., \cite{BalBerCas01, bertalmio2001navier, ChaShe01}.
Those which are based on applied harmonic analysis---typically combined with ideas of compressed sensing---encode prior information by assuming that a carefully selected representation system provides sparse approximations of the original image; and we refer to the pioneering paper \cite{ElaStaQue05} as well as \cite{CaiChaShe10, DonJiLi12, herrmann2008non}.
It should also be mentioned that most of these approaches assume that the position and shape of the area that requires inpainting is known.
If this is not the case, ``blind inpainting'' has to be applied, which is, as one can easily imagine, an even much more difficult task; first approaches can be found
in \cite{BobStaFad07, DonJiLi12}.

In this paper, we will focus on inpainting via a combination of applied harmonic analysis and compressed sensing in the sense of  $\l{1}$-analysis minimization and assuming that the missing area is known.
In this setting, the directional representation system of shearlets has recently been shown to outperform not only wavelets, but also most other directional systems \cite{kutyniok2014shearlab}.
This raises the challenge to develop a theoretical framework which is particularly able to relate the structural difference between wavelets and shearlets to the inpainting problem.
A first study was performed in \cite{king2014analysis} which however could only show superiority of shearlets over wavelets for a basic thresholding algorithm.

In the sequel, we will pursue a different path and analyze a parametrized family of systems ranging from wavelets to shearlets, aiming to derive a deep understanding of the transition from isotropy to anisotropy for the considered inpainting approach.
Along the way, with the introduction and analysis of \emph{universal shearlet systems} as sparsifying systems, we will also provide a general framework---including wavelets and shearlets---which has the additional freedom to choose a different degree of anisotropy for each scale.
Moreover, this approach has several impacts which also go beyond inpainting tasks (cf. Subsection~\ref{subsec:intro:impact}).

\subsection{Inpainting via Applied Harmonic Analysis and Compressed Sensing}
\label{subsec:intro:inp}

Results of \emph{compressed sensing} have shown that, assumed a signal is sparse within a frame, it can be recovered from highly underdetermined, non-adaptive linear measurements by $\l{1}$-minimization (cf. \cite{DDEK12}).
For this, recall that a \emph{frame} for a Hilbert space $\InpSp$ is a sequence $\PF=(\pf_i)_{i \in I}$ satisfying $A\norm{x}^2 \leq
\lnorm{(\sp{x}{\pf_i})_{i \in I}}[2][I] \leq B \norm{x}^2$ for all $x \in \InpSp$. If the frame bounds $A, B$ can chosen to be equal to $1$, it is typically called \emph{Parseval frame}.

This approach can be also utilized for solving the inverse problem of inpainting. Letting $\sig$ denote the original signal, we assume that $\sig$ is sparse within some frame $\PF$.
Due to the redundancy of frames, there exist infinitely many expansions of $\sig$ in $\PF$. Therefore, it is typically assumed that the \emph{analysis operator}
\begin{equation}
\OpAnalysis{\PF} \colon \InpSp \to \l[I]{2}, \ x \mapsto (\sp{x}{\pf_i})_{i \in I}
\end{equation}
maps $\sig$ to a sparse sequence---this is nowadays referred to as \emph{co-sparsity} \cite{nam2013cosparse}.
In order to recover the original signal $\sig$ from its known part $\ProK \sig$ (here $\ProK$ denotes the orthogonal projection onto the ``known'' subspace of $\InpSp$), the algorithm in \cite{ElaStaQue05} suggests to solve the \emph{$\l{1}$-analysis minimization problem} given by
\begin{equation}
\min_{x} \lnorm{(\sp{x}{\pf_i})_{i \in I}}[1][I] \quad \mbox{subject to } \ProK x = \ProK\sig.
\end{equation}
The key ingredient of this algorithmic approach is the choice of the sparsifying frame. In imaging science, this relates to the question of what an image actually is.
Certainly, we cannot give a general answer to this question.
However, one classical viewpoint states that images are governed by edge-like structures meaning that anisotropic (in the sense of directional) features are dominant.
This point of view is also backed up by neurophysiology because it was shown that the human visual cortex is highly sensible to curvilinear structures \cite{hubel1959receptive,daugman1988complete}.
Images from applications such as seismic data might have the strongest bias towards such structures since they basically only consist of parabolic curves \cite{hennenfent2006application}.
This is one of the reasons why directional representation systems are currently a focus of intense research.

\subsection{From Classical Shearlets to Universal Shearlets}
\label{subsec:intro:shearlets}

About 15 years ago, it was recognized that wavelet systems are not capable to optimally encode anisotropic features. Among the approaches for directional representation systems, that were suggested to cope with this problem, the framework of \emph{curvelets} can be justifiably regarded as a breakthrough.
Indeed, curvelet systems were the first representation system which have been shown to provide optimally sparse approximations for a certain class of model functions, called \emph{cartoon-like functions}.
These are functions of the form $f = f_0 + f_1 \indset{B}$, where  $B \subset \intvcl{0}{1}^2$ is a set with $\boundary{B}$ being a closed $\Smooth{2}$-curve with bounded curvature, and $f_i \in \Smooth[\R^2]{2}$ are functions with $\supp{f_i} \subset \intvcl{0}{1}^2$ and $\norm{f_i}_{\Smooth{2}} \leq 1$ for $i=0,1$ (cf. \cite{CD2004}).

Shearlet systems were originally introduced in \cite{GKL05} as an attempt to introduce a directional representation system with similar optimal approximation properties, but with the advantage of allowing for a unified treatment of the continuum and digital situation in the sense of faithful implementations.
This could be achieved by using \emph{parabolic scaling}, \emph{shearing} as a means to change the orientation---curvelets employ rotation which destroys the digital grid---, and \emph{translation} applied to very few generating functions.
More precisely, these geometric transformations operate on some generator $\unishplain \in \Leb[\R^2]{2}$, leading to the system
\begin{equation}\label{eq:intro:shearletsystem}
\{\unishplain_{j,l,k} := 2^{3j/2}\unishplain(\pshear^l \pscal^j \dotarg - k) \suchthat j \in \Z, l \in \Z, k \in \Z^2\},
\end{equation}
where
\begin{equation}
\pscal = \matr{4 & 0 \\ 0 & 2}, \quad \pshear = \matr{1 & 1 \\ 0 & 1}
\end{equation}
denote the parabolic scaling matrix and shearing matrix, respectively.
The tiling of frequency domain of such a simple shearlet system is illustrated in Figure~\ref{fig:intro:tiling:classical}.
However, to avoid a directional bias and to promote more symmetry in the definition, this naive approach is usually enhanced by a \emph{cone-adaption}, see also Figure~\ref{fig:intro:tiling:coneadapted}.

\begin{figure}
	\centering
	\subfigure[]{
		\centering
		\includegraphics[width=.3\textwidth]{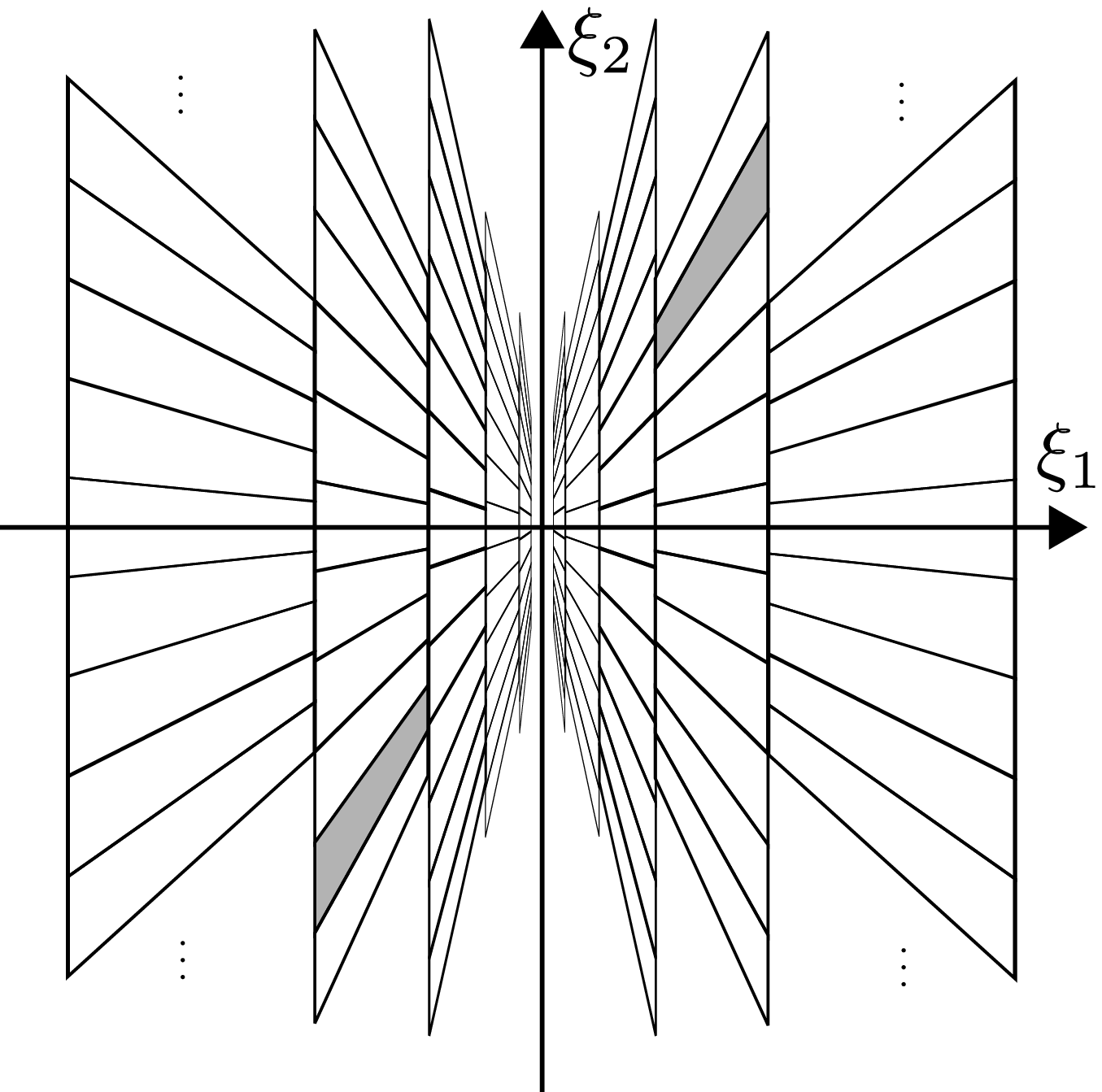}
		\label{fig:intro:tiling:classical}
	}%
	\qquad
	\subfigure[]{
		\centering
		\includegraphics[width=.3\textwidth]{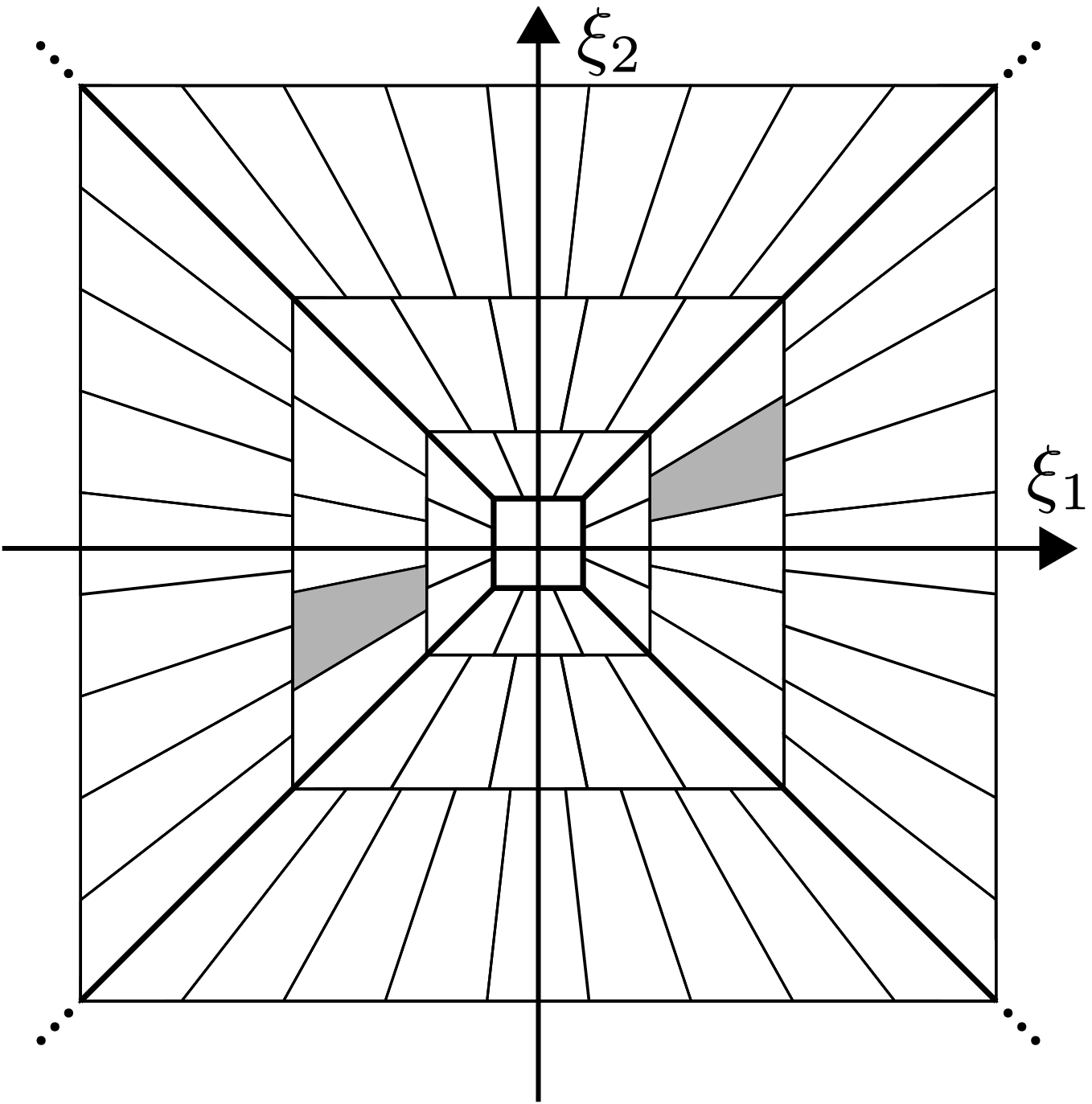}
		\label{fig:intro:tiling:coneadapted}
	}%
	\caption{\subref{fig:intro:tiling:classical} Frequency tiling of a classical shearlet system. \subref{fig:intro:tiling:coneadapted} Frequency tiling of a cone-adapted shearlet system.}
	\label{fig:intro:tiling}
\end{figure}

The original definition used band-limited generators, and the associated ``classical'' \emph{(cone-adapted) shearlet systems} indeed provide optimally sparse approximations of cartoon-like functions \cite{guo2007optimally}, while constituting a Parseval frame for $\Leb[\R^2]{2}$ at the same time.
However, due to an application of truncations in frequency domain, this frame has non-smooth Fourier transforms leading to bad spatial decay.
But as suggested in \cite{guo2013construction}, a slight modification of the seam elements can be used to obtain a band-limited Parseval frame for $\Leb[\R^2]{2}$ which consists of Schwartz functions. In order to achieve superior localization, compactly supported shearlets have then been introduced in \cite{KKL2012}, also providing similarly optimal sparsity (cf. \cite{KL11}), but coming with the drawback that they do not form a Parseval frame.

It is certainly desirable to understand the structural difference between wavelets and shearlets, not only from a theoretical side to bridge the gap and to derive a deep understanding of the transition of properties, but also from a practical perspective to provide a parametrized family that allows optimal encoding of different anisotropies.
A first approach in this direction were so-called \emph{$\alpha$-shearlets} \cite{Kei2012,KLL2012}, which form compactly supported (non-Parseval) frames.
But as mentioned before, for some applications band-limitedness or/and the Parseval property are crucial.
A natural approach to tackle this problem is to transfer the construction from \cite{guo2013construction} to this setting.

Intriguingly---and this is by far not straightforward---, one problem is the complicated interaction of the scaling parameter with the overlapping of the seam elements at the boundaries of the cones (see Figure~\ref{fig:intro:tiling:coneadapted}).
This requires the introduction of \emph{scaling sequences} $(\aj)_{j} \subseteq \intvop{-\infty}{2}$ with associated scaling matrices
\begin{equation}
	\pscal_{\aj}^j = \begin{pmatrix} 4 & 0 \\ 0 & 2^{\aj} \end{pmatrix},
\end{equation}
finally offering a maximal flexibility in scaling, since the scaling parameter could vary independently for each $j$.

Based on this idea, we introduce \emph{universal shearlet systems} in Definition~\ref{def:shearlets:unishsystem}. This notion now enables us for the first time to treat various band-limited systems---including classical cone-adapted shearlets, $\alpha$-shearlets, and band-limited wavelets---in a uniform manner. Closing the gap between (band-limited) shearlets and wavelets in a continuous way does in particular allow for a common analysis.
In addition, the relaxation in the directional scaling turns out to be even flexible enough to prove that, for a large class of scaling sequences, the associated universal shearlets form a band-limited Parseval frame for $\Leb[\R^2]{2}$ consisting of Schwartz functions (cf. Theorem~\ref{thm:shearlets:unishprop}). 

The underlying construction of universal shearlet systems presented in Section~\ref{sec:shearlets} is self-contained. The constructed systems might be interesting also for other applications because of their desirable properties: smoothness, excellent space-frequency localization, and infinitely many vanishing moments. It should be mentioned that the shearlet toolbox, presented in \cite{kutyniok2014shearlab} and available for download from \url{www.shearlab.org}, utilizes this novel construction technique and provides an implementation of compactly supported universal shearlets.

\subsection{Our Setting}

Universal shearlets are indeed perfectly suited to serve as sparsifying system in our analysis of the previously discussed inpainting algorithm, since they in particular allow for a unified view of wavelets and shearlets. 

Concerning a suitable image model, remembering that we aim to model curvilinear structures, a discrete setting seems inappropriate because it would obscure the geometric content. Therefore, we choose a continuous model,
more precisely, a \emph{distributional model} which mimics a straight (horizontal) line. Since images are typically defined on bounded domains, this line is compactly supported and smoothly truncated at its endpoints to avoid any point singularities. The mask has the shape of the vertical strip and cuts out the middle part of the line (see Figure~\ref{fig:intro:model:maskedmodel}). As simple as this approach seems, it even provides a quite accurate model for some applications; for example in seismology, where the absence of sensors produces white strips in the data, which basically only consist of smooth curves (see Figure~\ref{fig:intro:model:seismic}). Also, this model still enables us to relate the size of the gap with the length of the elements of the universal shearlet system, leading to very natural and intuitive necessary conditions for inpainting to succeed.

\begin{figure}
	\centering
	\subfigure[]{
		\centering
		\includegraphics[width=.3\textwidth]{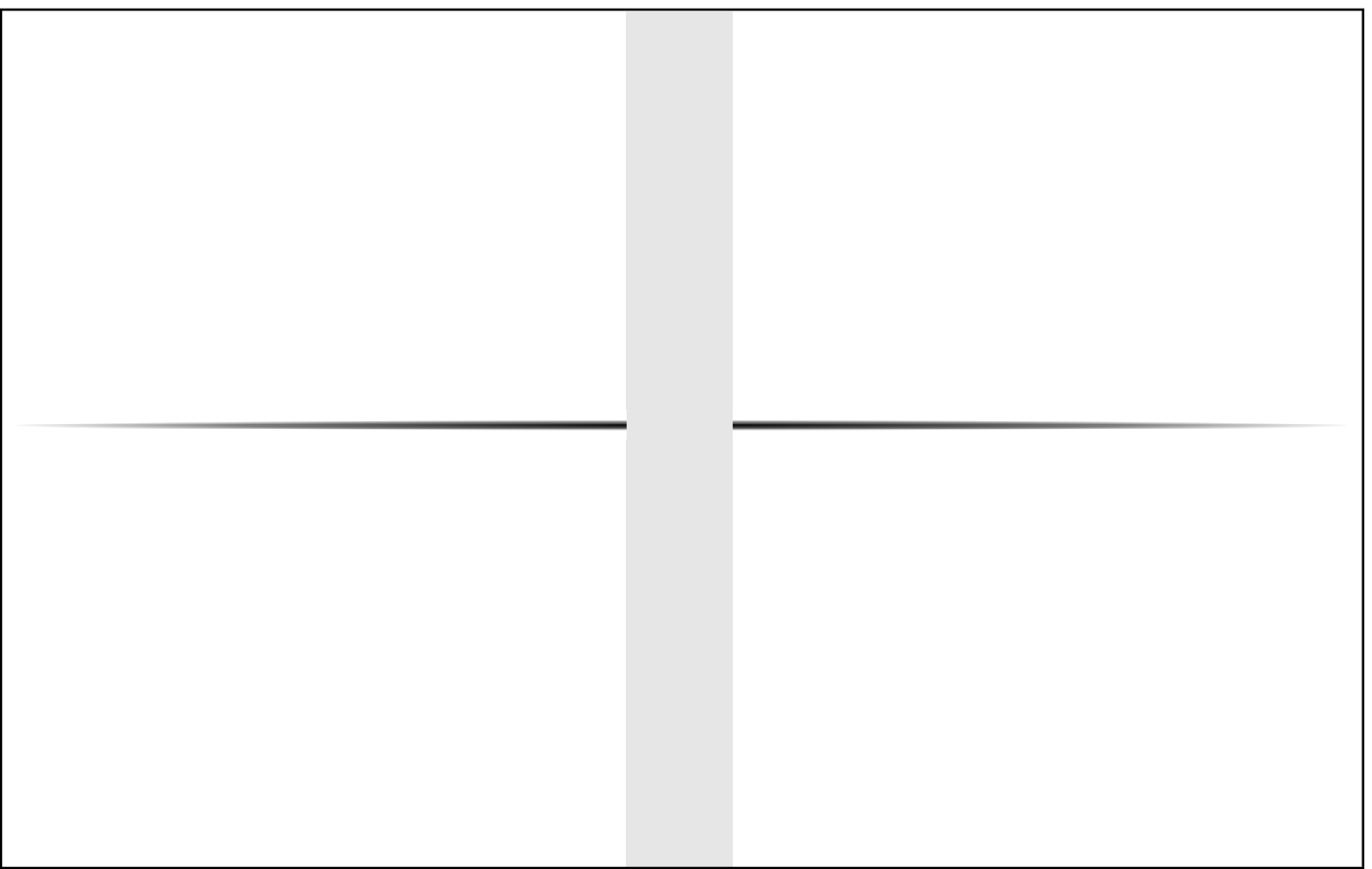}
		\label{fig:intro:model:maskedmodel}
	}%
	\qquad
	\subfigure[]{
		\centering
		\includegraphics[width=.3\textwidth]{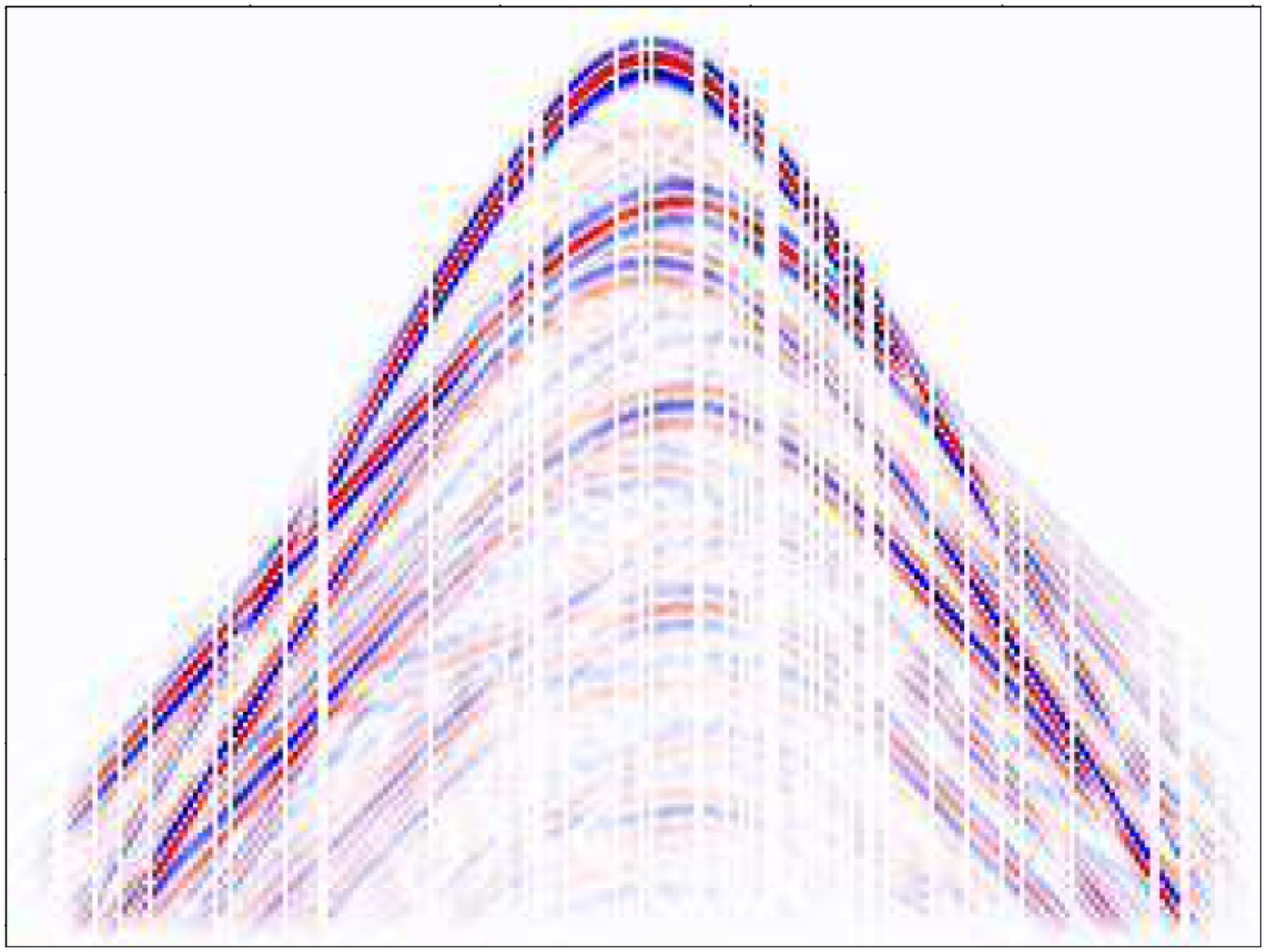}
		\label{fig:intro:model:seismic}
	}%
	\caption{\subref{fig:intro:model:maskedmodel} Line singularity with a vertical strip removed. \subref{fig:intro:model:seismic} Synthetic seismic data; some sensors failed to work resulting in randomly distributed missing traces (image from \cite{hennenfent2006application}).}
	\label{fig:intro:model}
\end{figure}

On the algorithmic side, we will choose the $\l{1}$-analysis minimization approach, which has been already discussed in Subsection~\ref{subsec:intro:inp}.

\subsection{Asymptotic Inpainting}

Through a bandpass filtering process on the model, and simultaneously allowing the gap size to vary with the scale, we are able to perform an asymptotic analysis depending on the scale. Certainly, in practical applications we can only consider a finite number of scales and the gap size usually stays fixed, but the main goal of our analysis is to derive theoretical insight into the interplay between the scaling behavior of the sparsifying system and the gap size.

Our main result concerning the recovery error of inpainting in the setting just described, Theorem~\ref{thm:imginp:asympconv}, then proves that, assuming the gap size is asymptotically smaller than the length of the corresponding shearlet elements, \emph{asymptotically perfect inpainting} is achieved. Indeed, the necessary condition provides a deep insight into the relation between the degree of anisotropy of the underlying system and the admissible gap size; and it particularly shows the structural difference between wavelets and shearlets.

\subsection{Expected Impact}
\label{subsec:intro:impact}

We anticipate our results to have the following impacts:

\begin{itemize}[leftmargin=3em]
\item
	\emph{ShearLab:} A comprehensive software package for the 2D and 3D shearlet transform is available at \url{www.shearlab.org}, see also \cite{kutyniok2014shearlab}. The provided algorithms offer the user to specify various parameters among which are also scaling parameters. However, the current shearlet theory does not incorporate this amount of flexibility. With the introduction and analysis of universal shearlets, now a first theoretical framework is available, giving also additional insight into convenient choices of scaling parameters.
\item
	\emph{Construction of compactly supported shearlet systems:} One main objective---which some people even coined the ``holy grail'' in applied harmonic analysis---is to construct a compactly supported directional representation system which forms a Parseval frame and provides optimally sparsifying approximations of cartoon-like functions. In shearlet theory, a certain class of band-limited shearlets which constitute a Parseval frame was shown to provide such optimal sparse approximations \cite{guo2007optimally}. Some years later, compactly supported shearlets were introduced and under certain conditions similar approximation rates for cartoon-like functions could be achieved \cite{KL11}; however those frames do not have the Parseval property. We believe that universal shearlets could be another attack point to this challenge due to their additional flexibility. It seems conceivable that a smart choice of scaling sequences could lead to optimally sparsifying, compactly supported shearlet systems whose frame 
bounds can be moved (arbitrarily) close to $1$.
\item
	\emph{Structured dictionary learning:} Most dictionary learning algorithms suffer from the fact that the generated systems lack in structure such as good frame bounds or fast associated transforms. Some attempts have been made to equip the generated systems with structural properties such as, for instance, the double sparsity approach in \cite{RZE10}. The system of universal shearlets can be seen in this light, allowing an adaption to the data at hand by learning the scaling sequence, but still preserving their advantageous functional analytic and numerical properties. Thus, we expect this work to eventually initiate a new approach toward structured dictionary learning with choosable objectives such as, for instance, inpainting.
\item
	\emph{Interplay of systems and masks:} It is well-known that directional representation systems are superior to wavelets for inpainting via compressed sensing methodologies assumed that the image is governed by curvilinear structures. Rigorous results are however missing; only a first analysis of a simple thresholding algorithm was presented as one result of \cite{king2014analysis}. Deriving an analysis which enables to precisely show superiority of certain representation systems over others (depending on the size and structure of the mask) will require novel tools. But we strongly believe that universal shearlets and our inpainting results can serve as a first step. For instance, the notion of scaling sequences provides a means to compare different systems with respect to their anisotropy in a much more careful manner as possible before.
\end{itemize}

\subsection{Outline}

The paper is organized as follows. The abstract framework of inpainting and associated error estimates, in the setting of a general Hilbert space, are introduced and discussed in Section~\ref{sec:framework}. The novel system of universal shearlets is introduced in Section~\ref{sec:shearlets} with the proof of the main result on its Parseval frame property being outsourced to Subsection~\ref{subsec:proofs:unishprop}. Section~\ref{sec:model} is then devoted to the specific image model, that we exploit in our work, as well as its asymptotic decomposition. Having detailed the considered setting, our main result on the asymptotic behavior of shearlet inpainting and its proof are presented in Section~\ref{sec:imginp}. Finally, some possible extensions are discussed (cf. Section~\ref{sec:extension}), followed by some proofs in Section~\ref{sec:proofs}, which we decided to outsource due to their non-intuitive and more technical nature.

\section{Abstract Inpainting Framework}
\label{sec:framework}

\subsection{Modeling Spaces}

We follow \cite{king2014analysis} and first tackle the problem of inpainting in the very general setting of a separable Hilbert space $\InpSp$.
As underlying and undamaged \emph{signal}, which we ultimately want to recover, we choose an arbitrary element $\sig \in \InpSp$.
To incorporate the fact that some ``information'' of $\sig$ got lost, we assume that $\InpSp$ decomposes into an orthogonal sum of two closed subspaces, namely a \emph{known part} and a \emph{missing part},
\begin{equation}
\InpSp = \InpSpK \orthsum \InpSpM = \ProK \InpSp \orthsum \ProM\InpSp,
\end{equation}
where $\ProK$ and $\ProM$ denote the orthogonal projections onto $\InpSpK$ and $\InpSpM$, respectively.
This leads to the following challenge: \emph{Given a corrupt signal $\ProK \sig$, recover the missing part $\ProM\sig$.}

This includes in particular that the space $\InpSpM$ is explicitly known to us.
Moreover, the above task implies that we are rather looking for the ``real'' missing part $\ProM\sig$ and not only for some completion of $\ProK \sig$ which appears ``reasonable'' to the beholder (for example, if objects are removed from the background, cf. \cite{KomTzi07}).
We finally wish to emphasize that the signal model and the space decomposition are defined independently, enabling us to analyze the $\sig$ and $\InpSpM$ almost separately in Subsection~\ref{subsec:framework:errorestimate}.

In terms of image inpainting, we will consider a continuous image model situation where $\InpSp = \Leb[\R^2]{2}$.
Here, the missing space is going to be $\InpSpM = \Leb[\mask{}]{2}$ for some measurable set $\mask{}\subset \R^2$, which should be seen as a \emph{mask} covering the corrupted parts of the painting. Thus, it might be helpful to already think of this concrete setting, even when dealing with the abstract framework.

\subsection{Inpainting via \texorpdfstring{$\l{1}$}{l\textonesuperior}-Minimization}

In order to formulate an appropriate reconstruction method, it is indispensable to make further assumptions on the signal and missing space.
Throughout this paper, we presume that $\sig$ can be ``efficiently'' represented by a certain Parseval frame $\PF = (\pf_i)_{i \in I}$ for $\InpSp$, which can be either chosen non-adaptively or adaptively, based on the given model.
In the classical theory of sparse representations, this would ask for solving the $\l{0}$-minimization problem
\begin{equation}\label{eq:framework:l0min}
\min_{c \in \l[I]{2}} \lnorm{c}[0][I] \quad \text{subject to } \sig = \OpSynthesis{\PF} c = \sum_{i \in I} c_i \pf_i,
\end{equation}
i.e., one tries to find a synthesis sequence of $\sig$ that has as few as possible non-zero elements measured by $\lnorm{c}[0] := \cardinality\{ c_i \suchthat c_i \neq 0 \}$. Those coefficients may be later identified with the ``significant'' features of  $\sig$.

As in \cite{donoho2013microlocal, king2014analysis}, we will pursue a different strategy---the so-called \emph{analysis approach}---which seems to be more appropriate for our purposes.
Instead of regarding all synthesis sequences for $\sig$, we merely consider the \emph{analysis coefficients} $\OpAnalysis{\PF}\sig = (\sp{\sig}{\pf_i})_{i \in I}$ which, due to the Parseval
property, allow for a reconstruction by $\sig = \OpSynthesis{\PF}(\sp{\sig}{\pf_i})_{i \in I} = \sum_{i \in I} \sp{\sig}{\pf_i} \pf_i$.
Since $\PF$ does not necessarily need to constitute a basis, $\OpAnalysis{\PF}\sig$ might not be a (sparse) solution of \eqref{eq:framework:l0min}.
It is therefore not completely clear what it means to provide a sparse representation in the sense of this analysis approach; a precise definition will follow in the next subsection.
However, one major benefit of this strategy is the circumvention of certain problems that could occur when using synthesis coefficients, e.g., numerical instabilities.
Moreover, we just aim for a good reconstruction in $\InpSp$; thus, knowing the sparsest representation might not give further profit \cite{king2014analysis}.

Inspired by the ideas of \emph{compressed sensing} (cf. \cite{DDEK12,FR2013}), we now present a recovery algorithm which was already used in \cite{kinganalysis, king2014analysis}:
\begin{framed}
\begin{algorithm}[Inpainting via $\l{1}$-minimization]\label{algo:framework:l1min}\leavevmode


\textsc{Input:} corrupted signal $\ProK \sig \in \InpSpK$, Parseval frame $\PF = (\pf_i)_{i \in I}$ for $\InpSp$

\textsc{Compute:} 
\begin{equation} \tag{$\l{1}$\scshape -Inp} \label{eq:framework:l1min}
	\sigrec = \argmin_{x \in \InpSp}\lnorm{\OpAnalysis{\PF} x}[1][I] \quad \text{subject to } \ProK \sig = \ProK x
\end{equation}

\textsc{Output:} recovered signal $\sigrec \in \InpSp$

\end{algorithm}
\end{framed}
Summarized in words, Algorithm~\ref{algo:framework:l1min} minimizes the $\l{1}$-norm among all possible reconstruction candidates, which is the set of all signals coinciding with $\sig$ on $\InpSpK$.
If we assume that the undamaged signal $\sig$ is sufficiently sparsified by $\PF$---this particularly implies a small $\l{1}$-norm of $\OpAnalysis{\PF}\sig$---the completion $\sigrec$ hopefully provides a good recovery respecting the significant features of $\sig$.

The numerical realization of Algorithm~\ref{algo:framework:l1min} implicitly requires that the minimum in \eqref{eq:framework:l1min} is really attained.
However, the error estimate of Theorem~\ref{thm:framework:errorl1min} (see Subsection~\ref{subsec:framework:errorestimate}) still holds for all $x \in \InpSp$ satisfying $\ProK \sig = \ProK x$ and $\lnorm{\OpAnalysis{\PF} x}[1] \leq \lnorm{\OpAnalysis{\PF} \sig}[1]$.
Assuming that $\lnorm{\OpAnalysis{\PF} \sig}[1] < \infty$, these two conditions are trivially satisfied for $x = \sig$.

We wish to mention that for various applications, e.g., geometric separation \cite{donoho2013microlocal}, the analysis approach, combined with an $\l{1}$-optimiza\-tion similar to \eqref{eq:framework:l1min}, has turned out to be successful for both empirical problems and theoretical analysis. Interestingly, Algorithm~\ref{algo:framework:l1min} is closely
related to some other optimization methods.
For instance, it can be viewed as a relaxation of the follwing \emph{co-sparsity} problem:
\begin{equation}
\sigrec = \argmin_{x \in \InpSp}\lnorm{\OpAnalysis{\PF} x}[0][I] \quad \text{subject to } \ProK \sig = \ProK x.
\end{equation}
General theoretical results on co-sparsity can be found in \cite{nam2011cosparse,nam2013cosparse}.
Another class of reconstruction methods is the one of thresholding algorithms.
In a certain sense, they can be seen as approximations of the $\l{1}$-problem given by \eqref{eq:framework:l1min}.
A very basic version of thresholding, called ``one-step thresholding'', was introduced in \cite{Kut14} and in an adapted form used in \cite{king2014analysis}.
It surprisingly turned out that very similar convergence results can be proven for this approach (cf. \cite{king2014analysis}), although the algorithmic
approach seems much cruder than \eqref{eq:framework:l1min}.

\subsection{Analysis Tools}
\label{subsec:framework:tools}

\subsubsection{\texorpdfstring{$\delta$}{delta}-Clustered Sparsity}

To prove appropriate error estimates for recovery by Algorithm~\ref{algo:framework:l1min}, we first need to introduce several analysis tools that give us further insight into the structure of the abstract model. The following definition introduces the novel concept of \emph{clustered sparsity}. It was originally used to study the geometric separation problem \cite{donoho2013microlocal}, and precisely specifies our notion of (approximate) sparsity:

\begin{definition}[\cite{king2014analysis}]
\label{def:framework:deltasparse}
Let $\delta > 0$ and $\cluster \subset I$.
A signal $x \in \InpSp$ is called \emph{$\delta$-clustered sparse} in $\PF$ with respect to $\cluster$, if
\begin{equation} \label{eq:framework:deltasparse}
	\lnorm{\indcoeff{\setcompl{\cluster}} \OpAnalysis{\PF} x}[1] \leq \delta.
\end{equation}
In this case, $\cluster$ is said to be a \emph{$\delta$-cluster} for $x$ in $\PF$.
\end{definition}

Intuitively spoken, being $\delta$-clustered sparse for a small $\delta$ means that all coefficients lying outside of $\cluster$ are small, or vice versa:
The analysis coefficients are highly concentrated on $\cluster$. However, the coefficients contained in $\cluster$ do not necessarily need to be large.
In particular, successively enlarging $\cluster$ ensures \eqref{eq:framework:deltasparse} for smaller $\delta$, i.e., better sparsity can be always achieved independent on the structure of $x$. According to this, one should regard $\cluster$ rather as an analysis tool. In fact, Definition~\ref{def:framework:deltasparse}
only makes sense when $x$ is $\delta$-clustered sparse with respect to a ``small'' cluster $\cluster$, and $\delta$ is sufficiently small at the same time.
The concept of clustered sparsity is closely related to classical notions of (approximate) sparsity in this situation because
$\OpAnalysis{\PF}x = (\sp{\sig}{\pf_i})_{i \in I}$ contains only a few large elements (lying in $\cluster$).

Concerning inpainting, one major task is therefore to design a dictionary $\PF$ which provides sufficient (clustered) sparsity for the given benchmark signal $\sig$.
In this way, the significant features of $\sig$ get encoded in terms of an appropriate $\delta$-cluster, giving us deeper insight into the ``geometric'' structure of $\sig$.

\subsubsection{Cluster Coherence}

We did not touch any specific properties of $\InpSpM$ up to now.
Hence, there is still the need for a concept that enables us to analyze the missing part of $\sig$:
\begin{definition}[\cite{king2014analysis}]
\label{def:framework:coherence}
Let $\cluster \subset I$.
The \emph{cluster coherence} of $\PF$ with respect to $\InpSpM$ and $\cluster$ is given by
\begin{equation}
	\clustercoh{\cluster}{\ProM \PF} := \max_{j \in I} \sum_{i \in \cluster} \abs{\sp{\ProM\pf_i}{\ProM \pf_j}},
\end{equation}
where $\ProM \PF := (\ProM \pf_i)_{i \in I}$.
\end{definition}

Compared with other measures for the ``size'' of $\InpSpM$, e.g., the notion of \emph{concentration} introduced in \cite{king2014analysis}, $\clustercoh{\cluster}{\ProM \PF}$ has a very explicit form and is rather easy to estimate.
Note that the signal $\sig$ is not involved in this definition, and $\cluster$ does not need to be a $\delta$-cluster; but later on, we will use one and the same cluster for both definitions.

The following aspects shall help to acquire a better understanding of why the cluster coherence can be viewed as an abstract measure for the gap size: Since we do not know $P_M\sig$,
it is natural to ask for the ``maximal amount'' of missing information.
More precisely, we consider the Parseval expansion $\sig = \sum_{j \in I} \sp{\sig}{\pf_j} \pf_j$ and estimate the analysis coefficients of its projection onto $\InpSpM$:
\begin{equation}\label{eq:framework:coherenceintuition}
	\abs{\sp{\ProM \sig}{\pf_i}} \leq \sum_{i \in I} \abs{\sp{\sig}{\pf_j}} \abs{\sp{\ProM\pf_j}{\pf_i}} = \sum_{i \in I} \abs{\sp{\sig}{\pf_j}} \abs{\sp{\ProM\pf_j}{\ProM \pf_i}}.
\end{equation}
Similar to classical notions of coherence (cf. \cite{donoho2001uncertainty}), the scalar products $\abs{\sp{\ProM\pf_j}{\ProM \pf_i}}$ indicate to which extent measurements with $\pf_i$ and $\pf_j$, respectively, are correlated on $\InpSpM$.
Assuming that $\sig$ is also $\delta$-clustered sparse (with respect to some cluster $\cluster$), we may restrict in \eqref{eq:framework:coherenceintuition} to those frame elements $\pf_i$ which are contained in $\cluster$.
In this sense, the cluster coherence can be interpreted as a signal-independent magnitude for the lack of information; but note that the actual value of $\clustercoh{\cluster}{\ProM \PF}$ results from a strong interaction between the system $\PF$ and the ``geometry'' of $\InpSpM$.

\subsection{Abstract Error Estimate}
\label{subsec:framework:errorestimate}

\subsubsection{The \texorpdfstring{$\l{1}$}{l\textonesuperior}-Analysis Space}

First of all, we shall investigate the question of how to measure the recovery error.
In the pioneer works of \cite{donoho2013microlocal, king2014analysis}, the Hilbert space norm was used to estimate the distance between $\sig$ and its $\l{1}$-recovery $\sigrec$.
For an estimate from above, this indeed worked well; but to show an optimality result---that is, an estimate from below---measuring with $\norm{\cdot}$ seems to be inappropriate.
This drawback is essentially caused by the (Parseval) frame property of $\PF$ implying $\norm{\sig - \sigrec} = \lnorm{\OpAnalysis{\PF} (\sig - \sigrec)}[2]$.
Compared to $\l{1}$, the $\l{2}$-norm has an averaging effect on the coefficients, and therefore, sparsity features might be hidden when using the Hilbert space norm.
For the concrete example of image inpainting (cf. Subsection~\ref{subsec:imginp:optimality}), we will see that $\lnorm{\OpAnalysis{\PF} (\sig - \sigrec)}[2]$ is virtually independent of the scaling parameters $\aj$, i.e., any estimate from below would not take into account the important feature of anisotropy.

Hence, we propose a new measure which naturally arises from our analysis approach of Algorithm~\ref{algo:framework:l1min}, and enables us to detect sparsity.
\begin{definition}
Let $x \in \InpSp$ and $\PF$ be a Parseval frame for $\InpSp$.
Then we define the \emph{$\l{1}$-analysis norm}, with respect to $\PF$, by
\begin{equation}
\anorm{x}{\PF} := \lnorm{\OpAnalysis{\PF} x}[1] = \lnorm{(\sp{x}{\pf_i})_{i \in I}}[1],
\end{equation}
and the \emph{$\l{1}$-analysis space} is given by $\InpSp_{1,\PF} := \{ x \in \InpSp \suchthat \anorm{x}{\PF} < \infty\}$, equipped with $\anorm{\cdot}{\PF}$.
\end{definition}


The tuple $(\InpSp_{1,\PF}, \anorm{\cdot}{\PF})$ indeed defines a normed vector space, due to the injectivity of $\OpAnalysis{\PF}$.
Moreover,
\begin{equation} \label{eq:framework:l1ananormbyhilbnorm}
	\norm{x} = \lnorm{\OpAnalysis{\PF} x}[2] \leq \lnorm{\OpAnalysis{\PF} x}[1] = \anorm{x}{\PF}
\end{equation}
yields the embedding $\InpSp_{1,\PF} \embeds \InpSp$.
In particular, by choosing $x = \sigrec - \sig$, we conclude that small recovery errors in $\InpSp_{1,\PF}$ always imply small ones in $\InpSp$.

To give more intuition, we may think of $\PF$ to be a shearlet system in $\Leb[\R^2]{2}$ for the moment.
In this situation, $\anorm{x}{\PF}$ becomes particularly large when $x$ is governed by curvilinear structures.
It was already mentioned in the introductory part that there is a surprising connection to the human visual system, which is highly sensible to directional features \cite{hubel1959receptive,daugman1988complete}.
Therefore, the $\l{1}$-analysis norm evaluates the content of images in a similar way as the human brain does, and a small reconstruction error will probably give good empirical inpainting results.

\subsubsection{Recovery Error Estimate}

We conclude this section with an abstract estimation of the recovery error.
A detailed proof of the following theorem can be found in \cite{king2014analysis} (actually, the corresponding statement in \cite{king2014analysis} is an estimate for $\norm{\sigrec - \sig}$, but the proof makes use of \eqref{eq:framework:l1ananormbyhilbnorm} in the very beginning and \eqref{eq:framework:errorl1min} is then shown explicitly).

\begin{theorem}
\label{thm:framework:errorl1min}
Let $\delta > 0$ and $\cluster \subset I$ be a $\delta$-cluster for $\sig$ in $\PF$.
Moreover, assume that $\clustercoh{\cluster}{\ProM \PF} < 1/2$.
If $\sig \in \InpSp_{1,\PF}$, then the $\l{1}$-minimizer $\sigrec$ of Algorithm~\ref{algo:framework:l1min} is also contained in $\InpSp_{1,\PF}$ and satisfies
\begin{equation} \label{eq:framework:errorl1min}
	\anorm{\sigrec - \sig}{\PF} \leq \frac{2\delta}{1 - 2 \clustercoh{\cluster}{\ProM \PF}}.
\end{equation}
\end{theorem}

Theorem~\ref{thm:framework:errorl1min} provides intuitively good estimates if $\PF$ sparsifies $\sig$ and the size of the missing space, measured in terms of cluster coherence, is not too large.
More precisely, we indent to choose a cluster $\cluster$ such that $\sig$ becomes $\delta$-clustered sparse for small $\delta$ and, simultaneously, $\clustercoh{}{}$ shall not exceed the bound of $1/2$.
While the notions of cluster sparsity and cluster coherence were introduced independently in the previous subsection, both terms are now coupled in terms of $\cluster$.
In particular, a variation of this cluster reveals the essential trade-off relation: Enlarging $\cluster$ makes \eqref{eq:framework:errorl1min} true for smaller $\delta$, whereas $\clustercoh{\cluster}{\ProM \PF}$ increases and could exceed $1/2$ at some point.
Applying Theorem~\ref{thm:framework:errorl1min} therefore involves the fundamental task of choosing an appropriate cluster \cite{king2014analysis}.

In the following sections, the above abstract framework is transferred to the specific case of a continuous image model in $\InpSp = \Leb[\R^2]{2}$.
Based on Theorem~\ref{thm:framework:errorl1min}, we will finally present a convergence result in Section~\ref{sec:imginp} which rigorously proves the success of image inpainting via (universal) shearlets.

\section{Universal Shearlets}
\label{sec:shearlets}

It was already emphasized in the introductory part that shearlets provide an excellent, non-adaptive dictionary to represent anisotropic structures. In this section---which is essentially self-contained, and might be of interest on its own---we present a novel shearlet construction which overcomes the limitations of parabolic scaling.
The crucial feature of our new system is a generalized scaling matrix,\footnote{The $(\hor)$ stands for horizontal and its meaning will become clear in the sequel.}
\begin{equation}
\pscalcone{\aj}{\hor} = \matr{2^{2j} & 0 \\ 0 & 2^{\aj j}},
\end{equation}
where the \emph{scaling parameters} $\aj$ can be chosen independently for each scaling level $j$.
Compared with usual cone-adapted shearlet systems, this offers much more flexibility, and allows us to achieve an almost arbitrary scaling behavior.
The following definitions are based on the original construction of Guo and Labate in \cite{guo2013construction}.
In particular, our adaption still constitutes a smooth Parseval frame for $\Leb[\R^2]{2}$ and preserves the band-limiting property, as well as rapid decay in the spacial domain (cf. Theorem~\ref{thm:shearlets:unishprop}). An implementation of universal shearlets is provided in \url{www.shearlab.org}, see also \cite{kutyniok2014shearlab}.

Before presenting more technical details, let us recall the notion of \emph{rapidly decreasing} functions (or \emph{Schwartz functions})
\begin{equation}
\Schwartz[\R^d] := \Big\{ \defsf \in \Smooth[\R^d]{\infty} \suchthat \forall K, N \in \Nzero \colon \sup_{x \in \R^d} (1 + \abs{x}^2)^{-N/2} \sum_{\abs{\alpha} \leq K} \abs{\der{\alpha}\defsf(x)} < \infty \Big\}.
\end{equation}
For the sake of convenience, we will make use of the compact notation $\Schwartzpoly{x} := (1 + \abs{x}^2)^{-N/2}$. The \emph{Fourier transform}, given by
\begin{equation}
\ftop \colon \Schwartz[\R^d] \to \Schwartz[\R^d], \ \ft{\defsf}(\xi) :=  \integ[\R^d]{\defsf(x) e^{-2\pi i \xi^\T x}}{dx},
\end{equation}
is then a well-defined operator, which extends to an unitary operator form $\Leb[\R^d]{2}$ to $\Leb[\R^d]{2}$ (cf. \cite{grafakos2008classical}).

\subsection{Construction and Basic Properties}
\label{subsec:shearlets:construction}

The basic component of our construction is a function $\meyerscal \in \Schwartz[\R]$ satisfying $0\leq \ft\meyerscal \leq 1$, $\ft\meyerscal(u) = 1$ for $u \in \intvcl{-\frac{1}{16}}{\frac{1}{16}}$, and $\supp\ft\meyerscal \subset \intvcl{-\frac{1}{8}}{\frac{1}{8}}$.
A function with these properties is commonly known as a \emph{Meyer scaling function} (see Figure~\ref{fig:shearlets:meyerscalplot}).
Now, we define the \emph{corona scaling functions} for $j \in \Nzero$ by ($\xi = (\xi_1, \xi_2) \in \R^2$)
\begin{align}
	\ft\Scalfunc(\xi) &:= \ft\meyerscal(\xi_1)\ft\meyerscal(\xi_2), \\
	\Corofunc(\xi) &:= \sqrt{\ft\Scalfunc^2(2^{-2}\xi) - \ft\Scalfunc^2(\xi)}, \\
	\Corofunc_j(\xi) &:= \Corofunc(2^{-2j} \xi).
\end{align}
The definition of $\meyerscal$ implies that $\Corofunc_j$ is compactly supported in corona-shaped \emph{scaling levels}, i.e.,
\begin{equation} \label{eq:shearlets:corosupp}
	\supp\Corofunc_j \subset \Coro_j := \intvcl{-2^{2j-1}}{2^{2j-1}}^2 \setminus \intvop{-2^{2j-4}}{2^{2j-4}}^2,
\end{equation}
and---this is the crucial feature---the Fourier domain is decomposed by the sequence of scaling functions:
\begin{equation} \label{eq:shearlets:corotiling}
	\ft\Scalfunc^2(\xi) + \sum_{j \geq 0} \Corofunc_j^2(\xi) = 1, \quad \xi \in \R^2.
\end{equation}
For an illustration, we refer to Figure~\ref{fig:shearlets:coro:scalinglevels}.
The function $\Corofunc$ can be viewed as a wavelet, in the sense that it satisfies \eqref{eq:shearlets:corotiling}, which is a version of the \emph{discrete Calder\'{o}n condition} (cf. \cite{daubechies1992ten}).
In fact, this property is key to ensuring that our shearlet system is capable to detect edge singularities.

\begin{figure}
	\centering
	\subfigure[]{
		\centering
		\includegraphics[width=.3\textwidth]{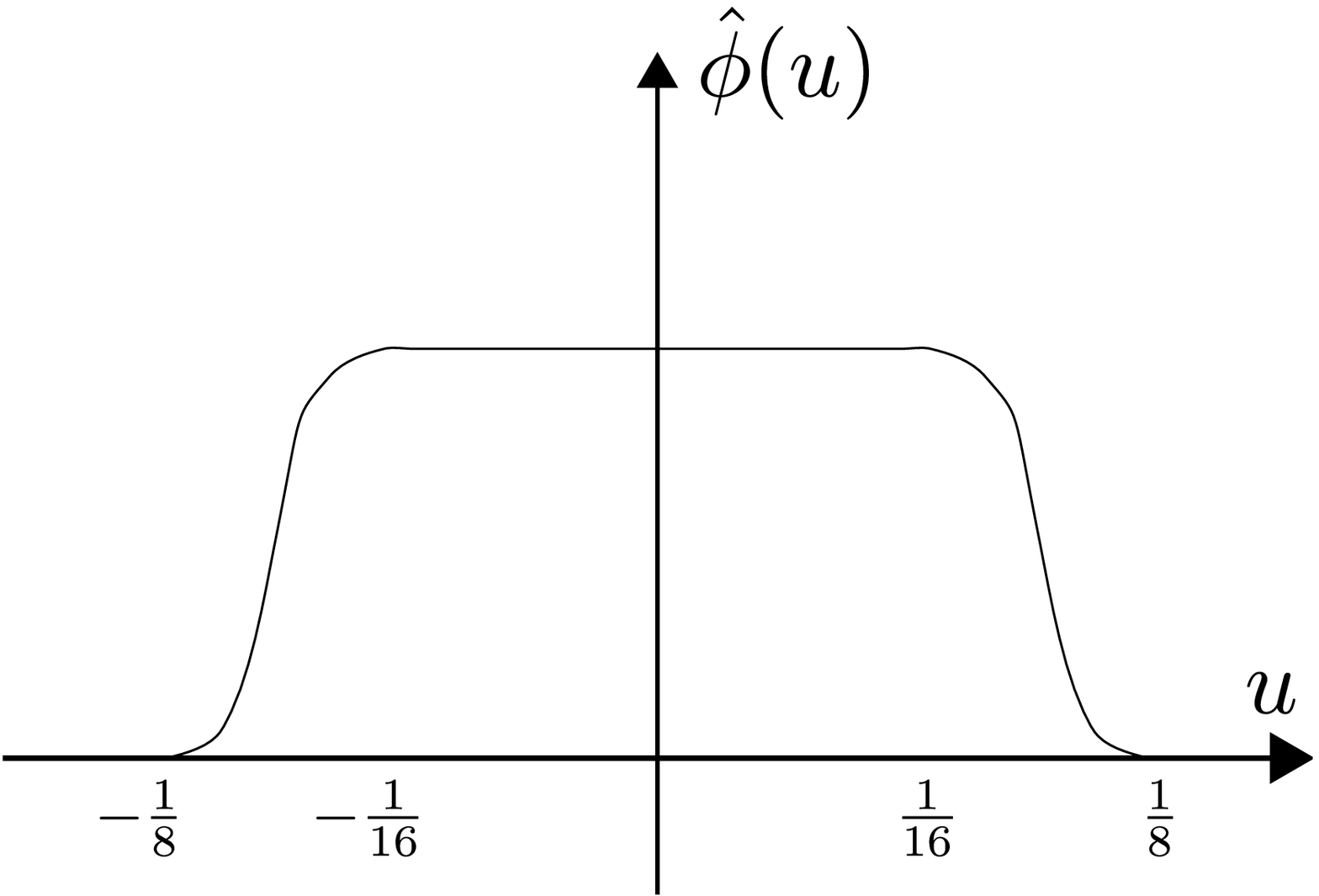}
		\label{fig:shearlets:meyerscalplot}
	}%
	\qquad
	\subfigure[]{
		\centering
		\includegraphics[width=.25\textwidth]{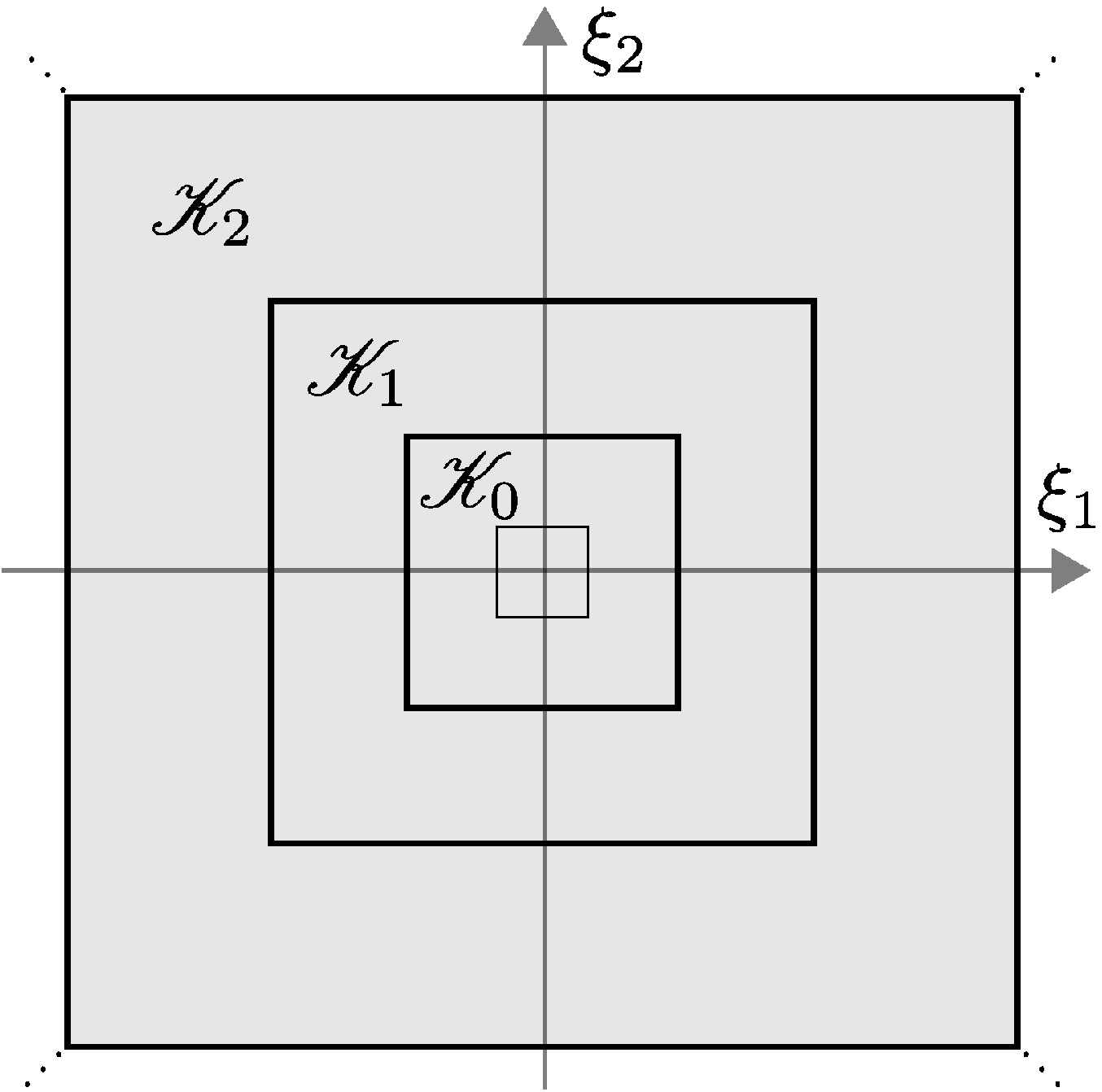}
		\label{fig:shearlets:coro:scalinglevels}
	}%
	\qquad
	\subfigure[]{
		\centering
		\includegraphics[width=.25\textwidth]{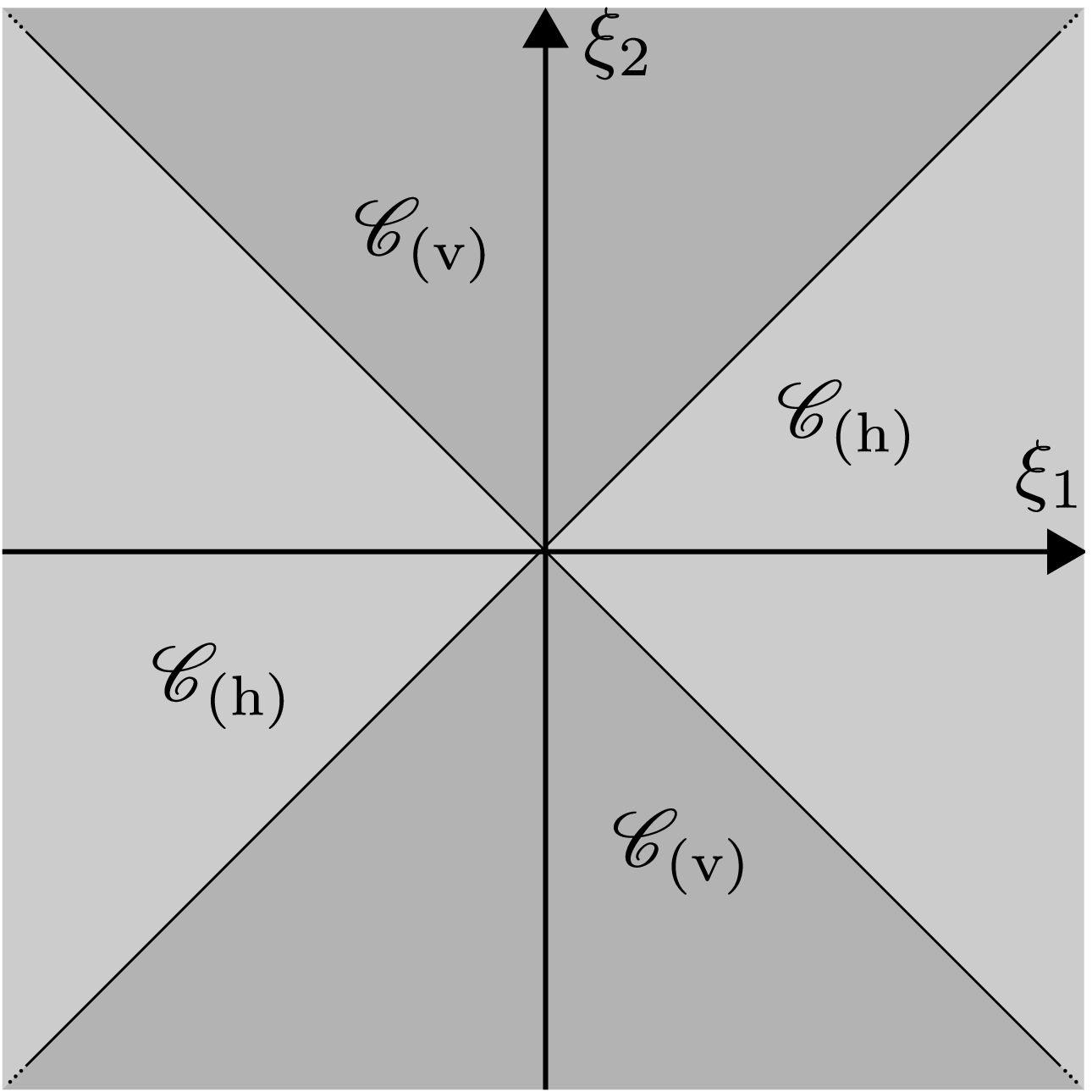}
		\label{fig:shearlets:cones}
	}
	\caption{Ingredients of the construction: \subref{fig:shearlets:meyerscalplot} Fourier transform of a Meyer scaling function.
	\subref{fig:shearlets:coro:scalinglevels} Decomposition of the frequency plane by corona functions $\Corofunc_j$ and $\Scalfunc$. Note that the corona shapes $\Coro_j$ could slightly overlap.
	\subref{fig:shearlets:cones} Symmetric frequency decomposition by cones.}
	\label{fig:shearlets:coro}
\end{figure}

Next, we consider a \emph{bump-like} function $\conefunc \in \Smooth[\R]{\infty}$ which satisfies $\supp \conefunc \subset \intvcl{-1}{1}$ and
\begin{align}
	&\abs{\conefunc(u -1)}^2 + \abs{\conefunc(u)}^2 + \abs{\conefunc(u+1)}^2 = 1 \text{ for $u \in \intvcl{-1}{1}$, and} \label{eq:shearlets:bump1} \\
	&\conefunc(0) = 1 \text{ and } \conefunc^{(n)}(0) = 0 \text{ for $n \geq 1$.} \label{eq:shearlets:bump2}
\end{align}
An explicit construction of $\conefunc$ can be found in \cite{guo2007optimally}.
Compared to $\Corofunc$, the actual role of $\conefunc$ will be to produce the directional scaling feature of the system.

The simplicity of the shearlet system that was defined in the introductory part of this work (see \eqref{eq:intro:shearletsystem}) causes a directional bias which is already recognizable in Figure~\ref{fig:intro:tiling:classical}: For large shearing parameters $l$, the corresponding elements do increasingly elongate along the $\xi_2$-axis (in the Fourier domain).
This drawback will lead to problems when approximating horizontal edges and renders the system ineffective for practical applications \cite{guo2013construction}.
To circumvent this problem, one first restricts the number of shears such that the remaining system is contained in the \emph{horizontal frequency cone}
\begin{equation}
\Cone{\hor} := \left\{ (\xi_1, \xi_2) \in \R^2 \suchthat \abs[auto]{\tfrac{\xi_2}{\xi_1}} \leq 1 \right\}.
\end{equation}
Then, by interchanging the variables $\xi_1$ and $\xi_2$, we can define a flipped system on the \emph{vertical frequency cone}
\begin{equation}
	\Cone{\ver} := \left\{ (\xi_1, \xi_2) \in \R^2 \suchthat \abs[auto]{\tfrac{\xi_1}{\xi_2}} \leq 1 \right\}.
\end{equation}
This \emph{cone-based approach}, initially introduced in \cite{GKL05}, provides a much more symmetric frequency tiling (see also Figure~\ref{fig:shearlets:cones}),
and has turned out to be suited for both theoretical analysis and applications (cf. \cite{KL2012}).

To give a precise definition of a cone-adapted shearlet system, we still need to introduce adapted versions of the usual shearing and scaling matrix,
\begin{alignat}{3}
	\pscalcone{\alpha}{\hor} &:= \matr{2^2 & 0 \\ 0 & 2^\alpha}, &\quad
	\shearcone{\hor} &:= \matr{1 & 1 \\ 0 & 1}, \\
	\pscalcone{\alpha}{\ver} &:= \matr{2^\alpha & 0 \\ 0 & 2^2}, &\quad
	\shearcone{\ver} &:= \matr{1 & 0 \\ 1 & 1},
\end{alignat}
where $\alpha \in \intvop{-\infty}{2}$ is the \emph{scaling parameter}.
Similarly, the adapted \emph{cone functions} are given by
\begin{equation}
	\Conefunc{\hor}(\xi) := \conefunc\opleft( \tfrac{\xi_2}{\xi_1} \opright), \quad \Conefunc{\ver}(\xi) := \conefunc\opleft( \tfrac{\xi_1}{\xi_2} \opright), \quad \xi \in \R^2.
\end{equation}

Now, we are able to define the single ingredients of a universal shearlet system; notice that the
definition of a universal shearlet system will then follow in Definition~\ref{def:shearlets:unishsystem}.
The respective parts of the following definition are visualized in Figure~\ref{fig:shearlets:unishdef}.
\begin{definition}
\label{def:shearlets:unishdef}
Let $\Scalfunc, \Corofunc, \Conefunc{\hor}, \Conefunc{\ver} \in \Leb[\R^2]{2}$ be defined as before.
\begin{enumerate}
\item
	\emph{Coarse scaling functions}:
	For $k \in \Z^2$, we set
	\begin{equation}
		\unishplain_{-1,k} (x) := \Scalfunc(x - k), \quad x \in \R^2.
	\end{equation}
\item
	\emph{Interior shearlets}:
	Let $\alpha \in \intvop{-\infty}{2}$, $j \in \Nzero$, $l \in \Z$ with $\abs{l} < 2^{(2-\alpha) j}$, $k \in \Z^2$ and $\dir \in \{\hor, \ver\}$.
	Then we define $\unish[\dir]{\alpha}{j}{l}{k}$ by its Fourier representation,
	\begin{equation}\label{eq:shearlets:unishdef:interior}
		\unishft[\dir]{\alpha}{j}{l}{k} (\xi) := 2^{-(2+\alpha)j/2} \Corofunc( 2^{-2j}\xi ) \Conefunc{\dir}\opleft( \xi^\T \pscalcone{\alpha}{\dir}^{-j} \shearcone{\dir}^{-l} \opright) e^{-2\pi i \xi^\T \pscalcone{\alpha}{\dir}^{-j} \shearcone{\dir}^{-l} k}, \quad \xi \in \R^2.
	\end{equation}
\item
	\emph{Boundary shearlets}:
	For $\alpha \in \intvop{-\infty}{2}$, $j \geq 1$, $l = \pm \ceil{2^{(2-\alpha) j}}$ and $k \in \Z^2$, we define
	\begin{equation}\label{eq:shearlets:unishdef:boundary}
		\unishft{\alpha}{j}{l}{k} (\xi) := \begin{cases}
			2^{-\frac{(2+\alpha)j}{2}-\frac{1}{2}} \Corofunc( 2^{-2j}\xi )\Conefunc{\hor}\opleft( \xi^\T \pscalcone{\alpha}{\hor}^{-j} \shearcone{\hor}^{-l} \opright) e^{-\pi i \xi^\T \pscalcone{\alpha}{\hor}^{-j} \shearcone{\hor}^{-l} k}, & \xi \in \Cone{\hor}, \\
			2^{-\frac{(2+\alpha)j}{2}-\frac{1}{2}} \Corofunc( 2^{-2j}\xi )\Conefunc{\ver}\opleft( \xi^\T \pscalcone{\alpha}{\ver}^{-j} \shearcone{\ver}^{-l} \opright) e^{-\pi i \xi^\T \pscalcone{\alpha}{\hor}^{-j} \shearcone{\hor}^{-l} k}, & \xi \in \Cone{\ver},
		\end{cases}
	\end{equation}
	and in the case $j = 0$, $l = \pm1$, we put
	\begin{equation}
		\unishft{\alpha}{0}{l}{k} (\xi) := \begin{cases}
			\Corofunc( \xi )\Conefunc{\hor}\opleft( \xi^\T \shearcone{\hor}^{-l} \opright) e^{-2\pi i \xi^\T k}, & \xi \in \Cone{\hor}, \\
			\Corofunc( \xi )\Conefunc{\ver}\opleft( \xi^\T \shearcone{\ver}^{-l} \opright) e^{-2\pi i \xi^\T k}, & \xi \in \Cone{\ver}.
		\end{cases}
	\end{equation}
\end{enumerate}
\end{definition}

\begin{figure}
	\centering
	\includegraphics[width=0.3\textwidth]{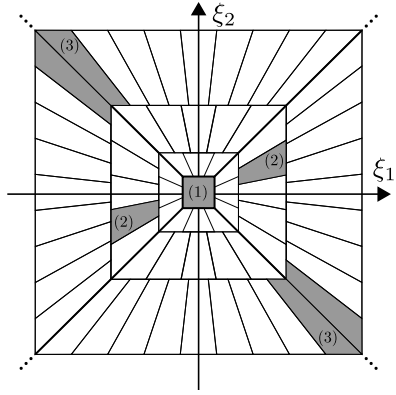}
	\caption{Illustration of the Fourier supports of the different components.
	(1) Coarse scaling part covering low frequencies.
	(2) Interior shearlets entirely contained in the interior of the cones.
	(3) Boundary shearlets ``glued'' together along the boundaries of the cones.}
	\label{fig:shearlets:unishdef}
\end{figure}

We remark that the case $\alpha = 1$ produces exactly the original parabolic-scaling construction of \cite{guo2013construction}.
Compared to the Fourier transform of the ``classical'' shearlet system $\{\unishplain_{j,l,k} \suchthat j \in \Z, l \in \Z, k \in \Z^2\}$ in Subsection~\ref{subsec:intro:shearlets}, the interior shearlets of equation \eqref{eq:shearlets:unishdef:interior} are not \emph{shear-invariant}, meaning that $\Corofunc_j$ contains no shearing term.
Hence, the system of Definition~\ref{def:shearlets:unishdef} has no chance to be singly generated.
However, $\unishft[\hor]{\alpha}{j}{l}{k}$ has compact support in the trapezoidal region
\begin{equation} \label{eq:shearlets:unishdef:trapezoidsupport}
	\left\{ \xi \in \R^2 \suchthat \xi_1 \in \intvcl{-2^{2j-1}}{2^{2j-1}} \setminus \intvop{-2^{2j-4}}{2^{2j-4}}, \abs[auto]{\tfrac{\xi_2}{\xi_1} - l 2^{-(2-\alpha)j}} \leq 2^{-(2-\alpha)j} \right\},
\end{equation}
which essentially coincides with $\supp \unishplainft_{j,l,k}$ at least when $\alpha = 1$.
Thus, $\unish[\hor]{\alpha}{j}{l}{k}$ and $\unishplain_{j,l,k}$ have virtually the same space-frequency behavior.
For a more detailed discussion see also \cite{guo2013construction}.

Up to now, it is quite unclear if the elements of Definition~\ref{def:shearlets:unishdef} form a (Parseval) frame for $\Leb[\R^2]{2}$.
A rather simple approach to satisfy the frame property would be to project the system onto the subspaces
\begin{equation}
\ift[long]{\Leb[\Cone{\dir}]{2}} = \{ f \in \Leb[\R^2]{2} \suchthat \supp\ft{f} \subset \Cone{\dir} \}, \quad \dir = \hor, \ver,
\end{equation}
corresponding to a rough truncation of the Fourier supports.
Then, one can show that these truncated systems constitute frames for $\ift[long][normal]{\Leb[\Cone{\dir}]{2}}$, respectively (cf. \cite{guo2013construction}); their union indeed yields a frame for the whole $\Leb[\R^2]{2}$.
However, cutting-off along the boundaries of the cones causes discontinuities which finally lead to bad decay behavior in space.

It is therefore much more favorable to consider the system on the entire space $\Leb[\R^2]{2}$.
Due to the piecewise definition of the boundary shearlets, smoothness (of the Fourier representations) is not guaranteed for all $\alpha \in \intvop{-\infty}{2}$.
It turns out that the exponent of $2^{(2-\alpha)j}$ needs necessarily to be integer-valued when analyzing the first partial derivatives of \eqref{eq:shearlets:unishdef:boundary}---this is explicitly done in the proof of Theorem~\ref{thm:shearlets:unishprop} (cf. equation \eqref{eq:proofs:unishprop:partderhor} and \eqref{eq:proofs:unishprop:partderver}).
To satisfy this for every $j \in \Z$, we would have to restrict the set of admissible $\alpha$ to the set of integers; but this would completely burst the major goal of arbitrary scaling behavior.

The key idea to overcome this problem is to relax the condition of a globally fixed $\alpha$ and, instead of that, to introduce a \emph{separate} scaling parameter $\aj$ on \emph{each} scale.
Thus, the universal shearlet system will be associated with a whole sequence $(\aj)_{j \in \Nzero}$.
In this looser setting, the set of admissible scaling parameters can be substantially enlarged because we have $(2-\aj)j \in \Z$ whenever $\aj$ is a multiple of $1/j$.
This approach of discretization naturally gives rise to the following notion:

\begin{definition}
\label{def:shearlets:scalingsequence}
A sequence $(\aj)_{j \in \Nzero} \subset \R$ is called a \emph{scaling sequence} if
\begin{equation}
	\aj \in \Scalparamdomain_j := \left\{ \tfrac{m}{j} \suchthat m \in \Z, m \leq 2j-1 \right\} = \left\{ \dotsc, -\tfrac{2}{j}, -\tfrac{1}{j}, 0, \tfrac{1}{j}, \dotsc, 2 - \tfrac{1}{j} \right\}
\end{equation}
for $j \geq 1$ and $\alpha_0 = 0$.
\end{definition}

Now, we are able to give the main definition of this section:
\begin{definition}
\label{def:shearlets:unishsystem}
Let $(\aj)_{j \in \Nzero}$ be a scaling sequence.
Then we define the associated \emph{universal-scaling shearlet system}, or shorter, \emph{universal shearlet system}, by
\begin{equation}
	\Unish := \UnishLow \union \UnishInt \union \UnishBound,
\end{equation}
where
\begin{align}
	& \UnishLow := \left\{ \unishplain_{-1,k} \suchthat k \in \Z^2 \right\}, \\
	& \UnishInt := \left\{ \unish[\dir]{\aj}{j}{l}{k} \suchthat j \geq 0, \abs{l} < 2^{(2-\aj) j}, k \in \Z^2, \dir \in \{\hor,\ver \} \right\}, \\
	& \UnishBound := \left\{ \unish{\aj}{j}{l}{k} \suchthat j \geq 0, \abs{l} = \pm 2^{(2-\aj) j}, k \in \Z^2  \right\}.
\end{align}
\end{definition}

Notice that the elements of a scaling sequence $\aj$ can be chosen independently for each scaling level and, in particular, the sequence $(\aj)_j$ does not need to converge.
Interestingly, this additional degree of freedom allows for a unification of the formal shearlet definition and its numerical implementation.
In practice, one usually determines the total number of shearings (on each scale) first, rather then selecting a fixed $\alpha$.
The scaling parameter is then implicitly given and can be computed by an easy relation (cf. \cite{kutyniok2014shearlab}).
But since the amount of shearings is necessarily integer-valued, the resulting $\aj$ might not coincide for all $j$.
According to this phenomenon, the formal concept of a scaling sequence matches precisely to what is done in practice.

The following theorem collects the most important properties of universal shearlets.
In particular, all systems provide excellent localization in both Fourier and spatial domain.
We postpone the proof to Subsection~\ref{subsec:proofs:unishprop}.
\begin{theorem}
\label{thm:shearlets:unishprop}
Let $(\aj)_{j}$ be a scaling sequence and $\Unish$ be an associated universal shearlet system.
Then $\Unish$ constitutes a Parseval frame for $\Leb[\R^2]{2}$ consisting of band-limited Schwartz functions.
Moreover, the interior and boundary shearlets have infinitely many vanishing moments.
\end{theorem}

To illustrate the flexibility of our new construction, we finally present two natural choices of $(\aj)_j$ inducing well-known representation systems.

\subsection{\texorpdfstring{$\alpha$}{alpha}-Shearlets}
\label{subsec:shearlets:alphashearlet}

Let $\alpha \in (-\infty, 2)$, and recall the definition of $\Scalparamdomain_j$ from Definition~\ref{def:shearlets:scalingsequence}.
We then choose $(\aj)_j$ such that it provides the best possible approximation of $\alpha$, i.e.,
\begin{equation}
	\aj := \argmin_{\tilde\alpha_j \in \Scalparamdomain_j}\abs{\tilde\alpha_j - \alpha}, \quad j \geq 1.
\end{equation}
One easily verifies $2^{\aj j} \in \asympeq{2^{\alpha j}}$ as $j \to \infty$.
Hence, the corresponding system $\Unish$ has asymptotically the same scaling behavior as \emph{$\alpha$-shearlets},
see \cite{KLL2012,Kei2012} and more generally \cite{GKKS14}.
If $\alpha = 1$, we exactly obtain the cone-adapted shearlet frame from \cite{guo2013construction}.
On the other hand, the choice of a small (possibly negative) $\alpha$ produces more anisotropic elements; letting $\aj \to -\infty$ this could even give rise to an approximation of the ``ridgelet case''.

\subsection{Wavelets}
\label{subsec:shearlets:wavelet}

\begin{figure}
	\centering
	\includegraphics[width=0.3\textwidth]{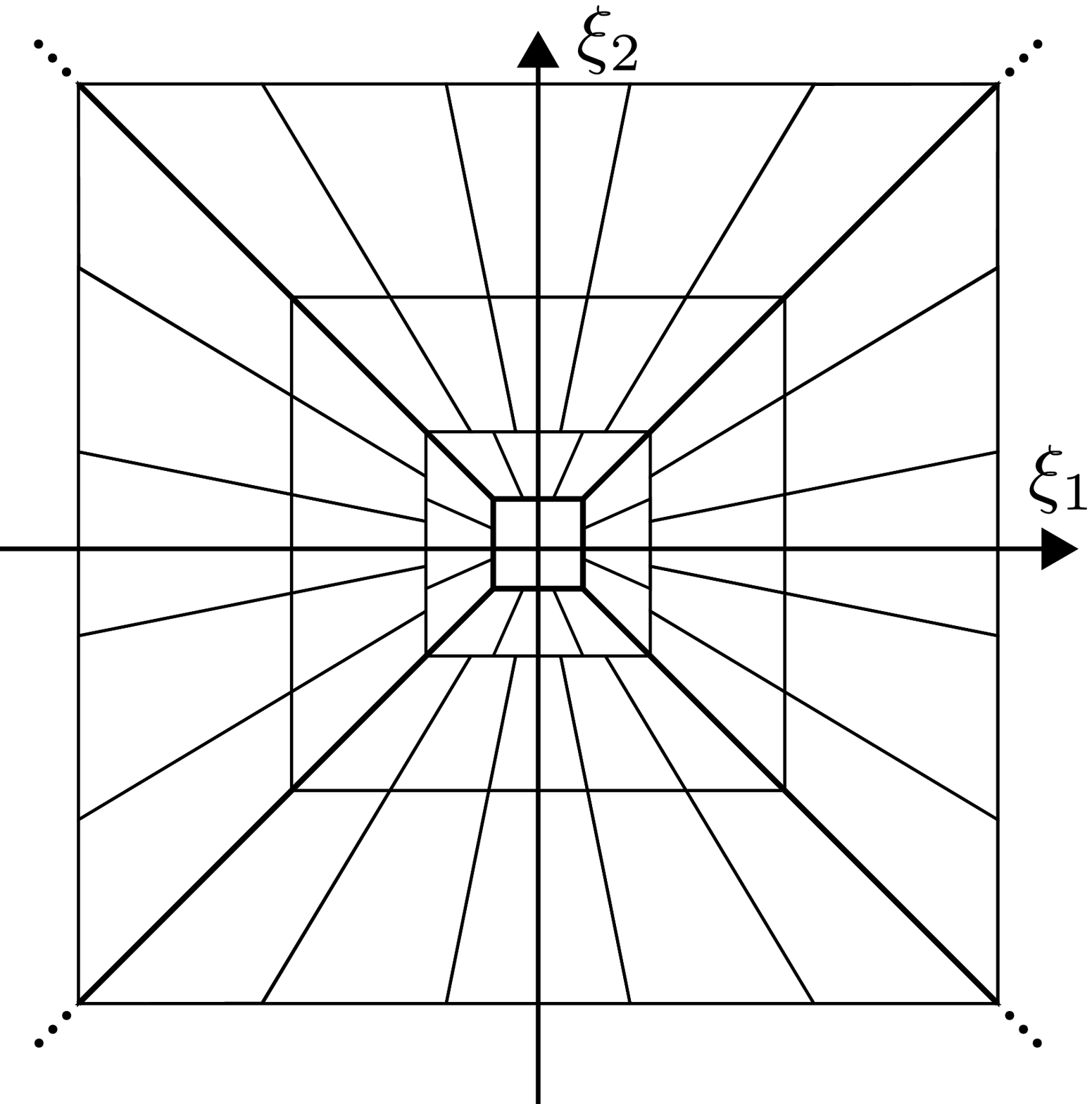}
	\caption{Frequency tiling obtained by $\Unish$ in the wavelet case.}
	\label{fig:shearlets:wavelet}
\end{figure}

While the original scaling parameter $\alpha$ had to be strictly less than $2$, it is now possible to reach this limit case.
Concretely, we choose $(\aj)_j$ as large as possible, namely $\aj := 2 - \frac{1}{j}$ for $j \geq 1$.
Then we have $2^{\aj j} \in \asympeq{2^{2 j}}$ implying that the elements of $\Unish$ scale in an \emph{isotropic} fashion.
Furthermore, the number of shears is constant (equal to $10$) on each scaling level (see Figure~\ref{fig:shearlets:wavelet}).
Therefore, $\Unish$ can be viewed as a special kind of a wavelet frame.
This particularly allows us to treat wavelets and shearlets in a uniform fashion and to compare their respective (inpainting) analysis directly.

\section{Image Model on \texorpdfstring{$\Leb[\R^2]{2}$}{L\texttwosuperior(R\texttwosuperior)}}
\label{sec:model}

\subsection{Line Distribution and Mask}
\label{subsec:model:distr}

As already motivated in the first section, we intend to analyze the model of a corrupted line segment.
From a theoretical point of view, such a line object shall not have any spatial widening (i.e., no thickness) and its support shall be parametrizable by a curve in $\R^2$.
It would not be clear how to distinguish surfaces and curvilinear structures otherwise.
However, one-dimensional objects have measure zero in $\R^2$ and cannot be directly represented by $\Leb{2}$-functions.
It is therefore convenient to work in the space of tempered distributions $\Tempdistr[\R^2]$, which is the set of all linear and continuous functionals on $\Schwartz[\R^2]$.
To adapt the example of seismic image data, we follow \cite{king2014analysis} again, and consider a compactly supported edge singularity along the $x_1$-axis.

Let $0 \neq \weight \in \Smooth[\R]{\infty}$ be a \emph{weighting function} satisfying $0 \leq \weight \leq 1$ and $\supp\weight \subset \intvcl{-\wlen}{\wlen}$ for some fixed $\wlen > 0$.
The \emph{weighted edge singularity} $\model \in \Tempdistr[\R^2]$ is then given by the product of $\weight$ and a line distribution along the $x_1$-axis,
\begin{equation}
	\distrinp{\model}{\defsf} := \distrinp{\Linedistr}{\weight\defsf} := \integ[-\wlen][\wlen]{\weight(x_1) \defsf(x_1, 0)}{dx_1}, \quad \defsf \in \Schwartz[\R^2].
\end{equation}
From this definition, it immediately follows that $\model$ has compact support on $\intvcl{-\wlen}{\wlen} \times \{0\}$.
The ``graph'' of $\model$ is sketched in Figure~\ref{fig:model:distr:edgesingularitygraph}.

\begin{figure}
	\centering
	\subfigure[]{
		\centering
		\includegraphics[width=.3\textwidth]{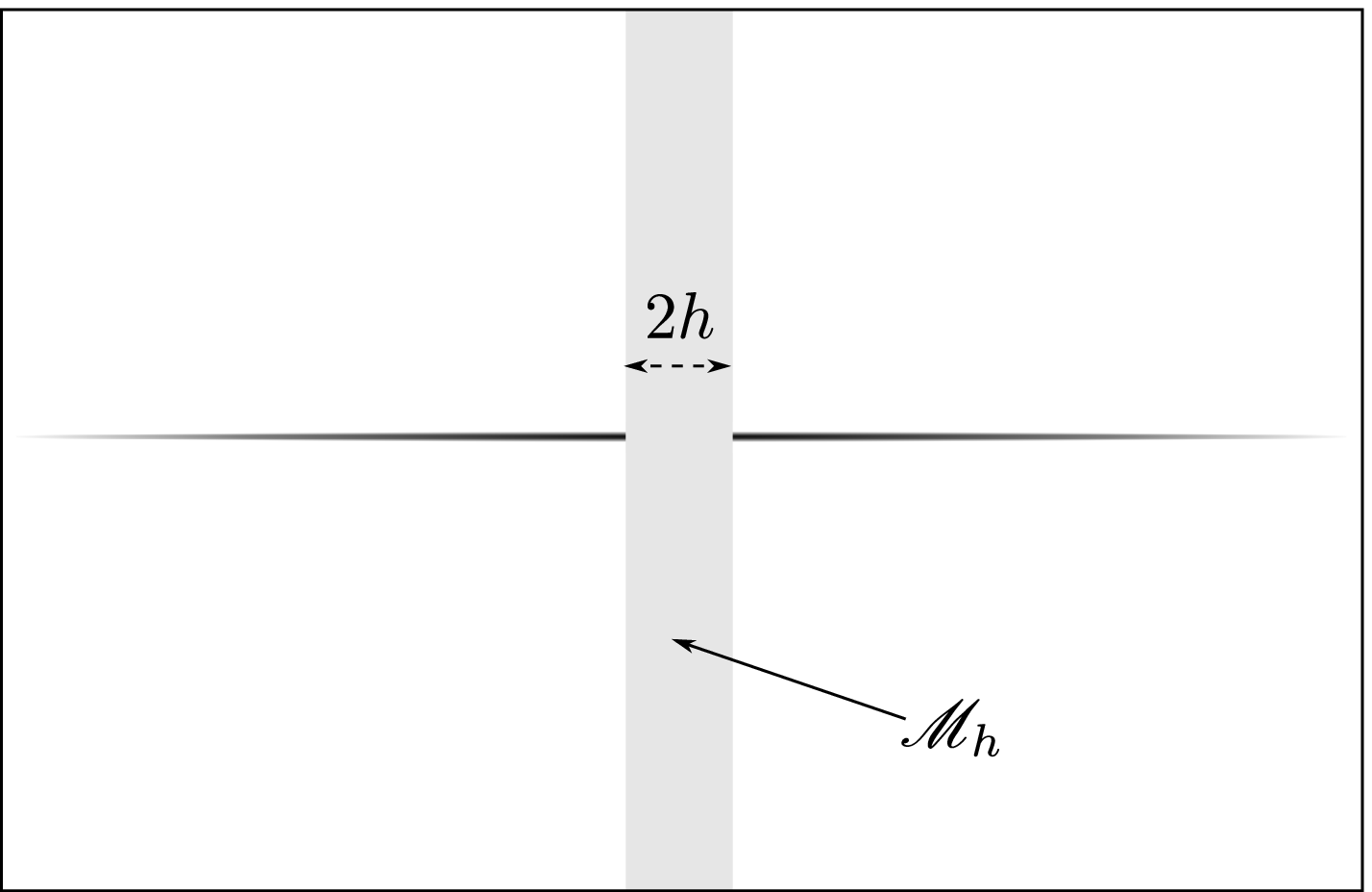}
		\label{fig:model:distr:edgesingularity}
	}%
	\qquad
	\subfigure[]{
		\centering
		\includegraphics[width=.3\textwidth]{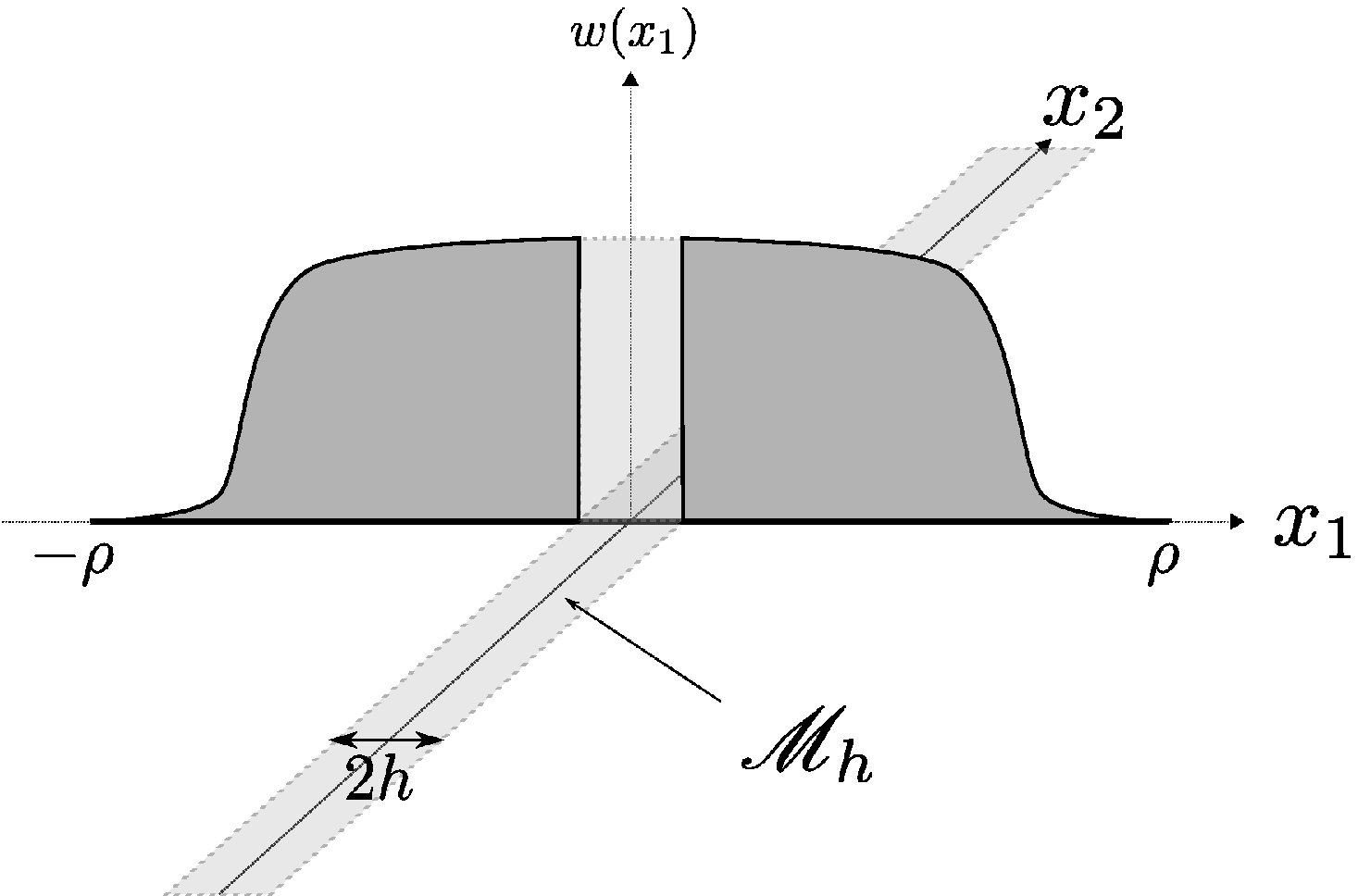}
		\label{fig:model:distr:edgesingularitygraph}
	}%
	\caption{\subref{fig:model:distr:edgesingularity} Sketch of the corrupted modeling image.
	Note that the line singularity has no ``thickness'' but some (gray-scale) intensity.
	\subref{fig:model:distr:edgesingularitygraph} ``Graph'' of the corrupted line distribution $\ProK\model$, which is compactly supported on the $x_1$-axis.}
	\label{fig:model:distr}
\end{figure}

We will frequently work with Fourier representations in the forthcoming section.
Thus, we already compute the Fourier transform of $\model$ at this point:
\begin{equation}
	\distrinp{\ft\model}{\defsf}
	=\distrinp{\model}{\ft\defsf}
	= \integ[-\wlen][\wlen]{\weight(x_1) \left( \integ[\R^2]{\defsf(\xi)e^{-2\pi i \xi_1 x_1}}{d\xi} \right)}{dx_1} = \integ[\R^2]{\ft\weight(\xi_1) \defsf(\xi)}{d\xi}, \quad \defsf \in \Schwartz[\R^2].
\end{equation}
This particularly shows that $\ft\model$ is a regular distribution with $\ft\model(\xi) = \ft\weight(\xi_1)$, and $\ft\model \in \Smooth[\R^2]{\infty}[b]$.
However, it has no decay in the $\xi_2$-direction and $\model$ ``carries'' therefore arbitrarily high frequencies.

To specify the destroyed region of our image model, we let us be inspired by the motivating example of seismic data once more: We cut off a small
vertical strip around the $x_2$-axis which is given by the symmetric \emph{mask}
\begin{equation}
	\mask{\mdiam} := \{ (x_1, x_2) \in \R^2 \suchthat \abs{x_1} \leq \mdiam \}
\end{equation}
of diameter $2\mdiam > 0$ (see Figure~\ref{fig:model:distr} for visualization).
The gap space is then set to $\InpSpM = \Leb[\mask{\mdiam}]{2}$.
Note that, for this choice, the orthogonal projections can be written in terms of characteristic functions:
\begin{equation}
	\ProM f = \indset{\mask{\mdiam}} \hphantom{\cdot} f, \quad \ProK f = \indset{\setcompl{\mask{\mdiam}}} \hphantom{\cdot} f, \quad f \in \Leb[\R^2]{2}.
\end{equation}

\subsection{Filter Decomposition}
\label{subsec:model:decomp}

Due to the fact that the energy of $\model$ is infinite, we need to break this model down to $\Leb[\R^2]{2}$ in a reasonable manner.
For the sake of convenience, this decomposition should be also compatible with the universal shearlet systems, which will serve as underlying Parseval frames for $\l{1}$-inpainting (cf. Algorithm~\ref{algo:framework:l1min}).
According to that, we introduce a sequence of \emph{frequency filters} $\Corofilter_j \in \Schwartz[\R^2]$, $j \geq 0$, which are defined by the inverse Fourier transform of the corona functions from the shearlet construction (cf. Subsection~\ref{subsec:shearlets:construction}), i.e.,
\begin{equation}
	\ft\Corofilter_j(\xi) := \Corofunc_j(\xi) = \Corofunc(2^{-2j} \xi), \quad \xi \in \R^2, \quad j \geq 0.
\end{equation}
So we obtain a whole sequence of (scale-dependent) \emph{benchmark models} $(\model_j)_{j \geq 0}$ given by
\begin{equation}
\model_j(x) := \model \conv \Corofilter_j(x) = \distrinp{\model}{\Corofilter_j(x-\dotarg)}, \quad x \in \R^2.
\end{equation}
Note that---although the definition of universal shearlets is involved here---there is no dependence on the scaling sequence $(\aj)_j$.
Hence, the models are comparable with different choices of $\aj$.
Using relation \eqref{eq:shearlets:corotiling}, the original image model can be recovered from $(\model_j)_{j \geq 0}$ by
\begin{equation}\label{eq:model:decomposerecover}
	\model = \Scalfunc \conv (\Scalfunc \conv \model) + \sum_{j \geq 0} \Corofilter_j \conv \model_j,
\end{equation}
where the first term of the sum corresponds to the low-frequency part of $\model$.

A computation of the Fourier transforms,
\begin{equation} \label{eq:model:fourierfiltered}
	\ft\model_j(\xi) = \ft\model(\xi) \cdot \ft\Corofilter_j(\xi) = \ft\weight(\xi_1) \cdot \Corofunc_j(\xi), \quad \xi \in \R^2,
\end{equation}
shows that the $\model_j$ are band-limited Schwartz functions and have therefore finite energy, that is $\model_j \in \Leb[\R^2]{2}$.
Furthermore, \eqref{eq:model:fourierfiltered} allows an important conclusion on the filter decomposition: Due to the multiplication with $\Corofunc_j$, the frequency support of $\model$ gets truncated to the corona $\Coro_j$ (see Figure~\ref{fig:shearlets:coro:scalinglevels}).
Recalling the geometry of $\Coro_j$, this implies that $\model_j$ covers increasingly higher frequency parts of $\model$ (as $j$ grows).
But at the same time, the energy of $\model_j$ strongly concentrates around the horizontal axis:
\begin{equation}
\abs{\model_j(x_1, x_2)} \leq C_N \Schwartzpoly{2^{2j}x_2}^{-N}, \quad \text{for every $N \in \N$,}
\end{equation}
see also computation \eqref{eq:proofs:spatialdecay:modeldecay} for more details.
Roughly spoken, we approximate a sharp edge singularity when sending $j$ to infinity.
This also coincides with our intuition that paintings with high frequency content are often governed by curvilinear structures.
Hence, the following inpainting analysis will mainly focus on the asymptotic case $j \to \infty$.

Contemplating, at the first sight, the above frequency decomposition appears to be a rather artificial and technical approach to
represent an edge model in $\Leb[\R^2]{2}$, but this is actually a crucial step toward a theoretical result on inpainting: If we would
analyze only one single image model, it would be essentially unclear how to evaluate the quality of the recovery.
On the contrary, an asymptotic analysis of a whole sequence of models---more precisely, we consider the inpainting behavior for $(\model_j)_{j \geq 0}$ as $j \to \infty$---allows us to make a statement about convergence, independently from any concrete error quantities.

\section{A Convergence Theorem for Image Inpainting}
\label{sec:imginp}

In the previous two sections, we have specified the ingredients of the image inpainting problem formulated by Algorithm~\ref{algo:framework:l1min} in case of $\InpSp = \Leb[\R^2]{2}$.
Now, we will apply the abstract error estimate of Theorem~\ref{thm:framework:errorl1min} to show the following convergence theorem, which is the main result of this paper and formally proves the success of image inpainting via universal shearlets. Note that in particular the inpainting results from \cite{king2014analysis} for shearlets and wavelets included as special cases.

\begin{theorem}
\label{thm:imginp:asympconv}
Let $(\aj)_{j}$ be a scaling sequence and let $\Unishshort := \Unish$ be an associated universal shearlet system.
We choose a gap width $\mdiam_j > 0$ for every $j \geq 0$ and consider the models $\ProK \model_j = \indset{\setcompl{\mask{\mdiam_j}}}\model_j$
of the corrupted image.
The respective recoveries provided by Algorithm~\ref{algo:framework:l1min}, using $\Unishshort$, are denoted by $\modelrec_j$.
Furthermore, we assume that the following two conditions are satisfied:
\begin{properties}[Inp]
\item\label{item:imginp:asympconv:aj}
	There exist $\delta > 0$ and $j_0 \in \N$ such that $\aj > \delta$ for all $j \geq j_0$, i.e., $\liminf_{j \to \infty} \aj > 0$.
\item\label{item:imginp:asympconv:gaps}
	For a fixed $\eps > 0$, the sequence of gap widths satisfies $(\mdiam_j)_j \in \asympffaster{2^{-(\aj+\eps) j}}$.
\end{properties}
Then, the (relative) recovery error decays \emph{rapidly} in the $\l{1}$-analysis norm, i.e.,
\begin{equation}\label{eq:imginp:asympconv}
	\left( \frac{\anorm{\modelrec_j - \model_j}{\Unishshort}}{\anorm{\model_j}{\Unishshort}} \right)_{\mathclap{j}} \in \asympffaster{2^{-Nj}}, \quad \text{as $j \to \infty$}
\end{equation}
for every $N \in \Nzero$.
(Note that the asymptotic constants in \eqref{eq:imginp:asympconv} may depend on $N$.)
\end{theorem}

Theorem~\ref{thm:imginp:asympconv} will be a direct consequence of the Propositions~\ref{prop:imginp:clusteredsparse} and \ref{prop:imginp:clustercoh}---their proof will be the major challenge of this section---applied to Theorem~\ref{thm:framework:errorl1min}.
Note that the normalization in \eqref{eq:imginp:asympconv} has only a minor influence on the actually statement, as Corollary~\ref{cor:imginp:analysisnorm} will show.

Before giving a complete proof, we wish to point out an important perspective of how to read this result.
As already discussed at the end of the previous section, Theorem~\ref{thm:imginp:asympconv} provides an asymptotical statement about the inpainting
performance for $\ProK \model_j$ when the scaling level becomes sufficiently high.
In fact, each model is reconstructed separately, and also the corresponding gap width $\mdiam_j$ depends on the scale $j$; this presumption is essential because the Fourier support of $\model_j$ is scale-depended.

Under the assumption of ``adequate control of gap size'' in terms of \ref{item:imginp:asympconv:gaps}, we indeed obtain rapid error decay.
As supposed in \cite{king2014analysis}, we call the underlying recovery process \emph{asymptotically perfect} whenever \eqref{eq:imginp:asympconv} holds.
In the sense of a theoretical analysis, an asymptotically perfect recovery of a model sequence should be rather seen as a \emph{benchmark criterion} for the success of inpainting.
The actual question that one should ask is therefore:
\begin{center}
\emph{Under which (sufficient and/or necessary) conditions do we obtain \eqref{eq:imginp:asympconv}?}
\end{center}

Form this point of view, the most interesting aspects of Theorem~\ref{thm:imginp:asympconv} are the (sufficient) conditions \ref{item:imginp:asympconv:aj} and \ref{item:imginp:asympconv:gaps}.
The first one excludes the artificial case of negative scaling: When the elements $\aj$ of the scaling sequence are all non-positive, the essential length of the shearlet element becomes constant or even blows up to infinity.
In fact, this restriction is quite natural because otherwise, we could fill in gaps of arbitrary size for a finitely supported model.

The second condition forms the ``heart'' of Theorem~\ref{thm:imginp:asympconv}.
Intuitively spoken, \ref{item:imginp:asympconv:gaps} restricts the allowed gap widths so that they can be covered by one single shearlet element (see also Figure~\ref{fig:imginp:gapcover}).
In this way, Theorem~\ref{thm:imginp:asympconv} relates the degree of anisotropic scaling to the admissible gap sizes: When $\aj$ becomes smaller, the asymptotical condition $(\mdiam_j)_j \in \asympffaster{2^{-(\aj+\eps) j}}$ is satisfied for larger $\mdiam_j$.
For the two sample systems from Subsection~\ref{subsec:shearlets:alphashearlet} and \ref{subsec:shearlets:wavelet}, we particularly conclude that, in terms of image inpainting, ($\alpha$-)shearlet frames are superior to isotropic wavelet systems; more general, increasing the degree of anisotropy (by making $\alpha$ smaller) ``continuously'' improves the above convergence result.

To summarize, the essence of our formal inpainting analysis is a \emph{comparison} of several representation frames that vary in their degree of directional scaling.
However, the theorem gives only sufficient conditions for asymptotically perfect inpainting; a rigorous proof that, for instance, shearlets are superior to wavelets
would also require a necessary condition. This (still open) problem is discussed in Subsection~\ref{subsec:imginp:optimality}.

\begin{figure}
	\centering
	\includegraphics[width=0.25\textwidth]{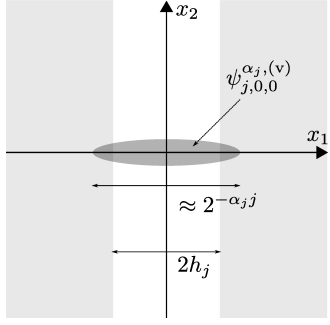}
	\caption{Visualization of condition \ref{item:imginp:asympconv:gaps}.
	The gap width $\mdiam_j$ shall not exceed the essential length of the horizontally oriented shearlet elements (on scale $j$).}
	\label{fig:imginp:gapcover}
\end{figure}

\subsection{Clustered Sparsity}
\label{subsec:imginp:clusteredsparse}

At first, we shall investigate the Fourier support properties of the modeling functions $\model_j$.
For the sake of convenience, we introduce some indexing sets for universal shearlet frames:
\begin{align}
	\Unishgroup &:= \begin{aligned}[t]
		& \{ (j, l, k; \aj, \dir) \suchthat j \geq 0, \abs{l} < 2^{(2-\aj)j}, k \in \Z^2, \dir \in \{\hor, \ver \} \} \\
		& \union \{ (j, l, k; \aj, \emptyset) \suchthat j \geq 0, \abs{l} = 2^{(2-\aj)j}, k \in \Z^2 \},
	\end{aligned} \\
	\Unishgroup_j &:= \{ (j', l, k; \aj, \dir) \in \Unishgroup \suchthat j' = j \}, \quad j \geq 0.
\end{align}
Using the notation $\unishplain_\Unishind := \unish[\dir]{\alpha}{j}{l}{k}$ for $\Unishind = (j, l, k; \aj, \dir) \in \Unishgroup$ (here, $\dir = \emptyset$ means that this index is omitted), we can write in short
\begin{equation}
	\bigunion_{\Unishind \in \Unishgroup} \unishplain_\Unishind = \UnishInt \union \UnishBound.
\end{equation}

By definition, the support of $\ft\model_j$ is contained in the frequency corona $\Coro_j$.
Due to \eqref{eq:shearlets:corosupp}, this region also overlaps with the neighboring coronas $\Coro_{j-1}$ and $\Coro_{j+1}$, but is disjoint with all the other ones.
Hence, for any $j \geq 1$, we easily conclude that
\begin{align}\label{eq:imginp:scalpm}
	\sp{\model_j}{\unishplain_{-1,k}} &= 0, \quad k \in \Z^2,\\
	\sp{\model_j}{\unish[\dir]{\alpha}{j'}{l}{k}} &= 0, \quad (j', l, k; \aj, \dir) \in \Unishgroup_{j'}, \ \abs{j' - j} > 1. 
\end{align}
This shows that all non-zero coefficients are actually contained in the subsequence $(\sp{\model_j}{\unishplain_\Unishind})_{\Unishind \in \scalpm{\Unishgroup_j} }$, where we used a comfortable notation from \cite{donoho2013microlocal}:
\begin{equation}
	\scalpm{A_j} := A_{j-1} \union A_{j} \union A_{j+1}, \quad j \geq 0, \quad \text{$(A_j)_{j \in \Nzero}$ are arbitrary sets and $A_{-1} := \emptyset$}.
\end{equation}
At this point, it also should become more clear why we have chosen the functions $\Corofilter_j = \ift\Corofunc_j$ as frequency filters in Section~\ref{sec:model}.

Now, we fix some $\eps > 0$ (this will be the same as in \ref{item:imginp:asympconv:gaps}) and define the clusters
\begin{equation}
	\cluster_j := \{ (j, l, k; \aj, \ver) \suchthat \abs{l} \leq 1, k \in \Z^2, \abs{k_2 - l k_1} \leq 2^{\eps j} \}\subset \Unishgroup_j, \quad j \geq 0.
\end{equation}
Figure~\ref{fig:imginp:modelcluster} visualizes the geometric shape of $\cluster_j$: We consider only those shearlet elements which are almost horizontally oriented and concentrated in a tube of width $2^{\eps j}$ around the $x_1$-axis (recall that ``$\ver$'' means a vertical direction in frequency which corresponds to a horizontal orientation in the space).
This choice is clearly influenced by the geometry of $\model$ whose support ($\supp \model = \intvcl{-\wlen}{\wlen} \times \{0\}$) is entirely covered by every $\cluster_j$.
But note that, although the clusters are carefully chosen, there may occur very small shearlet coefficients belonging to $\cluster_j$ as well; this is due the band-limiting property of $\Unish$ which produces strong oscillations in the spatial domain.
The proofs of the main Propositions~\ref{prop:imginp:clusteredsparse} and \ref{prop:imginp:clustercoh} will show that the term $2^{\eps j}$ plays a fundamental role in our analysis.
In fact, $\eps$ controls the trade-off relation between clustered sparsity and cluster coherence that was already mentioned in the course of Theorem~\ref{thm:framework:errorl1min}.
\begin{figure}
	\centering
	\includegraphics[width=0.3\textwidth]{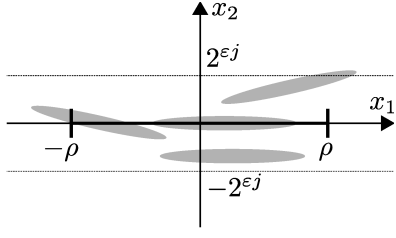}
	\caption{Some shearlet elements associated with the cluster $\cluster_j$.}
	\label{fig:imginp:modelcluster}
\end{figure}

Taking the technicality of ``neighboring coronas'' into account, we finally put
\begin{equation}
	\delta_j^{j'-j}
	:= \lnorm{\indcoeff{\Unishgroup_{j'} \setminus \cluster_{j'}} ( \OpAnalysis{\Unishshort} \model_j )}[1]
	= \sum_{\Unishind \in \Unishgroup_{j'} \setminus \cluster_{j'}} \abs{\sp{\model_j}{\unishplain_\Unishind}}, \quad \text{for } j, j' \geq 1, \ \abs{j' - j} \leq 1,
\end{equation}
and then sum up (recall that $\Unishshort = \Unish$),
\begin{equation}
	\delta_j
	:= \delta_j^{-1} + \delta_j^{0} + \delta_j^{+1}
	\stackrel{\eqref{eq:imginp:scalpm}}{=} \lnorm{\indcoeff{\setcompl{(\scalpm{\cluster}_j)}} ( \OpAnalysis{\Unishshort} \model_j )}[1] = \sum_{\Unishind \in \setcompl{(\scalpm{\cluster}_j)}} \abs{\sp{\model_j}{\unishplain_\Unishind}}, \quad \text{for } j \geq 1.
\end{equation}
By definition, $\model_j$ is $\delta_j$-clustered sparse in $\Unish$ with respect to the cluster $\scalpm{\cluster}_j$.
In the following, we will show that $\scalpm{\cluster}_j$ indeed captures the sparsity behavior:
\begin{proposition}
\label{prop:imginp:clusteredsparse}
Assuming that $\liminf_{j \to \infty} \aj > 0$, the sequence $(\delta_j)_{j \in \N}$ decays rapidly, i.e.,
\begin{equation}\label{eq:imginp:clusteredsparse}
	\delta_j \in \asympffaster{2^{-N j}}, \quad j \to \infty,
\end{equation}
for every $N \in \N$.
\end{proposition}

The proof of this proposition is essentially based on appropriate estimates of
\begin{align}
	\delta_j^{j'-j}
	= \sum_{\Unishind \in \Unishgroup_{j'} \setminus \cluster_{j'}} \abs{\sp{\model_j}{\unishplain_\Unishind}}
	={} &  \underbrace{\sum_{\substack{k \in \Z^2,\abs{l}\leq 1\\ \abs{k_2 - lk_1}>2^{\eps j'}}} \abs[\Big]{\sp[\Big]{\model_j}{\unish[\ver]{\alpha_{j'}}{j'}{l}{k}}}}_{=: T_1^{j'-j}} + \underbrace{\sum_{\substack{k \in \Z^2,\\ \abs{l}>1}} \abs[\Big]{\sp[\Big]{\model_j}{\unish[\ver]{\alpha_{j'}}{j'}{l}{k}}}}_{=: T_2^{j'-j}} \\
	&{} + \underbrace{\sum_{k \in \Z^2,l} \abs[\Big]{\sp[\Big]{\model_j}{\unish[\hor]{\alpha_{j'}}{j'}{l}{k}}}}_{=: T_3^{j'-j}} +  \underbrace{\sum_{k \in \Z^2} \abs[\Big]{\sp[\Big]{\model_j}{\unish{\alpha_{j'}}{j'}{\pm2^{(2-\alpha_{j'})j'}}{k}}}}_{=: T_4^{j'-j}}.
\end{align}
In the following, we will just investigate the case of $j' = j$.
The respective arguments for $j' = j \pm 1$ are literally the same; only the corresponding constants may differ.
To keep the notation simple, we introduce a shortcut for the transformed translations,
\begin{align}
	\translind{t}{\hor} &= (\translind{t_1}{\hor}, \translind{t_2}{\hor}) := \pscalcone{\aj}{\hor}^{-j} \shearcone{\hor}^{-l} k = (2^{-2j}(k_1 - l k_2), 2^{-\aj j} k_2), \\
	\translind{t}{\ver} &= (\translind{t_1}{\ver}, \translind{t_2}{\ver}) := \pscalcone{\aj}{\ver}^{-j} \shearcone{\ver}^{-l} k = (2^{-\aj j} k_1, 2^{-2j}(k_2 - l k_1)),
	\end{align}
where $j \in \Z$, $l \in \Z$ and $k \in \Z^2$.
Moreover, we put (the benefit of this notation will become clear with Lemma~\ref{lem:imginp:coeffdecay})
\begin{equation}
	\translind{N}{\dir} := \translind{N}{\dir}(\translind{t}{\dir}) := \begin{cases}
		N, & t^{(\dir)} \neq 0, \\
		0, & t^{(\dir)} = 0,
	\end{cases} \quad \dir \in \{ \hor, \ver \}, \quad N \in \Nzero.
\end{equation}
The following two lemmas yield several decay estimates for the shearlet coefficients.
A detailed proof as well as some remarks on the used techniques can be found in Section~\ref{subsec:proofs:coeffdecay}.\footnote{The reader should be aware of the fact that many estimates in the proofs make heavy use of \emph{generic constants}, i.e., although a constant may change from step to step, the same symbol, usually a ``$C$'', is used. Dependencies on other constants will be indicated by sub-indices. This is often referred to the \emph{Hardy-notation} in literature.} For what follows, also recall the compact notation of ``Schwartz polynomials'' $\Schwartzpoly{x} = (1 + \abs{x}^2)^{1/2}$.
\begin{lemma}
\label{lem:imginp:coeffdecay}
Let $(j, l, k; \aj, \dir) \in \Unishgroup_j$ with $j \geq 1$.
If $\aj \geq 0$, then the following estimates hold for arbitrary integers $N_1, N_2, M \geq 0$:
\begin{romanlist}
\item \label{lem:imginp:coeffdecay:ver}
	If $\dir = \ver$ and $\abs{l} > 1$, we have
	\begin{equation}
		\abs[\Big]{\sp[\Big]{\model_j}{\unish[\ver]{\aj}{j}{l}{k}}} \leq C_{N_1,N_2,M} \abs[auto]{\translind{t_1}{\ver}}^{-\translind{N_1}{\ver}} \abs[auto]{\translind{t_2}{\ver}}^{-\translind{N_2}{\ver}} 2^{(2-\aj)j/2} 2^{-\translind{N_2}{\ver}\aj j} \Schwartzpoly{2^{\aj j}}^{-M}.
	\end{equation}
\item \label{lem:imginp:coeffdecay:hor}
	If $\dir = \hor$, we have
	\begin{equation}
		\abs[\Big]{\sp[\Big]{\model_j}{\unish[\hor]{\aj}{j}{l}{k}}} \leq C_{N_1,N_2,M} \abs[auto]{\translind{t_1}{\hor}}^{-\translind{N_1}{\hor}} \abs[auto]{\translind{t_2}{\hor}}^{-\translind{N_2}{\hor}} 2^{(2-\aj)j/2} 2^{-\translind{N_2}{\hor}\aj j} \Schwartzpoly{2^{2 j}}^{-M}.
	\end{equation}
\item \label{lem:imginp:coeffdecay:boundary}
	If $\dir = \emptyset$ and $\abs{l} = 2^{(2-\aj)j}$, we have
	\begin{equation}
		\abs[\Big]{\sp[\Big]{\model_j}{\unish{\aj}{j}{l}{k}}} \leq C_{N_1,N_2,M} \abs[auto]{\translind{t_1}{\hor}}^{-\translind{N_1}{\hor}} \abs[auto]{\translind{t_2}{\hor}}^{-\translind{N_2}{\hor}} 2^{(2-\aj)j/2} 2^{-\translind{N_2}{\hor}\aj j} \Schwartzpoly{2^{2 j}}^{-M}.
	\end{equation}
\end{romanlist}
\end{lemma}

\begin{lemma}
\label{lem:imginp:spatialdecay}
Let $(j, l, k; \aj, \ver) \in \Unishgroup_j$ with $j \geq 1$ and $\abs{l}\leq 1$.
Then, for every integer $N \geq 2$, we have
\begin{equation} \label{eq:imginp:spatialdecay}
\abs[\Big]{\sp[\Big]{\model_j}{\unish[\ver]{\aj}{j}{l}{k}}} \leq C_N 2^{(5-\aj)j/2} \integ[\R]{\tilde\weight_{N,j}(2^{-\aj j}(x_1+k_1)) \Schwartzpoly{x_1}^{-N} \Schwartzpoly{l x_1 + lk_1 - k_2)}^{-N}}{dx_1},
\end{equation}
where $\tilde\weight_{N,j} := \abs{\weight} \conv \Schwartzpoly{2^{2j} \dotarg}^{-N}$.
\end{lemma}

Now, we are able to give a proof of the main result of this section:
\begin{proof}[Proof of Proposition~\ref{prop:imginp:clusteredsparse}]
\begin{proofsteps}{Step}
\item\label{prop:imginp:clusteredsparse:fourier}
	We start with an estimation of $T_2^0$.
	For this, we fix $j \geq 0$ and $\abs{l} > 1$ first, and consider some index $(j, l, k; \aj, \ver) \in \Unishgroup_j$.
	Then, for $N \geq 2$, we have
	\begin{align}
		&\sum_{\substack{k \in \Z^2,\\ \mathclap{\translind{t_1}{\ver}\neq0,\translind{t_2}{\ver}\neq0}}} \abs[auto]{\translind{t_1}{\ver}}^{-N} \abs[auto]{\translind{t_2}{\ver}}^{-N}
		= \sum_{\substack{k \in \Z^2,\\ \mathclap{k_1 \neq 0,k_2 \neq l k_1}}} \abs{2^{-\aj j}k_1}^{-N} \abs{2^{-2 j}(k_2 - l k_1)}^{-N} \\
		={} & 2^{(\aj N + 2 N) j}\sum_{\substack{k \in \Z^2, \\ \mathclap{k_1\neq0, k_2 \neq 0}}} \abs{k_1}^{-N} \abs{k_2}^{-N} \leq C_{N} 2^{N(\aj + 2) j} \integ[\abs{x_1} \geq 1,\abs{x_2}\geq 1]{\abs{x_1}^{-N} \abs{x_2}^{-N}}{dx} \leq C_{N} 2^{N(\aj + 2) j}.
	\end{align}
	Similar estimates for the cases $\translind{t_1}{\ver} = 0$ and/or $\translind{t_2}{\ver} = 0$ yield altogether
	\begin{equation}
		\sum_{k \in \Z^2} \abs[auto]{\translind{t_1}{\ver}}^{-\translind{N}{\ver}} \abs[auto]{\translind{t_2}{\ver}}^{-\translind{N}{\ver}} \leq C_{N} 2^{\translind{N}{\ver}(\aj + 2 ) j}.
	\end{equation}
	Finally, by Lemma~\ref{lem:imginp:coeffdecay}\ref{lem:imginp:coeffdecay:ver},
	\begin{align}
		T_2^0
		&= \sum_{\substack{k \in \Z^2,\\ \mathclap{1 < \abs{l} \leq 2^{(2-\aj)j}}}} \abs[\Big]{\sp[\Big]{\model_j}{\unish[\ver]{\aj}{j}{l}{k}}} \leq C_{N,M} \sum_{\substack{k \in \Z^2,\\ \mathclap{1 < \abs{l} \leq 2^{(2-\aj)j}}}} \abs[auto]{\translind{t_1}{\ver}}^{-\translind{N}{\ver}} \abs[auto]{\translind{t_2}{\ver}}^{-\translind{N}{\ver}} 2^{(2-\aj)j/2} 2^{-\translind{N}{\ver}\aj j} \Schwartzpoly{2^{\aj j}}^{-M} \\
		&\leq C_{N,M} 2^{3(2-\aj)j/2} 2^{-\translind{N}{\ver}\aj j} \Schwartzpoly{2^{\aj j}}^{-M} \sum_{k \in \Z^2} \abs[auto]{\translind{t_1}{\ver}}^{-\translind{N}{\ver}} \abs[auto]{\translind{t_2}{\ver}}^{-\translind{N}{\ver}} \\
		&\leq C_{N,M} 2^{3(2-\aj)j/2} 2^{2 \translind{N}{\ver} j} \Schwartzpoly{2^{\aj j}}^{-M}.
	\end{align}
	Using the assumption $\liminf_{j \to \infty} \aj > 0$ and choosing $M$ sufficiently large, we obtain the desired decay rates of \eqref{eq:imginp:clusteredsparse}.
	The estimates for $T_3^0$ and $T_4^0$ are done analogously and are therefore omitted at this point.
\item\label{prop:imginp:clusteredsparse:spatial}
	To show the respective decay rate for $T_1^0$, we are going to apply Lemma~\ref{lem:imginp:spatialdecay} for $N \geq 2$ to obtain
	\begin{align}
		T_1^0 &= \sum_{\substack{k \in \Z^2,\abs{l}\leq 1\\ \abs{k_2 - lk_1}>2^{\eps j}}} \abs[\Big]{\sp[\Big]{\model_j}{\unish[\ver]{\aj}{j}{l}{k}}} \\
		&\leq  C_N 2^{(5-\aj)j/2} \sum_{\substack{k \in \Z^2,\abs{l}\leq 1\\
		\abs{k_2 - lk_1}>2^{\eps j}}} \integ[\R]{\tilde\weight_{N,j}(2^{-\aj j}(x_1+k_1)) \Schwartzpoly{x_1}^{-N} \Schwartzpoly{l x_1 + lk_1 - k_2}^{-N}}{dx_1}\\
		&=  C_N 2^{(5-\aj)j/2} \sum_{\substack{k \in \Z^2,\abs{l}\leq 1\\
		\abs{k_2}>2^{\eps j}}} \integ[\R]{\tilde\weight_{N,j}(2^{-\aj j}(x_1+k_1)) \Schwartzpoly{x_1}^{-N} \Schwartzpoly{l x_1 - k_2}^{-N}}{dx_1}\\
		&=  C_N 2^{(5-\aj)j/2} \sum_{\substack{\abs{l}\leq 1 \\ k_2 \in \Z, \abs{k_2}>2^{\eps j}}} \integ[\R]{\Bigg( \sum_{k_1 \in \Z} \tilde\weight_{N,j}(2^{-\aj j}(x_1+k_1)) \Bigg) \Schwartzpoly{x_1}^{-N} \Schwartzpoly{l x_1 - k_2}^{-N}}{dx_1}.
	\end{align}
	The first factor of the integrand can be estimated by
	\begin{align}
		\sum_{k_1 \in \Z} \tilde\weight_{N,j}(2^{-\aj j}(x_1+k_1)) &= \sum_{k_1 \in \Z} \integ[\R]{\abs{\weight(y_1)} \Schwartzpoly{2^{2j}(y_1-(2^{-\aj j}(x_1+k_1)))}^{-N}}{dy_1} \\
		&=\sum_{k_1 \in \Z} \integ[\R]{\abs{\weight(y_1)} \Schwartzpoly{2^{(2-\aj)j}(k_1 + x_1 - 2^{\aj j}y_1)}^{-N}}{dy_1} \\
		&\leq  \integ[\R]{\abs{\weight(y_1)} \underbrace{\Bigg(\sum_{k_1 \in \Z} \Schwartzpoly{k_1 + x_1 - 2^{\aj j}y_1}^{-N}\Bigg)}_{\leq C_N}}{dy_1} \leq C_N.
	\end{align}
	For the last inequality recall that $2^{(2-\aj)j} \geq 1$. Finally, by a similar argumentation as in \eqref{eq:proofs:spatialdecay:splitargument}, we obtain
	\begin{align}
		T_1^0 &\leq C_N 2^{(5-\aj)j/2} \sum_{\substack{\abs{l}\leq 1 \\ \mathclap{ k_2 \in \Z, \abs{k_2}>2^{\eps j}}}} \integ[\R]{\Schwartzpoly{x_1}^{-N}
        \Schwartzpoly{l x_1 - k_2)}^{-N}}{dx_1} \\
		&\leq C_N 2^{(5-\aj)j/2} \sum_{k_2 \in \Z, \abs{k_2}>2^{\eps j}} \Schwartzpoly{k_2}^{-N} \leq C_N 2^{(5-\aj)j/2} \integ[\abs{x_2} > 2^{\eps j}]{\Schwartzpoly{x_2}^{-N}}{dx_2} \\
		&\leq C_N 2^{(5-\aj)j/2} 2^{-(N-1)\eps j}.
	\end{align}
	This proves the claim since we can choose $N$ arbitrarily large. \qedhere
\end{proofsteps}
\end{proof}

The necessity of the assumption $\liminf_j \aj > 0$ in Proposition~\ref{prop:imginp:clusteredsparse} becomes clear within \ref{prop:imginp:clusteredsparse:fourier} of the proof; the factor $\Schwartzpoly{2^{\aj j}}^{-M}$ would not produce the required decay otherwise.
Moreover, the calculations of \ref{prop:imginp:clusteredsparse:spatial} emphasize the importance of the condition $\abs{k_2 - l k_1} \leq 2^{\eps j}$ in the definition of $\cluster_j$: To obtain a sufficient coefficient decay, we require a certain distance between the model singularity and the essential shearlet support.
However, one should be aware of the fact that---although rapid decay is achieved for arbitrary small $\eps > 0$---the constants in \eqref{eq:imginp:clusteredsparse} will blow up as $\eps$ approaches $0$.

The following corollary, clarifies the behavior of the normalization factor in \eqref{eq:imginp:asympconv}:
\begin{corollary}\label{cor:imginp:analysisnorm}
There exists an $N \geq 0$ such that $\anorm{\model_j}{\Unishshort} \in \asympffaster{2^{Nj}}$ as $j \to \infty$.
\end{corollary}
\begin{proof}
Due to Proposition~\ref{prop:imginp:clusteredsparse} is sufficient to investigate what happens ``inside'' of $\scalpm{\cluster}_j$.
For this, one can make use of Lemma~\ref{lem:imginp:spatialdecay} again, and proceed similar to \ref{prop:imginp:clusteredsparse:spatial} in the proof of Proposition~\ref{prop:imginp:clusteredsparse}.\qedhere
\end{proof}

\subsection{Cluster Coherence}

We now analyze the cluster coherence.
At this point, the condition on the gap size \ref{item:imginp:asympconv:gaps}, which has not occurred yet, comes into play.
Conversely, the image models $\model_j$ are only implicitly involved, namely by the geometric shape of the clusters $\scalpm{\cluster_j}$.

\begin{proposition}
\label{prop:imginp:clustercoh}
Assuming that $(\mdiam_j)_j \in \asympffaster{2^{-(\aj + \eps) j}}$, we have (recall that $\Unishshort = \Unish$)
\begin{equation}
	\clustercoh{\scalpm{\cluster_j}}{\indset{\mask{\mdiam_j}} \Unishshort} \to 0, \quad j \to \infty.
\end{equation}
\end{proposition}

\begin{proof}
We first separate the neighboring scales,
\begin{equation} \label{eq:imginp:clustercoh:scales}
	\clustercoh{\scalpm{\cluster_j}}{\indset{\mask{\mdiam_j}} \Unishshort}
	\leq \clustercoh{\cluster_{j-1}}{\indset{\mask{\mdiam_j}} \Unishshort} + \clustercoh{\cluster_{j}}{\indset{\mask{\mdiam_j}} \Unishshort} + \clustercoh{\cluster_{j+1}}{\indset{\mask{\mdiam_j}} \Unishshort}.
\end{equation}
In the following, we will only consider the main-scale-term (the respective estimates of neighboring-scale-terms will then follow from slight modifications):
\begin{align}
	& \mathrel{\phantom{\leq}} \clustercoh{\cluster_{j}}{\indset{\mask{\mdiam_j}} \Unishshort}
	= \max_{\Unishind_2 \in \Unishgroup} \sum_{\Unishind_1 \in \cluster_j} \abs{\sp{\indset{\mask{\mdiam_j}} \unishplain_{\Unishind_1}}{\unishplain_{\Unishind_2}}} \\
	={} & \underbrace{\max_{\substack{\Unishind_2 \in \Unishgroup,\\ \dir=\ver}} \sum_{\Unishind_1 \in \cluster_j} \abs{\sp{\indset{\mask{\mdiam_j}} \unishplain_{\Unishind_1}}{\unishplain_{\Unishind_2}}}}_{=: T_\ver}
	{}+{} \underbrace{\max_{\substack{\Unishind_2 \in \Unishgroup,\\ \dir=\hor}} \sum_{\Unishind_1 \in \cluster_j} \abs{\sp{\indset{\mask{\mdiam_j}} \unishplain_{\Unishind_1}}{\unishplain_{\Unishind_2}}}}_{=: T_\hor}
	{}+{} \underbrace{\max_{\substack{\Unishind_2 \in \Unishgroup,\\ \dir=\emptyset}} \sum_{\Unishind_1 \in \cluster_j} \abs{\sp{\indset{\mask{\mdiam_j}} \unishplain_{\Unishind_1}}{\unishplain_{\Unishind_2}}}}_{=: T_\emptyset}.
\end{align}
To derive a better intuition for the following computations, it might be helpful to consider the visualizations of Figure~\ref{fig:imginp:clustercoh}.
\begin{proofsteps}{Case}
\item \label{prop:imginp:clustercoh:ver}
	We start by estimating $T_\ver$.
	For this, let $\Unishind_1 = (j, l, k; \aj, \ver) \in \cluster_j$ and $\Unishind_2 = (j, l', k'; \aj, \ver) \in \Unishgroup_j$.
	Using the rapid decay properties of the universal shearlets, we obtain ($N \geq 2$)
	\begin{align}
		& \abs[\Big]{\sp[\Big]{\indset{\mask{\mdiam_j}} \unish[\ver]{\aj}{j}{l}{k}}{\unish[\ver]{\aj}{j}{l'}{k'}}}
		\leq \integ[-\mdiam_j][\mdiam_j]{\integ[\R]{\abs[auto]{\unish[\ver]{\aj}{j}{l}{k}(x)} \abs[auto]{ \unish[\ver]{\aj}{j}{l'}{k'}(x)}}{dx_2}}{dx_1} \\
		\leq{} & C_N \integ[-\mdiam_j][\mdiam_j]{\integ[\R]{2^{(2+\aj)j} \Schwartzpoly{\shearcone{\ver}^l \pscalcone{\aj}{\ver}^j x - k }^{-N} \Schwartzpoly{\shearcone{\ver}^{l'} \pscalcone{\aj}{\ver}^j x - k'}^{-N}}{dx_2}}{dx_1} \\
		\leq{} & \begin{aligned}[t]
			C_N \integ[-\mdiam_j][\mdiam_j]{\integ[\R]{& 2^{(2+\aj)j} \Schwartzpoly{2^{\aj j}x_1 - k_1}^{-N} \Schwartzpoly{l 2^{\aj j} x_1 +2^{2 j} x_2 - k_2}^{-N} \\
				& \times \Schwartzpoly{2^{\aj j}x_1 - k'_1}^{-N} \Schwartzpoly{l' 2^{\aj j} x_1 + 2^{2 j} x_2 - k'_2}^{-N} }{dx_2}}{dx_1}.
		\end{aligned} \\
		={} & C_N \integ[-2^{\aj j}\mdiam_j][2^{\aj j}\mdiam_j]{\integ[\R]{\Schwartzpoly{x_1 - k_1}^{-N} \Schwartzpoly{l x_1 +x_2 - k_2}^{-N} \underbrace{\Schwartzpoly{x_1 - k'_1}^{-N}}_{\leq 1} \Schwartzpoly{l' x_1 + x_2 - k'_2}^{-N} }{dx_2}}{dx_1}.
	\end{align}
	Without loss of generality, we may assume that the maximum in $T_\ver$ is attained for some $\Unishind_2 \in \Unishgroup_j$.
	This leads to
	\begin{align}
		& T_\ver
		= \max_{\substack{\Unishind_2 \in \Unishgroup,\\ \dir=\ver}} \sum_{\Unishind_1 \in \cluster_j} \abs{\sp{\indset{\mask{\mdiam_j}} \unishplain_{\Unishind_1}}{\unishplain_{\Unishind_2}}}
		\leq \max_{(j,l',k';\aj,\ver) \in \Unishgroup_j} \sum_{\abs{l} \leq1} \sum_{k \in \Z^2} \abs[\Big]{\sp[\Big]{\indset{\mask{\mdiam_j}} \unish[\ver]{\aj}{j}{l}{k}}{\unish[\ver]{\aj}{j}{l'}{k'}}} \\
		\leq{} & C_N \max_{(j,l',k';\aj,\ver) \in \Unishgroup_j} \sum_{\substack{\abs{l} \leq1 \\ k \in \Z^2}} \integ[-2^{\aj j}\mdiam_j][2^{\aj j}\mdiam_j]{\integ[\R]{ \Schwartzpoly{x_1 - k_1}^{-N} \Schwartzpoly{l x_1 +x_2 - k_2}^{-N} \Schwartzpoly{l' x_1 + x_2 - k'_2}^{-N} }{dx_2}}{dx_1} \\
		\leq{} & C_N \max_{(j,l',k';\aj,\ver) \in \Unishgroup_j} \sum_{\abs{l}\leq1} \begin{aligned}[t]
			\integ[-2^{\aj j}\mdiam_j][2^{\aj j}\mdiam_j]{\integ[\R]{& \Schwartzpoly{l' x_1 + x_2 - k'_2}^{-N} \\
				& \times \underbrace{\Bigg( \sum_{k \in \Z^2}\Schwartzpoly{x_1 - k_1}^{-N} \Schwartzpoly{l x_1 +x_2 - k_2}^{-N} \Bigg)}_{\leq \tilde C_N} }{dx_2}}{dx_1}
		\end{aligned} \\
		\leq{} & C_N \max_{(j,l',k';\aj,\ver) \in \Unishgroup_j} \integ[-2^{\aj j}\mdiam_j][2^{\aj j}\mdiam_j]{\underbrace{\integ[\R]{ \Schwartzpoly{x_2 + l' x_1 - k'_1}^{-N} }{dx_2}}_{\leq \tilde C_N}}{dx_1} \leq C_N 2^{\aj j} h_j \to 0, \quad j \to \infty.
	\end{align}
\item \label{prop:imginp:clustercoh:hor}
	For the term $T_\hor$, let $\Unishind_1 = (j, l, k; \aj, \ver) \in \cluster_j$ and $\Unishind_2 = (j, l', k'; \aj, \hor) \in \Unishgroup_j$.
	For $N \geq 2$ we then have
	\begin{align}
		& \abs[\Big]{\sp[\Big]{\indset{\mask{\mdiam_j}} \unish[\ver]{\aj}{j}{l}{k}}{\unish[\hor]{\aj}{j}{l'}{k'}}}
		\leq \integ[-\mdiam_j][\mdiam_j]{\integ[\R]{\abs[auto]{\unish[\ver]{\aj}{j}{l}{k}(x)} \abs[auto]{\unish[\hor]{\aj}{j}{l'}{k'}(x)}}{dx_2}}{dx_1} \\
		\leq{} & C_N \integ[-\mdiam_j][\mdiam_j]{\integ[\R]{2^{(2+\aj)j} \Schwartzpoly{\shearcone{\ver}^l \pscalcone{\aj}{\ver}^j x - k}^{-N} \Schwartzpoly{\shearcone{\hor}^{l'} \pscalcone{\aj}{\hor}^j x - k'}^{-N}}{dx_2}}{dx_1} \\
		\leq{} &\begin{aligned}[t]
			C_N 2^{(2+\aj)j} & \integ[-\mdiam_j][\mdiam_j]{ \integ[\R]{\Schwartzpoly{2^{\aj j} x_1 - k_1}^{-N} \Schwartzpoly{l2^{\aj j} x_1 +2^{2j}x_2 - k_2}^{-N} \\
			& \times \Schwartzpoly{l'2^{\aj j} x_2 +2^{2j}x_1 - k'_1}^{-N} \Schwartzpoly{2^{\aj j} x_2 - k'_2}^{-N}}{dx_2}}{dx_1}
		\end{aligned} \\
		\leq{} & \begin{aligned}[t]
			C_N 2^{\aj j} & \integ[-\mdiam_j][\mdiam_j]{\integ[\R]{\Schwartzpoly{2^{\aj j} x_1 - k_1}^{-N} \Schwartzpoly{l2^{\aj j} x_1 +x_2 - k_2}^{-N} \\
			& \times \underbrace{ \Schwartzpoly{l'2^{(\aj-2)j} x_2 +2^{2j}x_1 - k'_1}^{-N} \Schwartzpoly{2^{(\aj-2) j} x_2 - k'_2}^{-N}}_{\leq 1}}{dx_2}}{dx_1}
		\end{aligned} \\
		\leq{} & C_N 2^{\aj j} \integ[-\mdiam_j][\mdiam_j]{ \Schwartzpoly{2^{\aj j} x_1 - k_1}^{-N}}{dx_1} \underbrace{\integ[\R]{\Schwartzpoly{x_2 + l2^{\aj j} x_1 - k_2}^{-N}}{dx_2}}_{\leq \tilde C_N}.
	\end{align}
	Summarizing, we obtain
	\begin{align}
		T_\hor ={}& \max_{\substack{\Unishind_2 \in \Unishgroup,\\ \dir=\hor}} \sum_{\Unishind_1 \in \cluster_j} \abs{\sp{\indset{\mask{\mdiam_j}} \unishplain_{\Unishind_1}}{\unishplain_{\Unishind_2}}}
		\leq \max_{(j,l',k';\aj,\hor) \in \Unishgroup_j} \sum_{\abs{l} \leq1} \sum_{\substack{k \in \Z^2,\\ \abs{k_2-lk_1}\leq2^{\eps j}}} \abs[\Big]{\sp[\Big]{\indset{\mask{\mdiam_j}} \unish[\ver]{\aj}{j}{l}{k}}{\unish[\hor]{\aj}{j}{l'}{k'}}} \\
		\leq{} & C_N 2^{\aj j} \max_{(j,l',k';\aj,\hor) \in \Unishgroup_j} 2^{\eps j} \sum_{k_1 \in \Z} \integ[-\mdiam_j][\mdiam_j]{ \Schwartzpoly{2^{\aj j} x_1 - k_1}^{-N}}{dx_1} \\
		={} & C_N 2^{\eps j} 2^{\aj j} \max_{(j,l',k';\aj,\hor) \in \Unishgroup_j} \integ[-\mdiam_j][\mdiam_j]{ \underbrace{\sum_{k_1 \in \Z} \Schwartzpoly{2^{\aj j} x_1 - k_1}^{-N}}_{\leq \tilde C_N}}{dx_1} \leq C_N 2^{(\aj + \eps) j} h_j \to 0, \quad j \to \infty.
	\end{align}
\item
	Since boundary shearlets are a combination of the horizontal and vertical elements, it is easy to show that there exists a constant $C > 0$ such that $T_\emptyset \leq C (T_\hor + T_\ver)$.
	Therefore, the claim follows from the first two cases.\qedhere
\end{proofsteps}
\end{proof}

\begin{figure}
	\centering
	\subfigure[{\ref{prop:imginp:clustercoh:ver} corresp. to $T_\ver$}]{
		\centering
		\includegraphics[width=.25\textwidth]{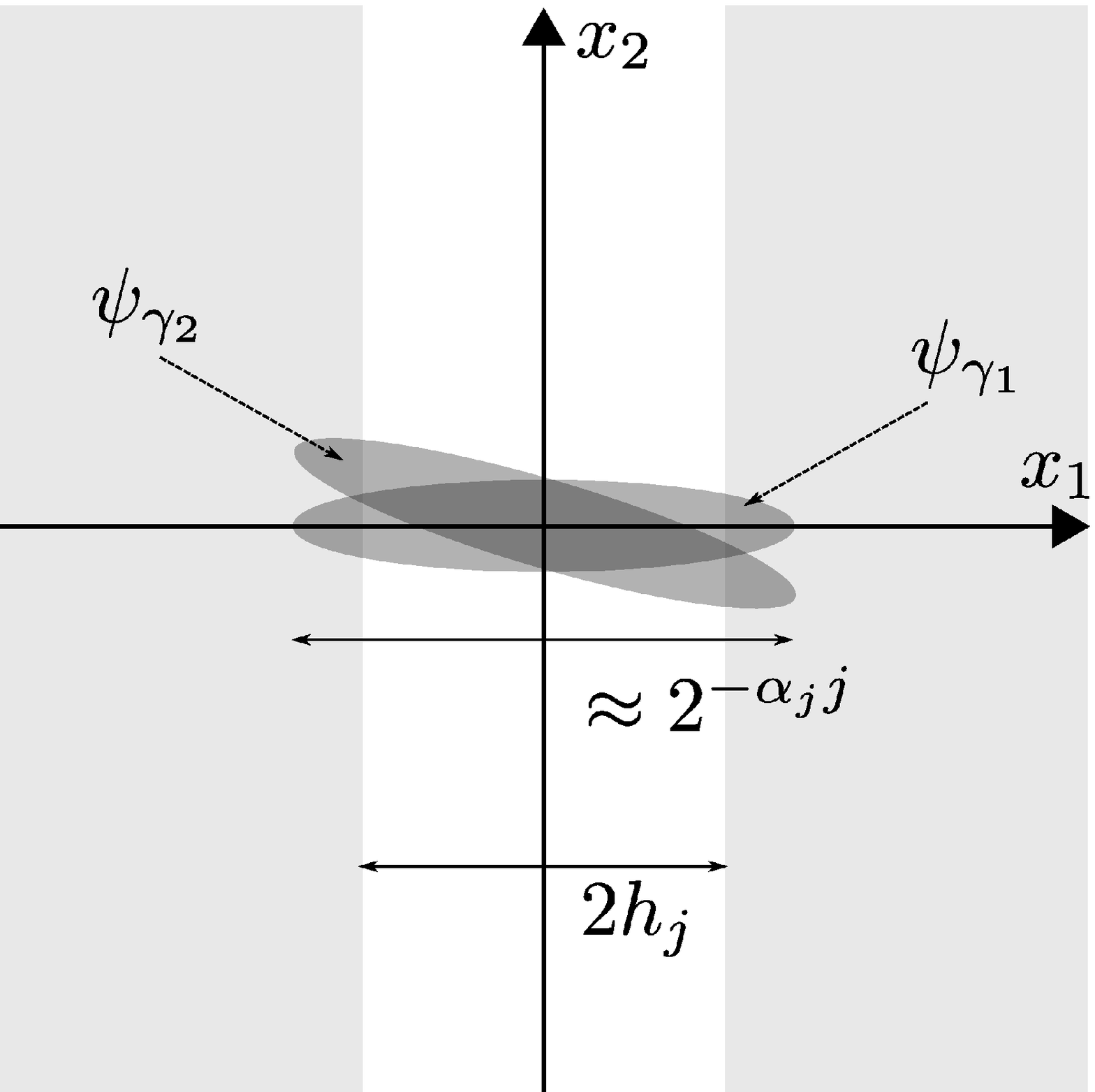}
		\label{fig:imginp:clustercoh:ver}
	}%
	\qquad
	\subfigure[{\ref{prop:imginp:clustercoh:hor} corresp. to $T_\hor$}]{
		\centering
		\includegraphics[width=.25\textwidth]{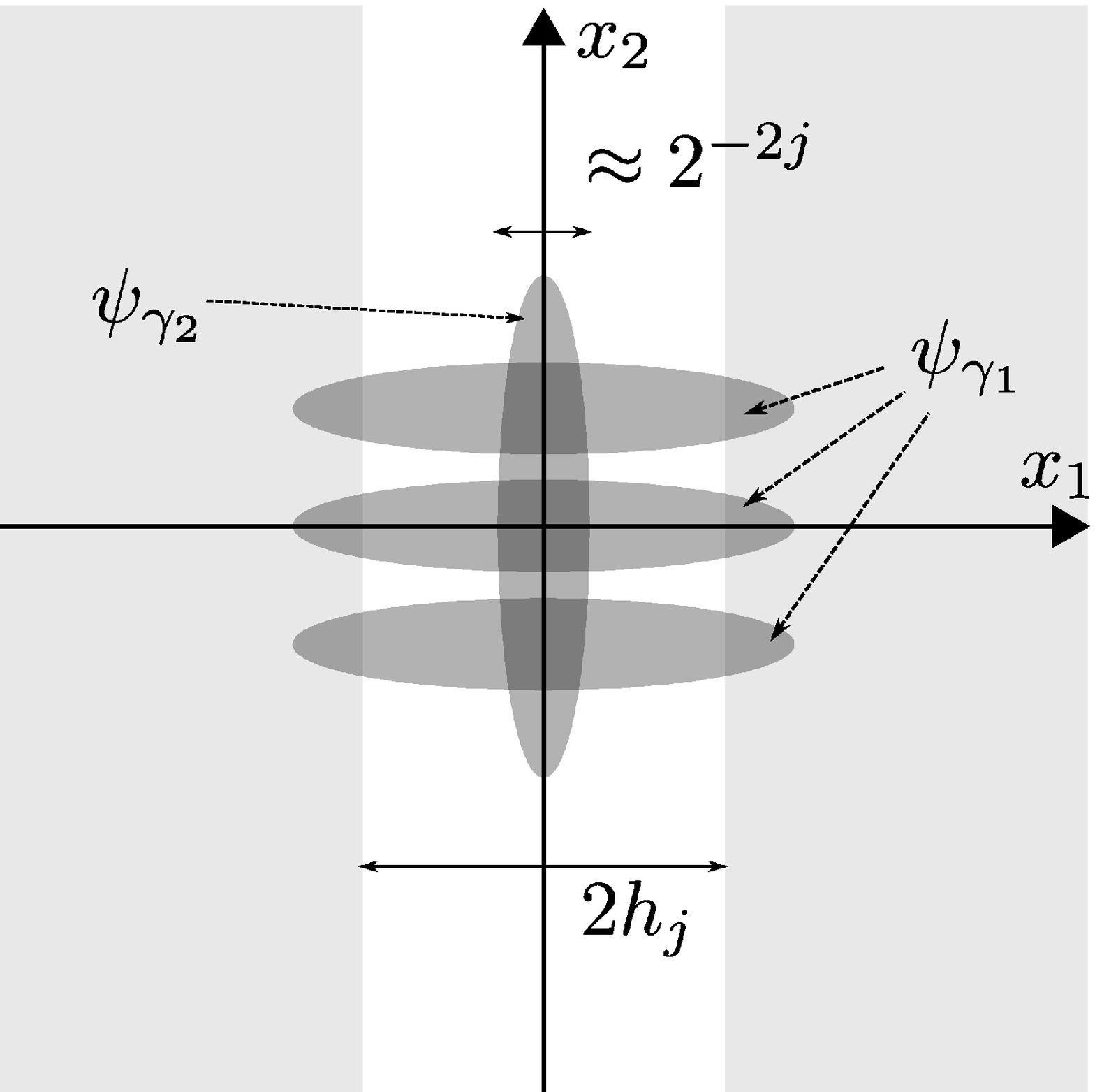}
		\label{fig:imginp:clustercoh:hor}
	}%
	\caption{This illustration shows how strong the essential shearlet supports may overlap (at worst) in the proof of Proposition~\ref{prop:imginp:clustercoh}.}
	\label{fig:imginp:clustercoh}
\end{figure}

Intuitively, the scalar products $\abs{\sp{\indset{\mask{\mdiam_j}} \unishplain_{\Unishind_1}}{\unishplain_{\Unishind_2}}}$ measure how strong the corresponding frame elements overlap (or interact) inside of $\mask{\mdiam_j}$ (cf. Figure~\ref{fig:imginp:clustercoh}).
The definition of $\cluster_j$ forces $\unishplain_{\Unishind_1}$ to be almost horizontally oriented, whereas $\unishplain_{\Unishind_2}$ can have arbitrary direction.

Also notice that on the one hand, in \ref{prop:imginp:clustercoh:ver} of the proof, it was sufficient to assume that $\mdiam_j$ decays faster than the essential length of the shearlets elements.
On the other hand, the essential support of the vertically oriented shearlets in \ref{prop:imginp:clustercoh:hor} could be entirely contained in $\mask{\mdiam_j}$ (cf. Figure~\ref{fig:imginp:clustercoh:hor}).
Hence, we shall not allow too many vertical translations in $\cluster_j$.
However, the ``tube width'' has to be at least $2^{\eps j}$ to guarantee sufficient sparsity (see final remark of previous subsection).
These observations show the fundamental trade-off relation between clustered sparsity and cluster coherence.
We have already described this interaction subsequent to Theorem~\ref{thm:framework:errorl1min}, and this can now be made precise for the specific setting of image inpainting.
In particular, we have derived an intuition of how the cluster coherence is related to the amount of missing information: In the proof of Proposition~\ref{prop:imginp:clustercoh}, we observed that $\clustercoh{\scalpm{\cluster_j}}{\indset{\mask{\mdiam_j}} \Unishshort}$ can be estimated by $2^{(\aj+\eps) j} \mdiam_j$ implying that the coherence will decrease whenever the gap width does.

\subsection{Optimal Gap Sizes}
\label{subsec:imginp:optimality}

As already mentioned at the beginning of this section, the condition \ref{item:imginp:asympconv:gaps} on the gap width in Theorem~\ref{thm:imginp:asympconv} is
only a sufficient criterion for asymptotically perfect inpainting. It is therefore an essential (and still open) question, whether this assumption is also necessary,
i.e., the choice of $\mdiam_j \approx 2^{-\aj j}$ would be \emph{optimal}. In fact, this would, for instance, formally prove that shearlets are superior to wavelets when recovering curvilinear structures and, more general, that directional scaling has an improving effect on the process of inpainting. Such a result however seems to be unreachable to us at the moment, in particular, because of the vastly unknown $\l{1}$-recoveries $\modelrec_j$. Without the introduction of novel analysis techniques, it appears almost impossible to find an appropriate lower estimate for $\anorm{\modelrec_j - \model_j}{\Unishshort}$.

In order to tackle this problem, one could use another inpainting algorithm that returns more explicit recoveries.
A very simple (but prominent) example is \emph{one-step thresholding} which was also analyzed in \cite{king2014analysis}.
In this method, the damaged signal $\ProK \model_j$ is first analyzed with $\Unishshort$ and small coefficients are truncated afterward.
A final synthesis yields the recovery
\begin{equation}\label{eq:imginp:thresholding}
	\modelrec_j := \OpSynthesis{\Unishshort} (\indcoeff{\cluster_{\beta_j}} \OpAnalysis{\Unishshort} (\ProK\model_j)),
\end{equation}
where $\cluster_{\beta_j} := \{ (j, l, k; \aj, \dir) \suchthat \abs{\sp{\ProK\model_j}{\unish[\ver]{\alpha}{j}{l}{k}}} \geq \beta_j \}$ for certain thresholds $\beta_j$.
Compared to $\l{1}$-minimization, $\modelrec_j$ is given by the explicit expression \eqref{eq:imginp:thresholding};
but conversely, we have virtually no information about the geometric form of the thresholding-clusters $\cluster_{\beta_j}$.
Due to the band-limiting of $\Unish$, the shearlet elements are strongly oscillating in spatial domain, which makes is extremely difficult to determine the large coefficients.
It is particularly impossible to choose the thresholds $\beta_j$ such that $\cluster_{\beta_j}$ has a ``tube shape'' as in Subsection~\ref{subsec:imginp:clusteredsparse}.
Although the drawbacks of simple thresholding and $\l{1}$-optimization are quite different---in both cases, a certain part of the recovery is only implicitly given---a statement on optimality seems to be currently out of our reach.

However, we strongly believe that at least a weaker version of $(\mdiam_j)_j \in \asympffaster{2^{-(\aj+\eps) j}}$ is necessary to achieve asymptotically perfect inpainting.
To give an intuitive argument for this conjecture, we may consider some cluster index $(j, l, k; \aj, \ver) \in \cluster_j$.
Recalling the notation $\translind{t}{\ver} = (2^{-\aj j} k_1, 2^{-2j}(k_2 - l k_1))$ from Subsection~\ref{subsec:imginp:clusteredsparse}, we observe that $\unish[\ver]{\alpha}{j}{l}{k}$ is ``concentrated'' inside of $\mask{\mdiam_j}$
provided that $\abs{k_1} \leq 2^{\aj j} \mdiam_j$. Hence---as a heuristic argument---the ``content'' of the missing part might be estimated by
\begin{align}
	\anorm{\ProM \model_j}{\Unishshort} &= \sum_{\unishplain \in \Unishshort} \abs{\sp{\ProM \model_j}{\unishplain}} \gtrsim \sum_{\substack{(j, l, k; \aj, \ver) \in \cluster_j \\ l=0,\abs{k_1} \leq 2^{\aj j} \mdiam_j}} \underbrace{\abs{\sp{\model_j}{\unish[\ver]{\alpha}{j}{l}{k}}}}_{\approx 2^{(5-\aj)j/2}} \\
	&\approx 2^{\eps j} 2^{\aj j} \mdiam_j 2^{(5-\aj)j/2} = 2^{5(1+2\eps)j/2} 2^{\aj j/2} \mdiam_j.
\end{align}
This at least indicates that excessively enlarging $\mdiam_j$ produces a lack of ``information'' and will cause recoveries deviating
strongly from the original image. Interestingly, a similar estimate for the energy of $\ProM \model_j$ gives
\begin{align}
	\Lebnorm{\ProM \model_j}^2 &= \sum_{\unishplain \in \Unishshort} \abs{\sp{\ProM \model_j}{\unishplain}}^2 \gtrsim \sum_{\substack{(j, l, k; \aj, \ver) \in \cluster_j \\ l=0,\abs{k_1} \leq 2^{\aj j} \mdiam_j}} \abs{\sp{\model_j}{\unish[\ver]{\alpha}{j}{l}{k}}}^2 \\
	&\approx 2^{\eps j} 2^{\aj j} \mdiam_j 2^{(5-\aj)j} = 2^{(5+\eps)j} \mdiam_j.
\end{align}
Since the term $2^{\aj j}$ cancels out here, our estimate from below does not involve the crucial feature of anisotropy.
Therefore, measuring the error in terms of energy seems to be inappropriate to prove an optimality result that compares shearlets and wavelets (see also Subsection~\ref{subsec:framework:errorestimate}).

\section{Extensions and Further Directions}
\label{sec:extension}

The two main Theorems~\ref{thm:shearlets:unishprop} and \ref{thm:imginp:asympconv} are amenable to several extensions.
In the following, we present a brief listing of various directions in which our modeling situation could be generalized to:

\begin{itemize}[leftmargin=3em]
\item
	\emph{$\alpha$-molecules:} We anticipate that the notion of universal shearlets, and in particular the idea of scaling sequences, can be extended to the setting of $\alpha$-molecules \cite{GKKS14}, which is a novel framework to unify numerous representation systems---including ridgelets, wavelets, shearlets, and curvelets---as well as their sparse approximation properties.
\item
	\emph{Non-asymptotic analysis:} The statement of Theorem~\ref{thm:imginp:asympconv} is of purely asymptotical nature and we did not investigate the decay constants in \eqref{eq:imginp:asympconv}.
	In particular, the constants will blow up when $\aj$ is closed to $0$.
	Therefore, the non-asymptotical case, i.e., when only a finite number of scales is considered, might be also interesting.
	This directly leads to the optimization problem of finding the ``best possible'' $\aj$ in the range of $0 \leq \aj \leq 2$ (the second bound corresponds to the wavelet case).
	In practice, this challenge could be tackled by a combination of $\l{1}$-minimization and an appropriate \emph{dictionary learning} algorithm such as K-SVD \cite{rubenstein2013analysis}, which tries to learn the optimal scaling sequence $(\aj)_j$.
\item
	\emph{Curvilinear singularities:} By the \emph{tubular neighborhood theorem}, it is possible---though
    very technical---to extend our results toward a curve model instead of the straight-line distribution $\model$. This can be done similarly to the approach in \cite{donoho2013microlocal}, where this technique
    was applied to the problem of geometric separation.
\item
	\emph{Other mask models:} Other models for the mask such as ball or box shaped structures can be imagined.
	The general strategy of our work should be a conceivable path for an analysis also in those cases, but certainly, the technical details will highly depend on the chosen model for the mask.
\item
	\emph{Other image models:} In this paper, we focused on the recovery of anisotropic features. However, one might also ask about inpainting of other structures such as texture, or even a superposition of different structural components. This certainly requires a careful adaption of the arguments, but we believe that our general strategy is applicable. In the case of texture, one could use for instance \emph{Gabor systems}; but in order to perform an asymptotic analysis, an additional parameter needs then to be inserted as it was done in \cite{Kut13}.
\item
	\emph{Higher dimensions:} The construction of Guo and Labate in \cite{guo2013construction} also contains higher
    dimensional versions of shearlets. The respective modifications of universal shearlets are straightforward, hence
    we decided to not include those. Of special interest is then the three-dimensional case, which corresponds to the
    problem of \emph{video-inpainting}.
\item
	\emph{Noise:} A very common question is to ask for the stability of a certain algorithm when the input signal is
    affected by noise. For the problem of (abstract) inpainting, theoretical results are proven in \cite{king2014analysis}
    and can be directly applied to our analysis, showing that it is indeed stable toward noise.
\end{itemize}

\section{Proofs}
\label{sec:proofs}

\subsection{Proof of Theorem \ref{thm:shearlets:unishprop}}
\label{subsec:proofs:unishprop}

\begin{proof}[\nopunct]\mbox{}

\emph{Band-limiting and vanishing moment property:}
	Observing that $(0, 0) \not\in \supp\Corofunc_j \subset \Coro_j$, this immediately follows from Definition~\ref{def:shearlets:unishdef}.

\emph{Smoothness:}
	Due to the band-limiting of $\Unish$, it suffices to show that the Fourier transforms of the elements are smooth.
	This then automatically implies that $\Unish$ is contained in $\Schwartz[\R^2]$.

	The smoothness of $\UnishLow$ and $\UnishInt$ are induced by their smooth defining functions $\meyerscal$ and $\conefunc$.
	Hence it remains to consider some element $\unish{\aj}{j}{l}{k} \in \UnishBound$. In the interior of $\Cone{\hor}$ and $\Cone{\ver}$, the smoothness is again obvious.
	Thus, we only need to analyze the boundary lines of the cones which are given by $\{\abs{\xi_1} = \abs{\xi_2}\}$.

	For this, let us use the shortcut $\lmax := 2^{(2-\aj) j}$ for the maximal shearing number (on scale $j$) and simplify the definition of $\unishft{\aj}{j}{\lmax}{k}$ (for $j \geq 1$):
	\begin{equation}
		\unishft{\aj}{j}{\lmax}{k} (\xi) = 2^{-\frac{(2+\aj)j}{2}-\frac{1}{2}} \cdot \begin{cases}
			\Corofunc( 2^{-2j}\xi )\conefunc\opleft( \lmax \left( \xi_2 / \xi_1 - 1 \right) \opright) e^{-\pi i [2^{-2j} \xi_1 k_1 + 2^{-\aj j}(\xi_2 - \xi_1)k_2]}, & \xi \in \Cone{\hor}, \\
			\Corofunc( 2^{-2j}\xi )\conefunc\opleft( \lmax \left( \xi_1 / \xi_2 - 1 \right) \opright) e^{-\pi i [2^{-2j} \xi_1 k_1 + 2^{-\aj j}(\xi_2 - \xi_1)k_2]}, & \xi \in \Cone{\ver}.
		\end{cases}
	\label{eq:proofs:unishprop:boundaryterm}
	\end{equation}
	In case of $\xi_1 = \pm\xi_2$, both terms coincide implying the continuity of $\unishft{\aj}{j}{l}{k}$.
	By the same argument, the continuity of $\unishft{\aj}{j}{-\lmax}{k}$ is verified.
	Next, we compute the partial derivatives of both terms in \eqref{eq:proofs:unishprop:boundaryterm} which are given by
	\begin{multline}
		\partder[\xi_1 = \xi_2]{}{\xi_1} \left[ \Corofunc( 2^{-2j}\xi )\conefunc\opleft( \lmax \left( \tfrac{\xi_2}{\xi_1} - 1 \right) \opright) e^{-\pi i [2^{-2j} \xi_1 k_1 + 2^{-\aj j}(\xi_2 - \xi_1)k_2]} \right] \\
		\shoveleft  = 2^{-2j} \partder{\Corofunc}{\xi_1}( 2^{-2j}\xi_1, 2^{-2j}\xi_1 ) \conefunc(0) e^{-2^{-2j} \pi i  \xi_1 k_1} - \frac{\lmax}{\xi_1} \Corofunc( 2^{-2j}\xi_1, 2^{-2j}\xi_1 ) \conefunc'(0) e^{-2^{-2j} \pi i  \xi_1 k_1} \\*
		- \pi i (2^{-2j} k_1 - 2^{-\aj j}k_2) \Corofunc( 2^{-2j}\xi_1, 2^{-2j}\xi_1 ) \conefunc(0) e^{-2^{-2j} \pi i  \xi_1 k_1}
	\label{eq:proofs:unishprop:partderhor}
	\end{multline}
	and
	\begin{multline}
		\partder[\xi_1 = \xi_2]{}{\xi_1} \left[ \Corofunc( 2^{-2j}\xi )\conefunc\opleft( \lmax \left( \tfrac{\xi_1}{\xi_2} - 1 \right) \opright) e^{-\pi i [2^{-2j} \xi_1 k_1 + 2^{-\aj j}(\xi_2 - \xi_1)k_2]} \right] \\
		\shoveleft  = 2^{-2j} \partder{\Corofunc}{\xi_1}( 2^{-2j}\xi_1, 2^{-2j}\xi_1 ) \conefunc(0) e^{-2^{-2j} \pi i  \xi_1 k_1} + \frac{\lmax}{\xi_1} \Corofunc( 2^{-2j}\xi_1, 2^{-2j}\xi_1 	) \conefunc'(0) e^{-2^{-2j} \pi i  \xi_1 k_1} \\*
		- \pi i (2^{-2j} k_1 - 2^{-\aj j}k_2) \Corofunc( 2^{-2j}\xi_1, 2^{-2j}\xi_1 ) \conefunc(0) e^{-2^{-2j} \pi i  \xi_1 k_1}.
	\label{eq:proofs:unishprop:partderver}
	\end{multline}
	These two expressions indeed coincide, since $\conefunc'(0) = 0$.
	The same can be done in a similar fashion for the partial derivative with respect to $\xi_2$, and by obvious modifications, one verifies the smoothness for $\unishft{\aj}{j}{-\lmax}{k}$ as well as for the case $j = 0$.
	Finally, the differentiability of higher order is proven by induction and successive use of \eqref{eq:shearlets:bump2}.
	
\emph{Parseval frame property:}
	In order to prove this property, the general strategy is to first make use of Parseval's identity to eliminate the translations (modulations in the Fourier representation,
    respectively), followed by application of the covering properties \eqref{eq:shearlets:corotiling} and \eqref{eq:shearlets:bump1}.

	For this, let $f \in \Leb[\R^2]{2}$ be arbitrary. We now consider the respective parts of $\Unish$ separately:

	\begin{proofsteps}{Case}
	\item\label{thm:shearlets:unishprop:bnd}
		Boundary shearlets with $j \geq 1$: By Plancherel's Theorem, we obtain
		\begin{multline}
			\sum_{k \in \Z^2} \abs{\sp{f}{\unish{\aj}{j}{\lmax}{k}}}^2
			= \sum_{k \in \Z^2} \abs{\sp{\ft f}{\unishft{\aj}{j}{\lmax}{k}}}^2 = \sum_{k \in \Z^2} \abs[\bigg]{ \integ[\R^2]{2^{-\frac{(2+\aj)j}{2}-\frac{1}{2}} \ft f(\xi) \Corofunc( 2^{-2j}\xi ) e^{\pi i \xi^\T \pscalcone{\aj}{\hor}^{-j} \shearcone{\hor}^{-\lmax} k} \\
			\times \Big[ \indset{\Cone{\hor}}(\xi) \hphantom{\cdot} \Conefunc{\hor}\opleft( \xi^\T \pscalcone{\aj}{\hor}^{-j} \shearcone{\hor}^{-\lmax} \opright)
			+ \indset{\Cone{\ver}}(\xi) \hphantom{\cdot} \Conefunc{\ver}\opleft( \xi^\T \pscalcone{\aj}{\ver}^{-j} \shearcone{\ver}^{-\lmax} \opright) \Big]}{d\xi} }^2.
		\label{eq:proofs:unishprop:boundaryint}
		\end{multline}
		To apply Parseval's identity, we perform a substitution
		\begin{equation}
			\eta^\T := 2^{-1} \xi^\T \pscalcone{\aj}{\hor}^{-j} \shearcone{\hor}^{-\lmax}
			\quad\iff\quad \xi = \xi(\eta) = (2^{2j+1}\eta_1, 2^{2j+1}\eta_1 + 2^{\aj j + 1} \eta_2).
		\end{equation}
		The single expressions of \eqref{eq:proofs:unishprop:boundaryint} then have the form
		\begin{align}
			\Conefunc{\hor}\opleft( \xi^\T \pscalcone{\aj}{\hor}^{-j} \shearcone{\hor}^{-\lmax} \opright) &= \conefunc\opleft( \tfrac{\eta_2}{\eta_1} \opright), \\
			\Conefunc{\ver}\opleft( \xi^\T \pscalcone{\aj}{\ver}^{-j} \shearcone{\ver}^{-\lmax} \opright) &= \conefunc\opleft( 2^{(2-\aj)j} \left( \xi_1 / \xi_2 - 1 \right) \opright) = \conefunc\opleft( - \tfrac{\eta_2}{\eta_1 + 2^{(\aj-2)j} \eta_2} \opright), \\
			\Corofunc( 2^{-2j}\xi ) &= \Corofunc\opleft( 2\eta_1, 2[\eta_1 + 2^{(\aj-2)j} \eta_2] \opright).
		\end{align}
		By \eqref{eq:shearlets:corosupp}, the mapping $(\eta_1, \eta_2) \mapsto \Corofunc( 2\eta_1, 2[\eta_1 + 2^{(\aj-2)j} \eta_2])$ is supported in the ``strip'' $\{\abs{\eta_1} \leq 1/4\}$.
		Since $\supp \conefunc \subset \intvcl{-1}{1}$, we can conclude that the mappings
		\begin{align}
			(\eta_1, \eta_2) \mapsto U_{(\hor),j}(\eta) &:= \Corofunc\opleft( 2\eta_1, 2[\eta_1 + 2^{(\aj-2)j} \eta_2] \opright) \conefunc\opleft( \tfrac{\eta_2}{\eta_1} \opright), \\
			(\eta_1, \eta_2) \mapsto U_{(\ver),j}(\eta) &:= \Corofunc\opleft( 2\eta_1, 2[\eta_1 + 2^{(\aj-2)j} \eta_2] \opright) \conefunc\opleft( - \tfrac{\eta_2}{\eta_1 + 2^{(\aj-2)j} \eta_2} \opright)
		\end{align}
		are both supported in the square $Q^2 := \intvcl{-\frac{1}{2}}{\frac{1}{2}}^2$. Here, we used
		\begin{align}
			\abs[auto]{\tfrac{-\eta_2}{\eta_1 + 2^{(\aj-2)j} \eta_2}} \leq 1 \ \implies \ \abs[auto]{\tfrac{\eta_2}{\eta_1}} \leq \abs[auto]{1 + 2^{(\aj-2)j} \tfrac{\eta_2}{\eta_1}} \leq 1 + 2^{(\aj-2)j} \abs[auto]{\tfrac{\eta_2}{\eta_1}} \ \implies \ \abs[auto]{\tfrac{\eta_2}{\eta_1}} \leq \tfrac{1}{1 - 2^{(\aj-2)j}} \leq 2,
		\end{align}
		where in the last estimate is due to $\aj \leq 2 - 1 / j$.
		With this observation, we can proceed in \eqref{eq:proofs:unishprop:boundaryint} by
		\begin{align}
			& \sum_{k \in \Z^2} \abs{\sp{f}{\unish{\aj}{j}{\lmax}{k}}}^2 \\
			={} & \sum_{k \in \Z^2} \abs[\bigg]{ \integ[Q^2]{2^{\frac{(2+\aj)j}{2}+\frac{1}{2}} \ft f\opleft(\xi(\eta)\opright) \Big[\indset{\Cone{\hor}}\opleft(\xi(\eta)\opright) \hphantom{\cdot} U_{(\hor),j}(\eta) + \indset{\Cone{\ver}}\opleft(\xi(\eta)\opright) \hphantom{\cdot} U_{(\ver),j}(\eta) \Big] e^{2\pi i \eta^\T k}}{d\eta} }^2 \\
			={} & \integ[Q^2]{2^{(2+\aj)j+1} \abs{\ft f\opleft(\xi(\eta)\opright)}^2 \abs[\Big]{\indset{\Cone{\hor}}\opleft(\xi(\eta)\opright) \hphantom{\cdot} U_{(\hor),j}(\eta) + \indset{\Cone{\ver}}\opleft(\xi(\eta)\opright) \hphantom{\cdot} U_{(\ver),j}(\eta) }^2}{d\eta} \\
			={} & \integ[\Cone{\hor}]{ \abs{\ft f(\xi)}^2 \abs{\Corofunc( 2^{-2j}\xi )}^2 \abs{\conefunc\opleft( \lmax \left( \xi_2 / \xi_1 - 1 \right) \opright)}^2}{d\xi} + \integ[\Cone{\ver}]{ \abs{\ft f(\xi)}^2 \abs{\Corofunc( 2^{-2j}\xi )}^2 \abs{\conefunc\opleft( \lmax \left( \xi_1 / \xi_2 - 1 \right) \opright)}^2}{d\xi}.
		\end{align}
		A very similar computation gives
		\begin{align}
			\sum_{k \in \Z^2} \abs{\sp{f}{\unish{\aj}{j}{-\lmax}{k}}}^2
			& = \begin{aligned}[t]
			 &\integ[\Cone{\hor}]{ \abs{\ft f(\xi)}^2 \abs{\Corofunc( 2^{-2j}\xi )}^2 \abs{\conefunc\opleft( \lmax \left( \xi_2 / \xi_1 + 1 \right) \opright)}^2}{d\xi} \\
			 & + \integ[\Cone{\ver}]{ \abs{\ft f(\xi)}^2 \abs{\Corofunc( 2^{-2j}\xi )}^2 \abs{\conefunc\opleft( \lmax \left( \xi_1 / \xi_2 + 1 \right) \opright)}^2}{d\xi}.
			 \end{aligned}
		\end{align}
	\item\label{thm:shearlets:unishprop:bndj0}
		Boundary shearlets with $j = 0$:
		Observing that $\supp \Corofunc \subset Q^2$, we obtain
		\begin{align}
			& \mathrel{\phantom{=}} \sum_{k \in \Z^2} \abs{\sp{f}{\unish{\aj}{0}{\pm1}{k}}}^2
			= \sum_{k \in \Z^2} \abs{\sp{\ft f}{\unishft{\aj}{0}{\pm1}{k}}}^2 \\
			& = \sum_{k \in \Z^2} \abs[\bigg]{ \integ[Q^2]{\ft f(\xi) \Corofunc(\xi)  \Big[\indset{\Cone{\hor}}(\xi) \conefunc\opleft(\tfrac{\xi_2}{\xi_1} \mp 1 \opright) + \indset{\Cone{\ver}}(\xi) \conefunc\opleft(\tfrac{\xi_1}{\xi_2} \mp 1 \opright) \Big] e^{2\pi i \xi^\T k}}{d\xi} }^2 \\
			& = \integ[\Cone{\hor}]{\abs{\ft f(\xi)}^2 \abs{\Corofunc(\xi)}^2 \abs[auto]{\conefunc\opleft(\tfrac{\xi_2}{\xi_1} \mp 1 \opright)}^2 }{d\xi} + \integ[\Cone{\ver}]{\abs{\ft f(\xi)}^2 \abs{\Corofunc(\xi)}^2 \abs[auto]{\conefunc\opleft(\tfrac{\xi_1}{\xi_2} \mp 1 \opright)}^2}{d\xi}.
		\end{align}
	\item\label{thm:shearlets:unishprop:int}
		Interior shearlets:
		The substitution $\eta^\T = \xi^\T \pscalcone{\aj}{\dir}^{-j} \shearcone{\dir}^{-l}$ for $\abs{l} < \lmax$, $\dir \in \{ \hor, \ver \}$ yields
		\begin{align}
			\sum_{k \in \Z^2} \abs{\sp{f}{\unish[\hor]{\aj}{j}{l}{k}}}^2
			&= \integ[\R^2]{ \abs{\ft f(\xi)}^2 \abs{\Corofunc( 2^{-2j}\xi )}^2 \abs[auto]{\conefunc\opleft( \lmax  \tfrac{\xi_2}{\xi_1} - l \opright)}^2}{d\xi}, \\
			\sum_{k \in \Z^2} \abs{\sp{f}{\unish[\ver]{\aj}{j}{l}{k}}}^2
			&= \integ[\R^2]{ \abs{\ft f(\xi)}^2 \abs{\Corofunc( 2^{-2j}\xi )}^2 \abs[auto]{\conefunc\opleft( \lmax  \tfrac{\xi_1}{\xi_2} - l \opright)}^2}{d\xi}.
		\end{align}
	\item\label{thm:shearlets:unishprop:coarse}
		Coarse scaling functions:
		Since $\supp\Scalfunc \subset Q^2$, we have
		\begin{equation}
			\sum_{k \in \Z^2} \abs{\sp{f}{\unishplain_{-1,k}}}^2
			= \sum_{k \in \Z^2} \abs[\bigg]{\integ[Q^2]{\ft f(\xi) \ft\Scalfunc(\xi) e^{-2\pi i \xi^\T k}}{d\xi}}^2
			= \integ[\R^2]{ \abs{\ft f(\xi)}^2 \abs{\ft\Scalfunc(\xi)}^2}{d\xi}.
		\end{equation}
	\end{proofsteps}
	Finally, by using \ref{thm:shearlets:unishprop:bnd}--\ref{thm:shearlets:unishprop:coarse}, we can conclude that
	\begin{align}
		& \mathrel{\phantom{=}} \sum_{\unishplain \in \Unish} \abs{\sp{f}{\unishplain}}^2 \\
		& = \sum_{k \in \Z^2} \bigg( \abs{\sp{f}{\unishplain_{-1,k}}}^2 + \sum_{\dir \in \{ \hor,\ver \}}\sum_{j \in \Nzero} \sum_{\abs{l} < \lmax} \abs{\sp{f}{\unish[\dir]{\aj}{j}{l}{k}}}^2 + \sum_{j \in \Nzero} \sum_{l = \pm\lmax} \abs{\sp{f}{\unish{\aj}{j}{l}{k}}}^2 \bigg) \\
		& = \begin{aligned}[t]
			&\integ[\R^2]{ \abs{\ft f(\xi)}^2 \abs{\ft\Scalfunc(\xi)}^2}{d\xi} \\
			& + \integ[\R^2]{ \abs{\ft f(\xi)}^2 \sum_{j \in \Nzero} \abs{\Corofunc( 2^{-2j}\xi )}^2 \bigg[ \sum_{\abs{l} < \lmax} \abs[auto]{\conefunc\opleft( \lmax \tfrac{\xi_2}{\xi_1} - l \opright)}^2 + \sum_{\abs{l} < \lmax} \abs[auto]{\conefunc\opleft( \lmax \tfrac{\xi_1}{\xi_2} - l \opright)}^2 \bigg]}{d\xi} \\
			& + \integ[\Cone{\hor}]{ \abs{\ft f(\xi)}^2 \sum_{j \in \Nzero} \abs{\Corofunc( 2^{-2j}\xi )}^2 \abs{\conefunc\opleft( \lmax \left( \xi_2 / \xi_1 - 1 \right) \opright)}^2}{d\xi} \\
			& + \integ[\Cone{\ver}]{ \abs{\ft f(\xi)}^2 \sum_{j \in \Nzero} \abs{\Corofunc( 2^{-2j}\xi )}^2 \abs{\conefunc\opleft( \lmax \left( \xi_1 / \xi_2 - 1 \right) \opright)}^2}{d\xi}\\
			& + \integ[\Cone{\hor}]{ \abs{\ft f(\xi)}^2 \sum_{j \in \Nzero} \abs{\Corofunc( 2^{-2j}\xi )}^2 \abs{\conefunc\opleft( \lmax \left( \xi_2 / \xi_1 + 1 \right) \opright)}^2}{d\xi} \\
			& + \integ[\Cone{\ver}]{ \abs{\ft f(\xi)}^2 \sum_{j \in \Nzero} \abs{\Corofunc( 2^{-2j}\xi )}^2 \abs{\conefunc\opleft( \lmax \left( \xi_1 / \xi_2 + 1 \right) \opright)}^2}{d\xi}
		\end{aligned} \\
		& = \begin{multlined}[t]
			\integ[\R^2]{ \abs{\ft f(\xi)}^2 \abs{\ft\Scalfunc(\xi)}^2}{d\xi} + \integ[\R^2]{ \abs{\ft f(\xi)}^2 \sum_{j \in \Nzero} \abs{\Corofunc( 2^{-2j}\xi )}^2 \times \\
			\times \bigg[ \indset{\Cone{\hor}}(\xi) \underbrace{\sum_{\abs{l} \leq \lmax} \abs[auto]{\conefunc\opleft( \lmax \tfrac{\xi_2}{\xi_1} - l \opright)}^2}_{\stackrel{\eqref{eq:shearlets:corosupp}}{=} 1}
			+ \indset{\Cone{\ver}}(\xi) \underbrace{\sum_{\abs{l} \leq \lmax} \abs[auto]{\conefunc\opleft( \lmax \tfrac{\xi_1}{\xi_2} - l \opright)}^2}_{\stackrel{\eqref{eq:shearlets:corosupp}}{=} 1} \bigg]}{d\xi}
		\end{multlined} \\
		& = \integ[\R^2]{ \abs{\ft f(\xi)}^2 \underbrace{\bigg[ \abs{\ft\Scalfunc(\xi)}^2 + \sum_{j \in \Nzero} \abs{\Corofunc( 2^{-2j}\xi )}^2 \bigg]}_{\stackrel{\eqref{eq:shearlets:corotiling}}{=} 1}}{d\xi} = \Lebnorm{\ft f}^2 = \Lebnorm{f}^2.
	\end{align}
	This finishes the proof.\qedhere
\end{proof}

\subsection{Proofs of Decay Lemmas from Subsection \ref{subsec:imginp:clusteredsparse}}
\label{subsec:proofs:coeffdecay}

The approaches of the following two proofs are significantly different from each other: The first result, which is Lemma~\ref{lem:imginp:coeffdecay}, covers the cases of shearlet elements which ``touch'' or ``overlap'' the singularity $\model_j$. This requires application of the precise definitions (which means working with the shearlet Fourier representations) and utilizing specific properties such as vanishing moments conditions. In contrast to this, the statement of Lemma~\ref{lem:imginp:spatialdecay} is a rather general geometric insight: Due to the rapid decay of the models and shearlet functions, the coefficients $\abs{\sp{\model_j}{\unish[\ver]{\aj}{j}{l}{k}}}$ will become small whenever
the ``distance'' of $\model_j$ and $\unish[\ver]{\aj}{j}{l}{k}$ is sufficiently large. Its proof is therefore essentially based on spatial estimates, and avoids any Fourier techniques.

\begin{proof}[Proof of Lemma \ref{lem:imginp:coeffdecay}]
\begin{romanlist}
\item
	Plancherel's Theorem and the definition of universal shearlets imply
	\begin{align}
		\sp[\Big]{\model_j}{\unish[\ver]{\aj}{j}{l}{k}}
		&= \sp[\Big]{\ft\model_j}{\unishft[\ver]{\aj}{j}{l}{k}}
		= \integ[\R^2]{\ft\weight(\xi_1) \Corofunc(2^{-2j}\xi) \conj{\unishft[\ver]{\aj}{j}{l}{k}(\xi)}}{d\xi} \\
		&= \integ[\R]{e^{2\pi i \translind{t_2}{\ver} \xi_2} \integ[\R]{\ft\weight(\xi_1) \underbrace{\Corofunc(2^{-2j}\xi) \unishft[\ver]{\aj}{j}{l}{0}(\xi)}_{=: \sigma_{j,l}(\xi)} e^{2\pi i \translind{t_1}{\ver} \xi_1} }{d\xi_1}}{d\xi_2}.
	\end{align}
	Next, we perform partial integrations, where $\xi\mapsto \ft\weight(\xi_1) \sigma_{j,l}(\xi)$ is differentiated and the modulation terms are integrated, respectively.
	Repeating this $\translind{N_i}{\ver}$-times for $i = 1, 2$ (recall that $\translind{N_i}{\ver} = 0$ when $\translind{t_i}{\ver} = 0$), we obtain
	\begin{equation} \label{eq:proofs:coeffdecay:partder}
		\abs[\Big]{\sp[\Big]{\model_j}{\unish[\ver]{\aj}{j}{l}{k}}}
		\leq C_{N_1,N_2} \abs[auto]{\translind{t_1}{\ver}}^{-\translind{N_1}{\ver}} \abs[auto]{\translind{t_2}{\ver}}^{-\translind{N_2}{\ver}} \integ[\R]{\underbrace{\integ[\R]{\abs[\Big]{\der{(\translind{N_1}{\ver},\translind{N_2}{\ver})}[\ft\weight(\xi_1)\sigma_{j,l}(\xi)] }}{d\xi_1}}_{=: h_{N_1,N_2}(\xi_2)}}{d\xi_2},
	\end{equation}
	where the boundary terms vanish due to the compact support of $\xi\mapsto \ft\weight(\xi_1) \sigma_{j,l}(\xi)$.
	From \eqref{eq:shearlets:unishdef:trapezoidsupport} (the roles of $\xi_1$ and $\xi_2$ of course need to be interchanged), we conclude for any $\xi \in \supp\sigma_{j,l}$ that
	\begin{equation} \label{eq:proofs:coeffdecay:supp}
		(l-1)2^{(\aj-2)j} \leq \frac{\xi_1}{\xi_2} \leq (l+1)2^{(\aj-2)j} \quad \text{ and } \quad
		2^{2j-4} \leq \abs{\xi_2} \leq 2^{2j-1}.
	\end{equation}
	The assumption $\abs{l} > 1$ now implies that there exist constants $C_1, C_2 > 0$ such that
	\begin{equation}
		\xi_1 \in I_{j,l} := \intvcl{-C_2 l 2^{\aj j}}{-C_1 l2^{\aj j}} \union \intvcl{C_1 l 2^{\aj j}}{C_2 l 2^{\aj j}}.
	\end{equation}
	The Leibniz rule with respect to $\xi_1$ yields
	\begin{equation}
		\der{(\translind{N_1}{\ver},\translind{N_2}{\ver})}[\ft\weight(\xi_1)\sigma_{j,l}(\xi)] = \sum_{n_1 = 0}^{\translind{N_1}{\ver}} \binom{\translind{N_1}{\ver}}{n_1} \ft\weight^{(n_1)}(\xi_1) \der{(\translind{N_1}{\ver}-n_1,\translind{N_2}{\ver})} \sigma_{j,l}(\xi), \quad \xi \in \R^2,
	\end{equation}
	and, combined with Hölder's inequality, we obtain (recall the definition of $h_{N_1,N_2}$ from \eqref{eq:proofs:coeffdecay:partder})
	\begin{align}
		h_{N_1,N_2}(\xi_2)
		&\leq \sum_{n_1 = 0}^{\translind{N_1}{\ver}} \binom{\translind{N_1}{\ver}}{n_1} \integ[\R]{\abs[\Big]{\ft\weight^{(n_1)}(\xi_1) \der{(\translind{N_1}{\ver}-n_1,\translind{N_2}{\ver})} \sigma_{j,l}(\xi)}}{d\xi_1} \\
		&\leq \sum_{n_1 = 0}^{\translind{N_1}{\ver}} \binom{\translind{N_1}{\ver}}{n_1} \Lebnorm{\ft\weight^{(n_1)}}[1][I_{j,l}] \Lebnorm{\der{(\translind{N_1}{\ver}-n_1,\translind{N_2}{\ver})} \sigma_{j,l}}[\infty][\R^2].
	\end{align}
	To estimate the first factor, we use the rapid decay of $\ft\weight$ as well as the specific form of $I_{j,l}$:
	\begin{align}
		\Lebnorm{\ft\weight^{(n_1)}}[1][I_{j,l}]
		&\leq \measure{I_{j,l}} \sup_{\xi_1 \in I_{j,l}} \abs{\ft\weight^{(n_1)} (\xi_1)}
		\leq C_{M,n_1} \cdot \abs{l 2^{\aj j}} \cdot \Schwartzpoly{l 2^{\aj j}}^{-(M+1)} \\
		&\leq C_{M,n_1} \Schwartzpoly{l 2^{\aj j}}^{-M}
		\stackrel{\abs{l} > 1}{\leq} C_{M,n_1} \Schwartzpoly{2^{\aj j}}^{-M}. \label{eq:proofs:coeffdecay:weight}
	\end{align}
	For the other term, we proceed similar to \eqref{eq:proofs:unishprop:partderhor} und \eqref{eq:proofs:unishprop:partderver} and additionally use \eqref{eq:proofs:coeffdecay:supp}:
	\begin{align}
		2^{(2+\aj)j/2}\abs[auto]{\partder{\sigma_{j,l}}{\xi_1} (\xi)}
		&= \abs[auto]{\partder{}{\xi_1} \left[ \Corofunc^2( 2^{-2j}\xi )\conefunc\opleft( 2^{(2-\aj)j} \frac{\xi_1}{\xi_2} - l \opright) \right]} \\*
		&\leq C_1 2^{-2j} + C_2 \frac{2^{(2-\aj)j}}{\abs{\xi_2}}
		\stackrel{\aj \leq 2}{\leq} C 2^{-\aj j}, \\
		2^{(2+\aj)j/2}\abs[auto]{\partder{\sigma_{j,l}}{\xi_2} (\xi)}
		&= \abs[auto]{\partder{}{\xi_2} \left[ \Corofunc^2( 2^{-2j}\xi )\conefunc\opleft( 2^{(2-\aj)j} \frac{\xi_1}{\xi_2} - l \opright) \right]} \\*
		&\leq C_1 2^{-2j} + C_2 2^{(2-\aj)j} \frac{\abs{\xi_1}}{\abs{\xi_2}^2}
		\leq C 2^{-\aj j}.
	\end{align}
	By induction, we obtain
	\begin{equation}
		\Lebnorm{\der{(\translind{N_1}{\ver}-n_1,\translind{N_2}{\ver})} \sigma_{j,l}}[\infty][\R^2]
		\leq C_{N_1,N_2,n_1} 2^{-(2+\aj)j/2} 2^{-(\translind{N_1}{\ver} - n_1)\aj j} 2^{-\translind{N_2}{\ver}\aj j}.
	\end{equation}
	This implies
	\begin{align}
		h_{N_1,N_2}(\xi_2) & \leq \sum_{n_1 = 0}^{\translind{N_1}{\ver}} \binom{\translind{N_1}{\ver}}{n_1} C_{N_1,N_2,n_1,M} \Schwartzpoly{2^{\aj j}}^{-M} 2^{-(2+\aj)j/2} 2^{-(\translind{N_1}{\ver} - n_1)\aj j} 2^{-\translind{N_2}{\ver}\aj j} \\
		& \leq C_{N_1,N_2,M} \Schwartzpoly{2^{\aj j}}^{-M} 2^{-(2+\aj)j/2} 2^{-\translind{N_2}{\ver}\aj j} \underbrace{\sum_{n_1 = 0}^{\translind{N_1}{\ver}} \binom{\translind{N_1}{\ver}}{n_1} \left( 2^{-\aj j} \right)^{\translind{N_1}{\ver} - n_1}}_{= (1 + 2^{-\aj j})^{\translind{N_1}{\ver}} \leq 2^{\translind{N_1}{\ver}}} \\
		& \leq C_{N_1,N_2,M} \Schwartzpoly{2^{\aj j}}^{-M} 2^{-(2+\aj)j/2} 2^{-\translind{N_2}{\ver}\aj j},
	\end{align}
	where we made use of the assumption $\aj \geq 0$ to estimate $(1 + 2^{-\aj j})^{\translind{N_1}{\ver}}$.
	Since $\measure{\supp h_{N_1,N_2}} \leq C 2^{2j}$, a final estimate gives
	\begin{align}
		\abs[\Big]{\sp[\Big]{\model_j}{\unish[\ver]{\aj}{j}{l}{k}}}
		& \leq C_{N_1,N_2} \abs[auto]{\translind{t_1}{\ver}}^{-\translind{N_1}{\ver}} \abs[auto]{\translind{t_2}{\ver}}^{-\translind{N_2}{\ver}} \integ[\R]{h_{N_1,N_2}(\xi_2)}{d\xi_2} \\
		& \leq C_{N_1,N_2,M} \abs[auto]{\translind{t_1}{\ver}}^{-\translind{N_1}{\ver}} \abs[auto]{\translind{t_2}{\ver}}^{-\translind{N_2}{\ver}} \Schwartzpoly{2^{\aj j}}^{-M} 2^{(2-\aj)j/2} 2^{-\translind{N_2}{\ver}\aj j}.
	\end{align}
\item
	This proof works analogously to the one of \ref{lem:imginp:coeffdecay:ver}.
	Basically, only the intervals $I_{j,l}$ need to be modified slightly.
	Due to $\abs{\xi_1} \in \asympeq{2^{2j}}$ for any $\xi \in \supp\unishft[\hor]{\aj}{j}{l}{k}$, they have the form
	\begin{equation}
		\intvcl{-C_2 2^{2 j}}{-C_1 2^{2 j}} \union \intvcl{C_1 2^{2 j}}{C_2 2^{2 j}}, \quad C_1, C_2 > 0.
	\end{equation}
	This leads to the different factor $\Schwartzpoly{2^{2 j}}^{-M}$ in \ref{lem:imginp:coeffdecay:hor}, compared with \ref{lem:imginp:coeffdecay:ver}.
\item
	Recalling the definition of the boundary shearlets, one can easily verify this estimate by a combination of part \ref{lem:imginp:coeffdecay:ver} and \ref{lem:imginp:coeffdecay:hor}.\qedhere
\end{romanlist}
\end{proof}

It is worth noting that the crucial point of the proof above is the specific form of the intervals $I_{j,l}$, which is due to the ``geometry'' of $\supp\unishft[\dir]{\aj}{j}{l}{k}$.
Interestingly, the fact that $0$ is not contained in $I_{j,l}$, precisely corresponds to the fact that all moments of $\unish[\ver]{\aj}{j}{l}{k}$ vanish with respect to $x_2$ (at least when $\abs{l} > 1$).

\begin{proof}[Proof of Lemma \ref{lem:imginp:spatialdecay}]
At first, we provide appropriate decay estimates of our edge model,
\begin{align} \label{eq:proofs:spatialdecay:modeldecay}
	\abs{\model_j(x)} ={} & \abs{[\model \conv \Corofilter_j](x)} = \abs[\Big]{\integ[\R]{\weight(y_1) F_j(x - (y_1, 0))}{dy_1}} \\
	\leq{} & \integ[\R]{\abs{\weight(y_1)} 2^{4j} \underbrace{\abs{\ift{\Corofunc}(2^{2j}(x - (y_1, 0)))}}_{\mathclap{\leq C_N \Schwartzpoly{2^{2j} x_2}^{-N} \Schwartzpoly{2^{2j} (y_1 - x_1)}^{-N}}}}{dy_1} = C_N 2^{4j} \Schwartzpoly{2^{2j} x_2}^{-N} \underbrace{[\abs{\weight} \conv \Schwartzpoly{2^{2j} \dotarg}^{-N}](x_1)}_{= \tilde\weight_{N,j}(x_1)} \\
	={} & C_N  2^{4j} \Schwartzpoly{2^{2j} x_2}^{-N} \tilde\weight_{N,j}(x_1), \quad x = (x_1, x_2) \in \R^2,
\end{align}
as well as of the shearlet elements,
\begin{align}
	\abs{\unish[\ver]{\aj}{j}{l}{k}(x)} &\leq C_N 2^{(\aj + 2)j/2} \Schwartzpoly{\shearcone{\ver}^{l}\pscalcone{\aj}{\ver}^j x - k}^{-N} \\
	&\leq C_N 2^{(\aj + 2)j/2} \Schwartzpoly{2^{\aj j}x_1 - k_1}^{-N} \Schwartzpoly{l 2^{\aj j} x_1 + 2^{2 j}x_2 - k_2}^{-N}, \quad x = (x_1, x_2) \in \R^2.
\end{align}
With this, we can estimate the analysis coefficients:
\begin{align}\label{eq:proofs:spatialdecay:integest}
	\abs[\Big]{\sp[\Big]{\model_j}{\unish[\ver]{\aj}{j}{l}{k}}} \leq{} & C_N \integ[\R^2]{ 2^{4j} \Schwartzpoly{2^{2j} x_2}^{-N} \tilde\weight_{N,j}(x_1) 2^{(\aj + 2)j/2} \Schwartzpoly{2^{\aj j}x_1 - k_1}^{-N} \Schwartzpoly{l 2^{\aj j} x_1 + 2^{2 j}x_2 - k_2}^{-N}}{dx} \\
	={} & C_N 2^{(5-\aj)j/2} \integ[\R^2]{\Schwartzpoly{x_2}^{-N} \tilde\weight_{N,j}(2^{-\aj j}(x_1+k_1)) \Schwartzpoly{x_1}^{-N} \Schwartzpoly{l x_1 + x_2 + lk_1 - k_2}^{-N}}{dx}. 
\end{align}
Now, let us consider the product $\Schwartzpoly{x_2}^{-N} \Schwartzpoly{x_2 + (l x_1 + lk_1 - k_2)}^{-N}$.
One of the two factors (which one it is might depend on $x_2$) clearly has to be smaller than $\Schwartzpoly{\frac{1}{2}(l x_1 + lk_1 - k_2)}^{-N}$.
This yields ($K := l x_1 + lk_1 - k_2$)
\begin{align} \label{eq:proofs:spatialdecay:splitargument}
	\integ[\R]{\Schwartzpoly{x_2}^{-N} \Schwartzpoly{x_2 + K}^{-N}}{dx_2} ={}& \integ[\R]{\underbrace{\max\{\Schwartzpoly{x_2}^{-N}, \Schwartzpoly{x_2+K}^{-N}\}}_{\leq \Schwartzpoly{x_2}^{-N} + \Schwartzpoly{x_2+K}^{-N}} \underbrace{\min\{\Schwartzpoly{x_2}^{-N}, \Schwartzpoly{x_2+K}^{-N}\}}_{\leq \Schwartzpoly{\frac{K}{2}}^{-N}}}{dx_2} \\
	\leq{}& C_N \Schwartzpoly{K}^{-N} = C_N \Schwartzpoly{l x_1 + lk_1 - k_2)}^{-N}. 
\end{align}
Putting this into \eqref{eq:proofs:spatialdecay:integest}, we finally obtain
\begin{equation}
	\abs[\Big]{\sp[\Big]{\model_j}{\unish[\ver]{\aj}{j}{l}{k}}} \leq C_N 2^{(5-\aj)j/2} \integ[\R]{\tilde\weight_{N,j}(2^{-\aj j}(x_1+k_1)) \Schwartzpoly{x_1}^{-N} \Schwartzpoly{l x_1 + lk_1 - k_2)}^{-N}}{dx_1}.
\end{equation}
\end{proof}

\nocite{*}
\bibliographystyle{amsplain}
\bibliography{inpainting-paper.bib}

\end{document}